\documentclass[12pt,french,leqno]{amsart}

\def\Chapter{\vfe\section}
\def\Section{\subsection}
\def\Subsection{\subsubsection*}
\newcount\Partnumber\Partnumber=1
\def\Part#1{\null\vfill\eject
\hfil{\bf Partie \uppercase\expandafter{\romannumeral\Partnumber}.\ #1%
\global\advance\Partnumber by1}}

\usepackage{amssymb}
\usepackage{amsmath}
\usepackage{amscd}
\usepackage[french]{babel}
\usepackage[T1]{fontenc}
\usepackage{enumitem}
\usepackage{amsmath}
\usepackage{amssymb,url,xspace}
\usepackage{mathrsfs,euscript,color}
\usepackage{xcolor}
\usepackage{bbm}
\usepackage{bm}

\newtheorem{proposition}{Proposition}[subsection]
\newtheorem{definition}[proposition]{Definition}
\newtheorem{lemma}[proposition]{Lemme}
\newtheorem{theorem}[proposition]{Th\' eor\`eme}
\newtheorem{corollary}[proposition]{Corollaire}
\newtheorem{remark}[proposition]{Remarque}

\catcode`\Ž=13
\catcode`\=13
\catcode`\ˆ=13
\catcode`\=13
\catcode`\=13
\catcode`\‰=13
\catcode`\™=13
\catcode`\"=13
\catcode`\•=13
\catcode`\ž=13
\catcode`\=13
\catcode`\'=13
\catcode`\Ï=13

\def Ž{\'e}
\def {\`e}
\def ˆ{\`a}
\def {\`u}
\def {\^e}
\def ‰{\^a}
\def ™{\^o}
\def "{\^{\i}}
\def •{\"{\i}}
\def ž{\^u}
\def {\c c}
\def '{\"e}
\def Ï{\oe}

\catcode`\:=13
\def :{~\string:}
\catcode`\;=13
\def ;{~\string;}
\catcode`\!=13
\def !{~\string!}

\def\mun{^{-1}}
\def \cad{c'est-\`a-dire\ }

\def\AM{\mathbb A}
\def\CM{\mathbb C}
\def\FM{\mathbb F}

\def\HM{\mathbb H}
\def\NM{\mathbb N}
\def\RM{\mathbb R}
\def\QM{\mathbb Q}

\def\ZM{\mathbb Z}

\def\ade{\AM}
\def\adef{\AM_F}

\def\Fsep{F^{sep}}
\def\Fbar{\overline{F}}
\def\Ftp{F^\times_E}

\def\coker{\mathrm{Coker}}
\def\card{\mathrm{card}\,}
\def\Gal{\mathrm{Gal}\,}
\def\Ind{\mathrm{Ind}}
\def\trace{\mathrm{trace}\,}
\def\vol{\mathrm{vol}\,}

\def\da{\,d^\times \a}
\def\ddx{\,d\dot x}

\def\dt{\,d^\times t}
\def\dtx{\,d\tx}

\def\ctyc{\mathcal{C}_c^\infty}

\def\F{F}
\def\G{G}
\def\H{H}
\def\J{J}
\def\M{M}
\def\P{P}
\def\T{T}
\def\U{U}
\def\Z{Z}

\def\EC{\mathcal{E}}
\def\UC{\mathcal{U}}

\def\vf{\varphi}
\def\VF{\Phi}
\def\ep{\epsilon}
\def\ve{\varepsilon}
\def\th{\theta}

\def\tG{{\widetilde G}}
\def\tI{{\widetilde I}}

\def\tK{{\widetilde K}}
\def\tM{{\widetilde M}}
\def\tP{{\widetilde P}}
\def\tT{{\widetilde T}}
\def\tZ{{\widetilde Z}}

\def\tf{\widetilde{f}}
\def\tk{\widetilde{k}}
\def\tm{\widetilde{m}}
\def\tpi{\widetilde{\pi}}
\def\tth{{\widetilde\th}}

\def\txi{{\widetilde\lambda}}

\def\x{x}
\def\tx{{\widetilde\x}}
\def\ty{{\widetilde{y}}}
\def\tlambda{{\widetilde{\lambda}}}
\def\tgamma{{\widetilde{\gamma}}}

\def\vef{{\ve_{E/F}}}
\def\tTef{\tT_{E/F}}
\def\Tef{T_{E/F}}
\def\Nef{N_{E/F}}

\def\orb{\mathcal{O}}
\def\OF{\mathfrak{O}_F}
\def\OFv{\mathfrak{O}_{v}}
\def\OE{\mathfrak{O}_E}

\def\dfr{\mathfrak d}
\def\pfr{\mathfrak p}
\def\tfr{\mathfrak t}

\def\a{a}
\def\b{b}
\def\aa{a_1}
\def\bb{b_1}
\def\cc{c_1}
\def\dd{d_1}
\def\aaa{a_2}
\def\bbb{b_2}
\def\ccc{c_2}
\def\ddd{d_2}

\def\HP{\mathbf{\H}_\P}
\def\HtP{\H_\tP}
\def\orbu{g}

\def\Tr{X}
\def\bsl{\backslash}

\def\bs{\boldsymbol}
\def\mmu{\mu\mun}

\def\cft{\underline c}

\def\lg{\langle}
\def\lr{\rangle}

\def\delo{\tau}
\def\diag{{\bs d}}
\def\nil{{\bs n}}
\def\tGFp{\tG(F)_E}

\def\ll{\ell}
\def\tw{{\eta}}
\def\tv{{\bs \eta}}

\def\ws{{\bs s}}
\def\wt{{\bs t}}

\def\JL{{\textrm {JL}\,}}

\def\spec{\Lambda}
\def\speccont{\Lambda}

\def\ts{${\mathcal T}$}
\def\com#1{\quad\hbox{#1}\quad}
\def\comm#1{\qquad\hbox{#1}\qquad}
\def\pni{\par\noindent}
\def\ptf{~.}
\def\vfe{\vfill\eject}
\newcount\Cor\Cor=0
\def\newcor{\global\advance\Cor by 1
\par\bigskip\noindent (\romannumeral\Cor) - }

\date{\today}
\author{Jean-Pierre Labesse}
\address{Aix Marseille Univ, CNRS, I2M, Marseille, France}

\title[$\bs {SL(2)}$ en toutes caract\' eristiques]
{Stabilisation et germes pour $SL(2)$  en toutes caract\' eristiques}

\begin{document}


\begin{abstract} Nous explicitons la stabilisation des intŽgrales orbitales locales 
et de la formule des traces sur un corps global
pour $SL(2)$ en donnant des formulations et des preuves valables en toute caractŽristique.
Quelques faits nouveaux apparaissent en caractŽristique 2.
Le dŽveloppement asymptotique pour les intŽgrales orbitales locales
au voisinage de l'identitŽ, obtenu via la stabilisation, est Žquivalent 
au dŽveloppement en germes de Shalika en caractŽristique $p\ne2$ mais
il est nouveau en caractŽristique $2$. De mme le dŽveloppement
fin de la contribution unipotente ˆ la formule des traces (globale)
ne peut pas tre obtenu par les techniques d'Arthur
en caractŽristique $2$ et nŽcessite une prŽ-stabilisation.
\end{abstract}

\maketitle


\tableofcontents


\Chapter*{Introduction}
 
 L'essentiel des rŽsultats de cette note a ŽtŽ exposŽ lors d'une confŽrence ˆ Budapest en aožt 1971,
 ˆ l'exception de ce qui concerne les germes, encore que ceci est implicite 
 dans la preuve de l'existence du transfert endoscopique.
 Il s'agissait lˆ de la premire forme d'un travail en collaboration avec Langlands 
initiŽ ˆ Bonn entre avril et juin 1971 et qui s'est poursuivi dans les annŽes ultŽrieures.

Dans les proceedings de la confŽrence de Budapest on trouvera dans  \cite{Lab} une esquisse 
trs sommaire des arguments. 
L'article \cite{LL} donne une preuve dŽtaillŽe mais qui, contrairement ˆ \cite{Lab}, 
se limite ˆ la caractŽristique nulle. Le cas des corps de caractŽristique positive
aurait imposŽ quelques dŽveloppements supplŽmentaires, d'ailleurs prŽsents dans les notes prŽparatoires.
Les difficultŽs de la caractŽristique positive sont de deux ordres: du point de vue gŽomŽtrique
les contributions unipotentes pour $SL(2)$ doivent tre traitŽes diffŽremment 
si on ne veut pas exclure la caractŽristique 2, et  
du point de vue spectral il n'Žtait pas connu que les caractres des reprŽsentations admissibles
irrŽductibles de $SL(2,F)$, o $F$ est un corps local,
Žtaient reprŽsentables par des fonctions localement intŽgrables:
 \`a cette Žpoque seul le cas de $GL(2,F)$ pour $F$ local de caractŽristique arbitraire
Žtait connu gr‰ce ˆ \cite{JL}. Ceci est maintenant 
Žtabli pour $GL(n)$ et $SL(n)$ ainsi que pour leurs formes intŽrieures 
(mais encore pour des groupes intermŽdiaires) dans \cite{Lem}.

On observera que l'article \cite{LL}, essentiellement rŽdigŽ par Langlands,
n'a ŽtŽ soumis pour publication qu'en 1977 aprs la rŽdaction de \cite{BC} 
de sorte que divers lemmes techniques indispensables dans \cite{BC}, mais
qui rŽsultaient de notre collaboration et avaient leur place dans \cite{LL}, 
ont ŽtŽ incorporŽs dans \cite{BC}. 

Nous complŽtons ici ces articles anciens en donnant des preuves compltes 
en toute caractŽristique. L'essentiel des argument, sauf ce qui est spŽcifique ˆ la caractŽristique 2,
se trouve dŽjˆ dans \cite{LL}. Pour simplifier nous nous limitons ici au groupe $SL(2)$
et ˆ ses formes intŽrieures. Nous ne traitons pas les groupes intermŽdiaires entre $SL(2)$ et $GL(2)$ 
dŽfinis par une condition de nature arithmŽtique sur le dŽterminant; cela Žtait utile pour tester des
conjectures sur les variŽtŽs de Shimura, mais ne prŽsente
ni difficultŽ ni de phŽnomne nouveaux autres que des complications de notation, 
pour ce qui nous concerne ici
ˆ savoir une introduction ˆ l'endoscopie dans le cas particulier le plus simple.

On appelle stabilisation, ou endoscopie, l'Žtude du comportement par conjugaison stable
(qui dans notre exemple se rŽduit ˆ la conjugaison sous $GL(2,F)$) 
des intŽgrales orbitales et des caractres de reprŽsentations irrŽductibles pour $SL(2,F)$.
Nous commencerons par la stabilisation des intŽgrales orbitales des ŽlŽments semi-simples pour $SL(2)$
sur un corps local. Les deux rŽsultats importants sont l'existence du transfert endoscopique
et le Lemme Fondamental.

Le transfert endoscopique fournit un dŽveloppement asymptotique au voisinage de 1 valable sur un corps
local $\mathfrak{p}$-adique en toute caractŽristique.
En caractŽristique $p\ne2$ ce dŽveloppement est Žquivalent, modulo une transformŽe de Fourier,
au dŽveloppement en germes de Shalika mais il est nouveau en caractŽristique $2$.
L'observation principale est que pour $p=2$ les orbites unipotentes 
ne sont pas les bons objets pour paramŽtrer un dŽveloppement asymptotique car,
sur un corps local en caractŽristique 2, les orbites unipotentes pour $SL(2)$ 
forment un ensemble infini non dŽnombrable.
Nous Žtudierons ensuite la stabilisation locale du point de vue spectral \cad
l'Žtude des $L$-paquets et du transfert spectral.  
La reprŽsentation de Weil fournit une construction explicite
des $L$-paquets.

Dans un second temps nous rappellerons la formule des traces pour $SL(2)$ sur un corps global.
On note  $\rho$  la reprŽsentation rŽgulire droite
dans l'espace de Hilbert $$ L^2\big(\G(F)\bsl\G(\adef)\big)$$ 
et soit
$$\rho(f)=\int_{\G(\adef)}\rho(x)f(x)\,dx$$ 
l'opŽrateur dŽfini par l'intŽgrale (faible) de $\rho$ contre une fonction
$f\in\ctyc\big(\G(\adef)\big)$. On s'intŽresse au spectre discret 
$L^2_{disc}\big(\G(F)\bsl\G(\adef)\big)\subset L^2\big(\G(F)\bsl\G(\adef)\big)$.
On souhaite avoir une formule gŽomŽtrique pour calculer 
la trace de l'opŽrateur  $\rho(f)$  restreint au spectre discret:
$$\trace \big(\rho(f)\big\vert L^2_{disc}\big(\G(F)\bsl\G(\adef)\big)\ptf$$
On dŽfinit par troncature spectrale une variante de cette trace
$$J_{spec}(f)=\trace \big(\rho(f)\big\vert L^2_{disc}(\G(F)\bsl\G(\adef)\big)+c(f)$$
o $c(f)$ est un terme complŽmentaire explicite.
L'expression gŽomŽtrique $J_{geom}(f)$ est elle obtenue par une troncature de l'intŽgrale 
sur la diagonale du noyau reprŽsentant l'opŽrateur $\rho(f)$.
La formule des traces est une identitŽ $$J_{geom}(f)=J_{spec}(f)$$
qui, pour $SL(2)$, remonte pour l'essentiel ˆ Selberg \cite{Se}.

La stabilisation de la formule des traces est son Žcriture comme somme de distributions 
stablement invariantes (ce qui dans notre cas est l'invariance par conjugaison sous $GL(2,\adef)$)
et de termes correctifs, appelŽs contributions endoscopiques: elles mesurent la diffŽrence
entre la formule des traces ``ˆ la Selberg'' et sa variante stable.
Pour $\G=SL(2)$ ou pour $\G'$ une forme intŽrieure, 
ces termes correctifs proviennent de distributions sur des tores.
Cela se fait sans difficultŽs, une fois connus le transfert et le Lemme Fondamental,
pour les formes intŽrieures $\G'$ non dŽployŽes car 
pour de tels groupes le quotient $\G'(F)\bsl\G'(\adef)$ est compact
et la formule des traces est une identitŽ entre distributions invariantes sous $\G'(\adef)$.
Pour $\G=SL(2)$ la non compacitŽ du quotient $\G(F)\bsl\G(\adef)$
impose l'utilisation de troncatures qui produisent des distributions non-invariantes sous $SL(2,\adef)$
ce qui rend plus dŽlicate l'Žtude de la conjugaison sous $GL(2,\adef)$.
C'est pourquoi on parle de forme non-invariante de la formule des traces et
il est classique (cf. \cite{LL}, \cite{A10, A11, A12} et \cite{MW2}) de passer d'abord 
ˆ une forme invariante avant de tenter de la stabiliser. 
Mais, comme dans \cite{Lab}, nous ferons la stabilisation en partant de la forme non-invariante;
le passage ˆ une forme invariante sera la dernire Žtape.

Une premire Žtape appelŽe prŽ-stabilisation  permet d'obtenir la forme fine 
du dŽveloppement gŽomŽtrique \cad une expression qui est somme de produits d'intŽgrales locales.
Une telle prŽ-stabilisation semble une Žtape indispensable si on veut travailler en toutes caractŽristiques
car, comme dans le cas des germes ŽvoquŽs ci-dessus,
les techniques classiques (gŽnŽralisŽes par Arthur \cite{A5} aux groupes rŽductifs gŽnŽraux
en caractŽristique zŽro) 
ne sont pas utilisables pour calculer la contribution unipotente en caractŽristique $2$.

En combinant la prŽ-stabilisation avec le transfert on obtient une Žcriture
du c™tŽ gŽomŽtrique $J_{geom}(f)$ et du c™tŽ spectral $J_{spec}(f)$
de la formule des traces non-invariante comme une somme
$$J_{\bullet}(f)=\sum_{\EC}SJ_{\bullet}^\EC(f^\EC)$$
indexŽe par les donnŽes endoscopiques,
o les divers termes vŽrifient des identitŽs $$J_{geom}(f)=J_{spec}(f)\com{et}
SJ^\EC_{geom}(f^\EC)=SJ^\EC_{spec}(f^\EC)\ptf$$
Les distributions $SJ_{\bullet}^\EC(f^\EC)$ sont ce que nous appellerons
les termes gŽomŽtriques (resp. spectraux) de la formule des traces \ts-stable
pour les divers groupes endoscopiques, ˆ savoir $SL(2)$ lui mme
et les tores attachŽs aux extensions quadratiques sŽparables.

La formule des traces pour les tores est automatiquement stable
car ce sont des groupes abŽliens.
Le passage ˆ la forme stablement-invariante ˆ partir de la formule \ts-stable pour $SL(2)$
peut alors se faire comme dernire Žtape et la formule des traces invariante
est une somme d'expressions stablement invariantes sur chaque groupe endoscopique
$$I_{\bullet}(f)=\sum_{\EC}SI_{\bullet}^\EC(f^\EC)$$
et on a les identitŽs de formules des traces invariantes et stablement invariantes:
$$I_{geom}(f)=I_{spec}(f)
\com{et}SI^\EC_{geom}(f^\EC)=SI^\EC_{spec}(f^\EC)\ptf$$

\Chapter{PrŽliminaires}
 
\Section{Corps et extensions quadratiques}\label{ext}

Dans toute la suite $F$ dŽsigne un corps local ou global de caractŽristique $p\ge0$.
Si $F$ est un corps local non archimŽdien,
on note $\OF$ l'anneau des entiers, $\varpi_F$ une uniformisante,
$\pfr_F=\varpi_F\OF$ son idŽal maximal et $\FM_q$ le corps rŽsiduel.
Lorsque $F$ est global on note $\adef$ l'anneau des adles de $F$.
Si $F$ est un corps de nombres on pose $$q_F=e=\lim_{n\to\infty}\Big(1+\frac{1}{n}\Big)^n=
2,7\,18\,28\,18\,28\,45\,90\,45\,...\com{de sorte que}\log q_F=1$$ 
et si $F$ est un corps
de fonctions de caractŽristique $p$ on note $$q_F=q=p^f$$ le cardinal du corps $\FM_q$ 
des constantes. On notera $\log_{q_F}$ le logarithme en base $q_F$.
On notera $C_F$ le groupe multiplicatif $F^\times$ si $F$ est local et 
 $C_F$ sera le groupe des classes d'idles $\AM_F^\times/F^\times$ si $F$ est global.
Dans tous les cas le groupe $$Q_F:=C_F/(C_F)^2$$ est un groupe compact
dont le dual de Pontryagin est le groupe (discret) des caractres d'ordre 2 de $C_F$.
Les caractres du groupe $Q_F$ sont en bijection, par la thŽorie du corps de classes, 
avec les extensions quadratiques sŽparables $E/F$ (dŽployŽes si le caractre est trivial). 
Plus prŽcisŽment, le groupe $Q_F$ est d'ordre 1 si $F=\mathbb C$, d'ordre 2 si $F=\mathbb R$ et d'ordre 
4 si $F$ est local non archimŽdien de caractŽristique rŽsiduelle diffŽrente de 2. Il est toujours fini 
si $F$ est local de caractŽristique $p\ne2$ (par exemple il est d'ordre 8 pour $\mathbb Q_2$). C'est 
un groupe compact de cardinal non dŽnombrable si $F$ est un corps local de caractŽristique $2$ 
ou si $F$ est global.

Soit $R(X)$ un polyn™me de la forme
$$R(X)=X^2-\tfr X+\dfr$$ avec $\tfr\in F$ et $\dfr\in F^\times$.
On suppose que $R(X)$ n'est pas un carrŽ dans $F[X]$.
Alors, la $F$-algbre $E=F[X]/R(X)F[X]$ est semi-simple de dimension 2.
Tout ŽlŽment de $E$ induit par multiplication un $F$-endomophisme de $E$. 
Notons $\tau$  l'image de $X$ dans $E$ alors l'ensemble
$\{1,\tau\}$ est une base de $E$ sur $F$. On obtient ainsi un plongement de
$E$ dans $M(2,F)$. L'endomorphisme induit par $\tau$ a pour matrice
$$\bs{\tau}=\begin{pmatrix}0&-\dfr\cr1&\tfr\cr\end{pmatrix}\ptf$$
La matrice
$$\bs{\overline\tau}=\begin{pmatrix}\tfr&\dfr\cr -1&0\cr\end{pmatrix}$$
vŽrifie
$$\bs\tau+\overline{\bs\tau}=\tfr=\trace\bs\tau
\qquad\hbox{et}\qquad \bs\tau\overline{\bs\tau}=\dfr=\det\bs\tau=\Nef(\tau)$$
o $\Nef$ est la norme pour $E/F$.

On suppose dŽsormais $ \bs\tau\ne\overline{\bs\tau}$ 
(ce qui est toujours le cas sauf en caractŽristique $p=2$
o cela impose $\tfr\ne0$). L'extension $E/F$ est alors sŽparable, 
Žventuellement dŽployŽe (i.e. $E=F\oplus F$). On notera $\tTef$ le groupe algŽbrique des
matrices de la forme $t=a+b\bs\tau$ avec $\det(t)\ne0$; c'est un tore dans $GL(2)$.
On notera $\Tef$ le tore de $SL(2)$ dŽfini par le sous-groupe des matrices de dŽterminant 1.
L'automorphisme $\a+\b\bs\tau\mapsto \a+\b\overline{\bs\tau}$ est induit par la conjugaison sous
$$\bs w_E=\begin{pmatrix}0&\mathfrak d\cr 1&0\cr\end{pmatrix}\ptf$$
On notera $\vef$ le caractre d'ordre 2 de $C_F$ 
associŽ ˆ $E$ par la thŽorie du corps de classes:
c'est le gŽnŽrateur du groupe des caractre de $C_F/\Nef C_E$. Il est non trivial si et seulement si
$E$ est un corps. On observe que $\mathfrak d$ est une norme et on a le
\begin{lemma}\label{signa}
$$\vef\big(\det(\bs w_E)\big)=\vef(-1)\ptf$$
\end{lemma}


On notera $\psi$ un caractre non trivial du groupe additif $F$ si $F$ est local et
de $\adef/F$ si $F$ est global. Lorsque $F$ est local  le
scalaire $\lambda(E/F,\psi)$ est introduit par Langlands dans ses notes de Yale 
sur les facteurs $\ve$. Il apparait lorsque l'on calcule la transformŽe de
Fourier, au sens des distributions tempŽrŽes, de la fonction de $E$ dans $\CM$
$$x\mapsto \psi(x\overline x).$$
Formellement on a
 $$\lambda(E/F,\psi)=vp\int_E \psi(x\overline x)dx$$
o $vp$ dŽsigne la ``valeur principale'' de l'intŽgrale divergente.
On aura besoin du
\begin{lemma}\label{psideux}
$$\lambda(E/F,\psi)^2=\vef(-1)\ptf$$
Si $\psi$ est non trivial sur $\pfr^{-1}$ mais trivial sur $\OF$ et si $E/F$ est non ramifiŽ alors
$$\lambda(E/F,\psi)=1\ptf$$
En particulier $\lambda(E/F,\psi)$ est une racine quatrime de l'unitŽ.
Pour une extension quadratique sŽparable de corps globaux $E/F$ le produit sur toutes les places
des facteurs $\lambda(E_v/F_v,\psi)$ o $E_v=E\otimes F_v$ vaut 1.
\end{lemma}
\begin{proof}On renvoie au Lemme 1.1 de \cite{JL} et ˆ \cite{W}.
\end{proof}

\Section{Groupes et sous-groupes}

Comme dans \cite{LL} on pose $$\G=SL(2)\qquad\hbox{et}\qquad\tG=GL(2)\ptf$$ 
On note $\Z$ le centre de $\G$ et $\tZ$ celui de $\tG$.
Le dŽterminant induit, aprs passage au quotient, un isomorphisme
$$G(F)\tZ(F)\bsl \tG(F)\to Q_F$$
si $F$ est local et, si $F$ est global, un isomorphisme
$$ G(\adef)\tG(F)\tZ(\adef)\bsl\tG(\adef)\to Q_F\ptf$$

 On note $\P$ le sous-groupe de Borel 
de $G$ et $\tP$ celui de $\tG$ formŽs de matrices triangulaires supŽrieures,
$\M$ le tore des matrices diagonales de $G$ et $\tM$ pour $\tG$. 
On note $\U$ le radical unipotent de $\P$ et de $\tP$. Le normalisateur de $\U$ dans $G$ est $\P$.
Le radical unipotent $U$ du sous-groupe de Borel $\P$
est l'union de deux sous-ensembles $\P$-invariants
ˆ savoir $\{1\}$ et $U'$ formŽ d'unipotents rŽguliers.

Pour $\gamma\in\G(F)$ on note $I_\gamma$ son centralisateur schŽmatique. 
Les centralisateurs des ŽlŽments semi-simples sont lisses et connexes.
On rappelle qu'un ŽlŽment semi-simple $\gamma\in\G(F)$ est dit elliptique 
si le centralisateur du tore dŽployŽ maximal $A_{\gamma}$ du centre 
de $I_{\gamma}$ est Žgal ˆ $\G$. 
Pour $G=SL(2)$ cela signifie que $A_{\gamma}$ est trivial.
Pour $\gamma$ unipotent rŽgulier le centralisateur $I_\gamma$ est non connexe
et non rŽduit en caractŽristique $2$.
On observera que l'ŽlŽment neutre est elliptique et unipotent.

On notera $\G_{ell}$ (resp. $\G_{par}$) le sous-ensemble de $\G(F)$
formŽ des ŽlŽments elliptiques (resp. non elliptiques).

\Section{Compact maximal et dŽcomposition d'Iwasawa}

Soit $F$ un corps local. On munit l'espace vectoriel $F\oplus F$ 
de la norme euclidienne pour les corps archimŽdiens et de la norme
``$\sup$'' pour les corps non archimŽdiens et
on notera $\vert\vert(x,y)\vert\vert$ la norme du vecteur  $(x,y)\in F\oplus F$.
On note $\tK$  le groupe d'isomŽtrie de la norme  sur $F\oplus F$.
On a $\tK=O(2,F)$ si $F=\RM$, $\tK=U(2)$ si $\F=\CM$ et 
$\tK=GL(2,\OF)$ si $F$ est non archimŽdien.
C'est le compact maximal naturel de $\tG(F)$ et $K=\tK\cap\G(F)$ est celui de $\G(F)$. 
On dispose alors des dŽcompositions d'Iwasawa
$\G(F)=\M(F)U(F)K$ {et} $\tG(F)=\tM(F)\U(F)K$.

Pour un groupe sur les adles le compact maximal que nous utiliserons
sera le produit sur toutes les places du compact naturel sur les corps locaux.
Il sera encore notŽ $K$ si il n'y a aucune ambigu•tŽ et on a les dŽcompositions d'Iwasawa globales
$\G(\adef)=\M(\adef)U(\adef)K$ et $\tG(\adef)=\tM(\adef)\U(\adef)K$.

On note $\alpha$ la racine positive de $\M$ dans $\U$ et, pour $m\in\M(\adef)$, on pose
 $$\HP(m)=\log(\vert m^\alpha\vert^{1/2})/\log q_F=\log_{q_F}(\vert m^\alpha\vert^{1/2})\ptf$$ 
On a ainsi dŽfini une application surjective
$$\HP:M(\adef)\to \mathfrak a_\M$$
o $\mathfrak a_\M\simeq\RM$ si $F$ est un corps de nombres et
$\mathfrak a_\M\simeq\ZM$ pour les corps de fonctions.
On prolonge $\HP$ en une fonction
$$\HP:\G(\adef)\to \mathfrak a_\M$$
au moyen de la dŽcomposition d'Iwasawa : si $x=muk$ est une telle dŽcomposition
on pose $\HP(x)=\HP(m)$.

ConsidŽrons les matrices
$$u_n=\begin{pmatrix}1&n\cr 0&1\end{pmatrix}\com{,}
w=\begin{pmatrix}0&-1\cr 1&0\end{pmatrix}\com{et}
m=\begin{pmatrix}a&0\cr 0&a\mun\end{pmatrix}\ptf$$
Soit $wu_n=muk$ une dŽcomposition d'Iwasawa pour $wu_n$.
Donc $\HP(wu_nk)=\HP(m)$ et
$(0,1)wu_n=(1,n)$ mais
$$\vert\vert(0,1)muk\vert\vert=\vert\vert(0,1)m\vert\vert=
\vert\vert(0,a\mun)\vert\vert$$ 
puisque $k\in K$ est une isomŽtrie et donc
$\vert\vert(1,n)\vert\vert=\vert a\mun\vert$.  En rŽsumŽ:

\begin{lemma}\label{wuk}
$$\HP(wu_nk)=\HP(m)=\log_{q_F}\vert a\vert=
-\log_{q_F}\big(\vert\vert(1,n)\vert\vert\big)\ptf$$
\end{lemma}

On dispose sur $\G(\adef)$, $U(\adef)$ et sur les tores des mesures de Tamagawa. On pose
$$[\G]=:\G(F)\bsl\G(\adef)\qquad [U]=:U(F)\bsl U(\adef)\ptf$$
On a
$$\tau(G)=\vol[\G]=1\comm{et}\tau(U)=\vol[U]=1\ptf$$
\begin{lemma}
Pour un tore anisotrope $\T$ de $\G$ 
\cad un groupe isomorphe au groupe $\Tef$ des ŽlŽments de norme 1 dans une
extension quadratique sŽparable $E/F$ on a
$$\tau(T)=\vol[T]=2\ptf$$
\end{lemma}
\begin{proof}
En effet, d'aprs Ono
$$\tau(T)=\frac{\#H^1(\adef/F,T)}{\#\ker^1(F,T)}\ptf$$ On a la suite exacte 
$$1\to F^\times\to E^\times\to T(F)\to 1$$ o l'application $E^\times\to T(F)$ est $e\mapsto e/\overline e$.
La dualitŽ de Tate-Nakayama montre alors que 
$$\#H^1(\adef/F,T)=2\quad\hbox{et}\quad\#\ker^1(F,T)=1\ptf$$
\end{proof}
Soit $\T$ un tore dans $\G$.
On note $ \Delta_\T$ la valeur absolue du ``dŽnominateur de Weyl'' pour $SL(2)$
$$ \Delta_\T(t)=
\Big\vert (t^{\alpha/2}-t^{-\alpha/2})\Big\vert 
=\vert t^{\alpha}\vert^{1/2}\vert 1-t^{-\alpha}\vert
$$
 Soit $w_\T$ l'ordre du groupe de Weyl: $N_\T(F)/\T(F)$ o $N_\T$ 
 le normalisateur de $\T$. On a la formule d'intgration de Weyl:

\begin{lemma}\label{Weyl}
Pour $f\in\ctyc\big(\G(F)\big)$ 
 $$\int_{\G(F)}f(x) dx
=\sum_\T w_\T\mun\int_{\T(F)\bsl\G(F)}\int_{\T(F)} \Delta_\T(t)^2f(x\mun tx)\,dt\,\ddx$$
o la somme porte sur un ensemble de reprŽsentants des classes de conjugaison de tores $\T$.
\end{lemma}

\Section{Conjugaison stable}
Dans la littŽrature la conjugaison stable n'est dŽfinie que pour les
ŽlŽments semi-simples d'un groupe rŽductif $H(F)$
pour $F$ de caractŽristique $0$.
Pour $x\in H(F)$ semi-simple on note $I_x$ le sous-groupe engendrŽ 
par $H_x$ la composante neutre du centralisateur  de $x$
dans $\H$ et le centre $Z$ de $\H$: $I_x=H_x.Z$.
On appelle $I_x$ le centralisateur stable. On dit que $x$ et $x'=yxy\mun$
avec $x$ et $x'$ dans $\G(F)$ et $y\in\G(\Fsep)$
sont stablement conjuguŽs si le cocycle $a_\sigma=y \sigma(y)\mun$ prend ses valeurs
dans $I_x$. L'orbite stable est paramŽtrŽe par 
l'ensemble de cohomologie galoisienne $H^0(F,I_x\bsl\H)$.

On notera $\Fbar$ une cl™ture algŽbrique de $F$ et
 $\Fsep$ la cl™ture sŽparable de $F$ dans $\Fbar$. 
Pour $G=SL(2)$ deux ŽlŽments semi-simples $x$ et $x'$ de $G(F)$ sont dits stablement conjuguŽs
 s'ils sont conjuguŽs sur la cl™ture algŽbrique:
$$x'=yxy\mun\comm{pour un}y\in G(\Fbar)\ptf$$
On observe que pour $x\in SL(2,F)$ le centralisateur stable $I_x$ est le centralisateur.
L'ensemble des points rationnels de l'orbite de $x$ sous $G(\Fbar)$
est paramŽtrŽ par l'ensemble de cohomologie plate (flat topology)
$$H^0_{f}(F,I_x\bsl\G).$$
L'ensemble $C(x)$ des orbites par conjugaison sous $\G(F)$ dans l'ensemble 
des points rationnels de l'orbite gŽomŽtrique
est paramŽtrŽ par $$C(x)\simeq\coker[H^0_{f}(F,\G)\to H^0_{f}(F,I_x\bsl\G)]
\simeq\ker[H^1_{f}(F,I_x)\to H^1_{f}(F,\G)]\ptf$$
Mais on sait que le $H^1$ de notre $\G$ est trivial on a donc ici
$$C(x)\simeq H^1_{f}(F,I_x)\ptf$$

Supposons que $x$ est semi-simple non central;
ses valeurs propres sont distinctes et engendrent une extension quadratique sŽparable $E$
(Žventuellement dŽployŽe).
Il est ŽlŽmentaire de voir que si $x$ et $x'$ sont stablement conjuguŽs ils sont
alors conjuguŽs dans $\G(E)$.
Donc, pour les ŽlŽments semi-simples la conjugaison sur la cl™ture algŽbrique co•ncide 
avec la conjugaison sur la cl™ture sŽparable. 
Maintenant pour $y\in G(\Fsep)$ et $\sigma\in\Gal(\Fsep/F)$ on a aussi
$$x'=\sigma(y)x\sigma(y)\mun\comm{et donc}a_\sigma=y \sigma(y)\mun\in I_x\ptf$$
On a ainsi dŽfini un ŽlŽment du noyau du morphisme entre ensembles pointŽs en cohomologie Žtale
$$\ker [H^1_{e}(F,I_x)\to H^1_{e}(F,G)]=H^1_{e}(F,I_x)$$
puisque $H^1_{e}(F,G)$ est trivial. Le cocycle $a_\sigma$ dŽfinit
une classe $c(x)$ dans $H^1_{e}(F,I_x)$. Cet ensemble dŽcrit donc l'ensemble des
classes de conjugaison rationnelles dans la classe de conjugaison stable d'un ŽlŽment 
semi-simple\footnote
{En fait ce rŽsultat est sans surprise car il est connu
que pour les groupes lisses et connexes la cohomologie plate co•ncide avec la cohomologie Žtale.}
.
Maintenant on a l'inclusion de $G$ dans $\tG$; notons $\widetilde{I}_x$ 
le centralisateur de $x$ dans $\tG$. 
Puisque $x$ est semi-simple $\widetilde{I}_x$ est un tore isomorphe ˆ $E^\times$.
La 1-cohomologie du tore $\widetilde{I}_x$
est triviale et donc le 1-cocycle $a_\sigma$ dŽfinissant $c(x)$
est un 1-cobord ˆ valeurs dans $\widetilde{I}_x$.
 Donc, quitte ˆ modifier $y$
par un ŽlŽment de $\widetilde{I}_x$, on obtient un $y'\in\tG(\Fsep)$
tel que $y'\sigma(y')\mun=1$ et donc on a $y'\in\tG(F)$ avec $x'=y'x(y')\mun$. 
On a ainsi montrŽ que deux ŽlŽments semi-simples de $G(F)$
stablement conjuguŽs sont conjuguŽs sous $\tG(F)$.
En d'autres termes, dans $\G(F)$ la conjugaison stable co•ncide avec
la conjugaison ordinaire sous $\tG(F)$.

Soit $x\in\G(F)$ semi-simple rŽgulier. Son centralisateur dans $\G$ (resp $\tG$) est un tore
$\T$ (resp. $\tT$). Il rŽsulte des remarques ci-dessus que l'application naturelle
$$H^0_e(F,\T\bsl\G)\to H^0_e(F,\tT\bsl\tG)\simeq \tT(F)\bsl\tG(F)$$
induite par l'inclusion $\G\subset\tG$ est un isomorphisme
qui sera utilisŽ systŽmatiquement dans la suite pour la dŽfinition et le calcul
des intŽgrales $\kappa$-orbitales.

\Subsection{Distributions stablement invariantes}\label{qs}

Soit $F$ un corps local.
On dit qu'une distribution sur $\G(F)$ est  invariante 
si elle est invariante par conjugaison et qu'elle est stablement invariante 
si elle est invariante par conjugaison stable.
Dans notre cas cela se rŽduit ˆ demander l'invariance par conjugaison sous 
$\tG(F)$ dans le cas local.  
Dans le cas d'un corps global une distribution adŽlique stablement invariante est une distribution
stablement invariante localement partout.

\Section{DonnŽes endoscopiques}
Nous ne ferons pas ici la thŽorie de l'endoscopie en gŽnŽral et nous nous contenterons
du cas $\G=SL(2)$. Une donnŽe endoscopique est un couple $\EC=\{H_\EC,\kappa_\EC\}$
o $H_\EC$ est un groupe rŽductif qui ici sera soit $\G$ soit un tore $\T$ de $\G$,
et $\kappa_\EC$ un caractre de $\G(F)\tZ(F)\bsl\tG(F)$ si $F$ est local ou de 
$\G(\adef)\tG(F)\tZ(\adef)\bsl\tG(\adef)$ si $F$ est global.
Un tel caractre, qui peut tre vu comme un caractre de $Q_F$, est nŽcessairement d'ordre 2.
Les classes d'Žquivalence de donnŽes endoscopiques pour $\G$ sont de deux types
\pni Type 1 : $\EC=\{SL(2),1\}$
\pni Type 2 : $\EC=\{\Tef,\vef\}$ o $E/F$ est un extension quadratique sŽparable, dŽployŽe ou non,
et $\Tef$ le sous groupe des ŽlŽments de norme 1 dans $E^\times$.

\Chapter{Endoscopie gŽomŽtrique locale}

\Section{Stabilisation et transfert g\'eom\'etrique local}

Soient $F$ un corps local et $f\in\ctyc\big(G(F)\big)$.
Soit $T$ un tore dans $\G=SL(2)$ et soit $t\in\T(F)$
rŽgulier; son centralisateur est $T(F)$.
L'intŽgrale orbitale $\orb(t,f)$ de $t\in\T(F)$ est, par dŽfinition, l'intŽgrale
$$\orb(t,f)=\int_{T(F)\bsl\G(F)}f(x\mun tx)dx\ptf$$
On note $E$ l'extension quadratique sŽparable (Žventuellement dŽployŽe)
dŽfinie par $t$ et $T$ est isomorphe ˆ $\Tef$, le sous-groupe des ŽlŽments de norme 1 dans 
$$E^\times=\tTef(F)\ptf$$
On note $\vef$ le caractre de $F^\times$ associŽ ˆ l'extension quadratique $E/F$
par la thŽorie du corps de classe. Si $E$ est un corps
c'est le caractre non trivial du groupe
$\Nef E^\times\bsl F^\times$ qui est d'ordre 2.
C'est le caractre trivial si $E=F\oplus F$.

Soit $\kappa$ un caractre du groupe compact 
$$\tZ(F)G(F)\bsl\tG(F)\simeq Q_F= (F^\times)^2\bsl F^\times\ptf$$ 
On s'intŽresse aux intŽgrales du type
$$\int_{\tZ(F)T(F)\bsl\tG(F)}\kappa(\det \tx)f(\tx\mun t \tx)d\tx$$
On observe que $$\tx\mapsto f(\tx\mun t \tx)$$ est invariant ˆ gauche par 
le centralisateur $\tT(F)$ de $t$ dans $\tG(F)$ et donc une telle intŽgrale est nulle sauf si
$\kappa$ est trivial sur $\tT(F)$
ce qui impose $\kappa=1$ ou $\kappa=\vef$ et on pose
$$\orb^\kappa(t,f)=\int_{H^0_e(F,\T\bsl\G)}\kappa(\det \tx)f(\tx\mun t \tx)d\tx=\int_{\tT(F)\bsl\tG(F)}\kappa(\det \tx)f(\tx\mun t \tx)d\tx\ptf$$
L'inclusion
$$\T(F)\bsl\G(F)\subset\tT(F)\bsl\tG(F)$$ 
munit $\tT(F)\bsl\tG(F)$ d'une mesure invariante ˆ droite 
et c'est celle qui est utilisŽe pour calculer l'intŽgrale.
Si $E$ est un corps, \cad si $\vef\not=1$ on a \label{staborb}
$$\orb^\kappa(t,f)=\orb(t,f)+\kappa(\det \tx)\orb(t',f)
\com{et}
\orb(t,f)=\frac{1}{2}\Big(\orb^1(t,f)+\orb^{\vef}(t,f)\Big)
$$
o $t'=\tx\mun t\tx$ et $t'$ est stablement conjuguŽ mais non conjuguŽ ˆ $t$.
Par contre dans le cas o $E=F\oplus F$ on a $\orb^\kappa(t,f)=\orb(t,f)$.
L'intŽgrale $\orb^1(t,f)$ est appelŽe intŽgrale orbitale stable.

\Subsection{Transfert principal}
Lorsque  $\EC=\{SL(2),1\}$ le transfert est simplement l'identitŽ
$$f^\EC=f$$

Dans la suite de ce paragraphe $\EC=\{\Tef,\vef\}$. 
On choisit un isomorphisme entre $\tT(F)$ et $E^\times$ \cad une diagonalisation simultanŽe
de tous les ŽlŽments de $\tT(F)$. On a alors un isomorphisme entre $T(F)$ et le tore $\Tef(F)$
des ŽlŽments de norme 1 dans $E^\times$; on note $t'$ l'image de $t\in\T(F)$ par cet isomorphisme.
On fixe un ŽlŽment rŽgulier $\bs\tau$ dans $\tT(F)$.
On note $\gamma$ et $\overline\gamma$ les valeurs propres de $t'$ (resp. $\tau$
et $\overline\tau$ les valeurs propres de $\bs\tau$).
\begin{definition}\label{FTr} \label{ft}
Lorsque $\EC=\{\Tef,\vef\}$ le facteur de transfert est
l'expression suivante
$$\Delta^\EC(t,t')=\cft\,\vef\Big(\frac{\gamma-\overline\gamma}{\tau-\overline\tau}\Big)
{\vert{\gamma-\overline\gamma}\vert}$$ 
o $\cft$ est une constante.
\end{definition}
 Une variante consiste ˆ Žcrire $$t'=a+b\tau\in E^\times\simeq\tTef(F)$$
avec $a$ et $b$ dans $F$ et on a alors
$$\Delta^\EC(t,t')=\cft\,\vef(b){\vert{b(\tau-\overline\tau)}\vert_E}\ptf$$ 
Dans la suite nous prendrons $$\cft=\lambda(E/F,\psi)\mun$$ ce qui est l'inverse
du choix fait dans \cite{LL}. Nous verrons plus loin (\ref{choix}) que le choix fait ici semble plus naturel.
On posera
$$\orb^\EC(t',f)=\Delta^\EC(t,t')\orb^\vef(t,f)\ptf$$
\begin{remark}{\rm  Les intŽgrales $\kappa$-orbitales $\orb^\kappa(t,f)$ se prtent ˆ une
prŽ-stabilisation: elles permettent d'Žcrire les intŽgrales orbitales comme des
sommes d'intŽgrales faisant intervenir les classes de conjugaison stable. Le produit
avec le facteur de transfert fournit des expressions $\orb^\EC(t',f)$ qui sont
constantes sur les classes de conjugaison stable.  Nous allons maintenant voir que, de
plus, elles se prtent au transfert endoscopique.
}\end{remark}

Nous allons montrer qu'il existe une fonction $$f^{\EC}
\in \ctyc\big(H_\EC(F)\big)$$
appelŽe transfert de $f$ pour la donnŽe endoscopique $\EC$ telle que, pour $t$ rŽgulier, on ait
$$f^{\EC}(t)=\orb^\EC(t,f)\ptf$$
On observera que la dŽfinition du transfert dŽpend du choix des mesures de Haar
utilisŽes pour calculer les intŽgrales orbitales.
Nous distinguons les diffŽrents cas.

\Section{Transfert dŽployŽ}

Soit  $\EC=\{\Tef,\vef\}$ avec $E=F\oplus F$ et $\vef=1$. Alors
$\Tef$ est un tore dŽployŽ et ˆ conjugaison prs on peut supposer que $\Tef$
est le tore diagonal $\M$.
\begin{lemma}\label{deploy} 
La fonction $t\mapsto\Delta^\EC(t,t)\orb(f,t)$ dŽfinie pour $t\in\M(F)-\Z(F)$
se prolonge en une fonction $f^\EC$ lisse sur $\Tef(F)\simeq\M(F)$.
Plus prŽcisŽment, si $$t=\begin{pmatrix}a&0\cr0&a\mun\end{pmatrix}\in\M(F)$$
alors
$$f^\EC(t)=\int_K\int_{F}f\Big(k\mun 
\begin{pmatrix}a&n\cr0&a\mun\end{pmatrix}k\Big)dn\,dk\ptf$$
\end{lemma}

\begin{proof}
Le facteur de transfert vaut $\Delta^\EC(t,t)=\vert a-a\mun\vert\ptf$
On a donc
$$\Delta^\EC(t,t)\orb(f,t)=\vert a-a\mun\vert\int_K\int_{U(F)}f\Big(k\mun u\mun
\begin{pmatrix}a&0\cr0&a\mun\end{pmatrix}uk\Big)du\,dk$$
soit encore
$$\Delta^\EC(t,t)\orb(f,t)=\int_K\int_{F}f\Big(k\mun 
\begin{pmatrix}a&n\cr0&a\mun\end{pmatrix}k\Big)dn\,dk\ptf$$
\end{proof}
On observera que pour $\bs z\in\Z(F)$ (o $\bs z=z.\bs 1$ et $\bs 1$ est la matrice unitŽ)
$$f^{\EC}(\bs z)=\orb^1(\bs z{\bs\nu},f)$$
o
$$\orb^1(\bs z{\bs\nu},f)=\int_{K}\int_{U(F)} 
f\Big(k\mun \begin{pmatrix}z &n\cr 0& z\end{pmatrix}k\Big)\, dn\,dk
\com{et}
{\bs\nu}=\begin{pmatrix}1&1\cr0&1\cr\end{pmatrix}\ptf$$
\begin{corollary}\label{lfddep}
Si $F$ est non archimŽdien,
si $f$ la fonction caractŽristique de $K$ et si $dn$ et $dk$ sont les mesures canoniques
(volume 1 pour les entiers de $F$ et volume 1 pour $K$),
alors $f^\EC$ est la fonction caractŽristique du sous-groupe $\OF^\times$ des unitŽs de $F^\times$.
\end{corollary}

\Section{Transfert elliptique}

Soit $F$ un corps local. Soit $E$ un corps  extension quadratique sŽparable de
$F$. Le groupe $E^\times$ des ŽlŽments inversibles est un tore elliptique $\tTef(F)$ 
que l'on suppose plongŽ dans $\tG(F)$ comme dans \ref{ext}:
$$E^\times=\tTef(F)\subset\tG(F)\simeq GL(E)\ptf$$ 
On  note $E^1$ le sous groupe des ŽlŽments de norme 1
\cad $E^\times\cap\G(F)$
et on posera $E^\star=E^1-\{\pm1\}$.
On pourra observer que l'on a une bijection naturelle
$$E^1\bsl\G(F)\cup E^1\bsl\G(F) \eta\to E^\times\bsl\tG(F)$$
o $\eta$ vŽrifie $\vef(\det(\eta))=-1$. 

On note $A_+$ le semi-groupe des matrices
$$\alpha(\mu)=\begin{pmatrix}\mu&0\cr0&1\cr\end{pmatrix}$$
avec $\vert \mu\vert\ge1$. 
\begin{proposition}\label{decE} 
Tout $g\in\tG(F)$ s'Žcrit $g=e\alpha(\mu)k$ avec
$e\in \tTef(F)=E^\times$, $k\in K $ et $\vert\mu\vert\ge1$:
$$\tG(F)=\tTef(F)A_+ K \ptf$$
Soient $d\mu$ la mesure de Haar standard pour le groupe $F^\times$
et $dk$ la mesure normalisŽe pour $K$. 
Pour $e\in E^\star$ on a la formule d'intŽgration
$$\int_{E^\times\bsl\tG(F)}f(\tx\mun e \tx)\,\dtx
=\int_K\int_{\vert\mu\vert\ge1}f(k\mun \alpha(\mu)\mun e \alpha(\mu)k)C(\mu)\,d\mu\,dk$$
o $C(\mu)=\vert\mu-\mmu\vert$ pour $\vert\mu\vert\ge1$ si $F=\RM$ 
pour un choix convenable de la mesure de Haar sur $\G(\RM)$.
Si $F$ non archimŽdien et $E/F$ ramifiŽ on a si $\vert\mu\vert\ge1$ 
$$C(\mu)=2q^m=2\vert\mu\vert\ptf$$
 Si $E/F$ est non ramifiŽ $C(\mu)=1$ si $\vert\mu\vert=1$ et 
$$C(\mu)=\bigg(1+\frac{1}{q}\bigg)q^m=\bigg(1+\frac{1}{q}\bigg)\vert\mu\vert
\com{pour}\vert\mu\vert>1
\ptf$$

\end{proposition}

\begin{proof} 
Si $F=\RM$ et $E=\CM$ la dŽcomposition de Cartan peut s'Žcrire
$$\tG(F)=\CM^\times A_+K$$
o $K=SO(2,\RM)$. La formule d'intŽgration 
$$\int_{\tG(F)}\varphi(x)\,dx=
\int_{(\tk_1,k_2)\in \tK\times K}\int_{\vert\mu\vert\ge1}\varphi(\tk_1\alpha(\mu)k_2)
C(\mu)\,d\tk_1\,dk_2\,d\mu$$
est classique (voir par exemple \cite[Chap.~X, Proposition 1.17]{Helgason}). 
Ceci conclut la preuve lorsque $F=\RM$.
Soit maintenant $F$ un corps local non archimŽdien. On observera que
la mesure invariante sur $E^\times\bsl\tG(F)$ est dŽterminŽe par le choix de mesures de Haar
donnant la masse 1 au quotient $E^1\bsl K$ qui est un ouvert de $E^\times\bsl\tG(F)$.
L'argument ci-dessous (dŽjˆ utilisŽ dans \cite{LL})
est empruntŽ ˆ la dŽmonstration du Lemme 7.3.2 de \cite{JL}. 
L'anneau $\OE$ des entiers de $E$
est un $\OF$-module libre de rang 2; on choisit $\delo\in\OE$ tel que
$$\OE=\OF\oplus\OF\delo\ptf$$ 
En particulier $\{1,\delo\}$ est une base de $E$ vu comme $F$-espace vectoriel.
On note $\tK $ le sous-groupe compact maximal
$$\tK =GL(\OE)\subset GL(E)\simeq GL(2,F)\ptf$$
On note $\mathfrak M$ l'ensemble des $\OF$-modules $\mathfrak{m}$ de 
type fini dans $E$ et qui l'engendrent comme espace vectoriel sur $F$.
En d'autres termes $\mathfrak{m}$ contient une $F$-base de $E$. 
Maintenant $\mathfrak{m}$ et $\mathfrak{m}'$ sont dits equivalents s'il existe $e\in E^\times$
tel que $\mathfrak{m'}=e\mathfrak{m}$.
On observe que tout $g\in\tG(F)$ dŽfinit un module
$\mathfrak{m}=g\OE\in\mathfrak M$ et ils sont tous de cette forme. De plus,
$\mathfrak{m}=g\OE$ et $\mathfrak{m}'=g'\OE$ sont Žquivalents si et seulement si 
$$g'=eg\tk\comm{avec} e\in \tTef(F)=E^\times\comm{et}\tk\in \tK =\tG(\OF)\ptf$$
Donc les classes d'Žquivalences dans $\mathfrak M$ sont en bijection avec l'ensemble
des doubles classes $$\tTef(F)\bsl\tG(F)/ \tK \ptf$$
Un $\mathfrak m\in\mathfrak M$ qui est de plus un sous-anneau de $\OE$ est appelŽ un ``ordre''.
La classe d'Žquivalence de $\mathfrak{m}$ contient un unique ``ordre'' $\mathfrak{a}$:
$$\mathfrak{a}=\{a\in\OE\,\vert\, a\mathfrak{m}\subseteq \mathfrak{m}\}\ptf$$
On a donc une bijection entre l'ensemble des ``ordres'' et l'ensemble des doubles classes.
Maintenant un ``ordre'' $\mathfrak{a}$ admet une base de la forme $(1,\delta)$:
$$\mathfrak{a}=\OF\oplus \delta\OF\comm{o}\delta=\varpi_F^m\delo$$
et o $m\in\NM$ est uniquement dŽterminŽ. Donc $\mathfrak{a}=\varpi_F^m\alpha(\varpi_F^{-m})\OE$. 
On a ainsi montrŽ que chaque double classe
$$\tTef(F)\bsl\tG(F)/ \tK $$ contient un unique ŽlŽment de la forme $\alpha(\varpi_F^{-m})$
et on remarque que 
$$\tk=\begin{pmatrix}\eta &0\cr 0&1\end{pmatrix} k$$
avec $k\in K$ et $\vert\eta\vert=1$. Reste ˆ calculer $C(\mu)$. 
On Žcrira $\tT$ pour $\tTef$ et on
pose 
$$\tT(F,\mu):=\alpha(\mu)\mun\tT(F)\alpha(\mu)\ptf$$
On a alors
$$C(\mu)=\card\Big\{\tT(F)\alpha(\mu)\tK/ \tZ(F)\tK\Big\}=\card\Big\{\tT(F,\mu)/ \tT(F,\mu)\cap \tZ(F)\tK\Big\}\ptf$$
On observe que
$$\tT(F,\mu)\cap \tZ(F)K=\Bigg\{g=\alpha \tk\com{o}\alpha\in F^\times\com{et}
\begin{pmatrix}a&-b\mmu\dfr\cr b\mu &a+b\tfr\cr\end{pmatrix}
\in\tK\Bigg\}$$
ce qui impose $$a\in\OF\com{,}b\mu\in\OF\com{et}
a^2+ab\tfr+b^2\dfr\in \OF^\times\ptf$$ 
On pose $E(\mu)=1+\mu\mun\OE$. 
Comme $\b\in\mmu\OF=\pfr_F^m$ avec $m>1$ si $\vert\mu\vert>1$
alors $a$ est une unitŽ et on peut donc choisir $\alpha$ de sorte que $a=1$.
Il en rŽsulte que $C(\mu)$ est l'ordre du quotient $E^\times/F^\times E(\mu)$.
Supposons tout d'abord $E/F$ non ramifiŽ.
Dans ce cas
$$F^\times\simeq \OF^\times\times\varpi_F^\ZM\com{et}
 E^\times\simeq \OE^\times\times\varpi_F^\ZM  \ptf$$
Comme $\tT(F)\simeq E^\times$
et $\tZ(F)\simeq F^\times$ le cas $\vert\mu\vert=1$ est clair. 
Supposons dŽsormais que $\vert\mu\vert>1$. Comme $E/F$ est non ramifiŽ
$ E(\mu)=1+\pfr_E^m$ et donc
 $$C(\mu)=\card\big(\OE^\times/\OF^\times E(\mu)\big)=
\frac{\card\big(\OE^\times/E(\mu)\big)}{\card\big(\OF^\times/\OF^\times\cap E(\mu)\big)}
=\frac{(q^2-1)q^{2(m-1)}}{(q-1)q^{(m-1)}}$$
soit encore
$$C(\mu)=\bigg(1+\frac{1}{q}\bigg)q^m=\bigg(1+\frac{1}{q}\bigg)\vert\mu\vert\ptf$$
Si $E/F$ est ramifiŽ $\varpi_F\OE=\varpi_E^2\OE$ et 
$E(\mu)=1+\pfr_E^{2m}$ si $\vert\mu\vert=q^m$ et $m>0$. Donc
$$C(\mu)=\frac{\card({\varpi_E^\ZM\OE}/{\varpi_F^\ZM\OE})\,\card\big(\OE^\times/E(\mu)\big)}
{\card\big(\OF^\times/\OF^\times\cap E(\mu)\big)}
=\frac{2(q-1)q^{(2m-1)}}{(q-1)q^{(m-1)}}=2q^m=2\vert\mu\vert\ptf$$
\end{proof}

\begin{proposition}\label{elltrans}
Il existe une fonction lisse $f^\EC$ sur $H_\EC(F)=\Tef(F)$ vŽrifiant
$$f^{\EC}(t)=\Delta^\EC(t,t)\orb^\kappa(t,f)$$
avec $\kappa=\vef$ lorsque $t\in\Tef(F)$ est rŽgulier.
Pour $\bs z\in\Z(F)$ (o $\bs z=z.\bs 1$ et $\bs 1$ est la matrice unitŽ)
$$f^{\EC}(\bs z)=\orb^\kappa(\bs z{\bs\nu},f)$$
o
$$\orb^\kappa(\bs z{\bs\nu},f)=\int_{K}\int_U \kappa(n)
f\Big(k\mun \begin{pmatrix}z &n\cr 0& z\end{pmatrix}k\Big)\, dn\,dk
\com{et}
{\bs\nu}=\begin{pmatrix}1&1\cr0&1\cr\end{pmatrix}\ptf$$
\end{proposition}

\begin{proof} Nous reproduisons la preuve donnŽe dans \cite[Lemma 2.1]{LL} 
valable en toute caractŽristique. 
Un ŽlŽment $t\in\Tef(F)$ s'Žcrit sous la forme $t=\a+\b\tau$  et donc 
$$t=\begin{pmatrix}a &-\b\dfr\cr \b&a+\b\tfr
\end{pmatrix}$$
Maintenant, si
$$\tx=\begin{pmatrix}\aa &\bb\cr \cc&\dd\end{pmatrix}$$
un calcul ŽlŽmentaire montre que
$$\tx\mun t \tx=\frac{1}{\det(\tx)}
\begin{pmatrix}\star &-\b \Nef(\bb+\dd\tau)\cr \b \Nef(\aa+\cc\tau)&\star
\end{pmatrix}$$
En posant
$$\tx\mun t \tx=\begin{pmatrix}\aaa &\bbb\cr \ccc&\ddd\end{pmatrix}$$
on a
$$\ccc= \frac{\b \Nef(\aa+\cc\tau)}{\det(\tx)}\comm{et}\bbb= \frac{-\b \Nef(\bb+\dd\tau)}{\det(\tx)}\ptf$$
On en dŽduit que $$\kappa(\ccc)=\kappa(-\bbb)=\kappa(\det \tx)\kappa(b)\ptf$$
Par ailleurs, si $\tx=e \alpha(\mu) k$ avec 
$e\in\tTef(F)$, $\alpha(\mu)\in A_+$ et $k\in K $ on a aussi
$$\tx\mun t \tx=k\mun\begin{pmatrix}a &-\mmu\b\dfr\cr \mu\b& a+\b\tfr
\end{pmatrix}k\ptf$$
Nous allons Žvaluer l'intŽgrale
$$\Delta^\EC(t,t)\orb^\kappa(t,f)=
\lambda(E/F,\psi)\mun\int_{\tZ(F)T(F)\bsl\tG(F)}D(b)\kappa(c_2)f(\tx\mun t \tx)d\tx$$
o $D(b)=\vert b(\tau-\overline\tau)\vert_E$.
Compte tenu de la dŽcomposition:
$$\tG(F)=\tTef(F)A_+ \tK $$
Žtablie en \ref{decE} on a 
 $$\Delta^\EC(t,t)\orb^\kappa(t,f)=\lambda(E/F,\psi)\mun\,A$$
o
$$A=\int_K\int_{\vert\mu\vert\ge1}D(b)C(\mu)\kappa(\b\mu\mun)
f\Big(k\mun \begin{pmatrix}a &-\b\mmu
\dfr\cr \b\mu& a+\b\tfr\end{pmatrix}k\Big)\,\,d\mu\,dk\ptf$$
Si $t$ tend vers $\bs z$ dans $\T(F)$ alors $a$ tend vers $z\in\bs\mu_2(F)$
 et $b$ tend vers $0$. Maintenant
il existe une fonction lisse sur $F^2\times K$: 
$$\ve:(a,b,\mu,k)\mapsto \ve(a,b,\mu,k)$$ qui tend vers zŽro avec $b$ uniformŽment pour $\vert\mu\vert\ge1$ 
et $k\in K$ telle que, si $b\mu$ est bornŽ (ce qui est loisible puisque $f$ est ˆ support compact), on ait
$$f\Big(k\mun \begin{pmatrix}a &-\b\mmu\dfr\cr \b\mu& a+\b\tfr\end{pmatrix}k\Big)=
\big(1+\ve(a,b,\mu,k)\big)f\Big(k\mun \begin{pmatrix}z &0\cr u& z\end{pmatrix}k\Big)$$
avec $u=b\mu$.
Pour un corps non archimŽdien on peut prendre $\ve(a,b,\mu,k)=0$ si $b$ est assez petit.
Enfin on utilise que ${\bs\nu}$ est conjuguŽ sous $K$ de son transposŽ inverse.

\end{proof}

La combinaison de \ref{deploy} et de \ref{elltrans} fournit le

\begin{theorem}\label{trans}
Soit  $\EC=\{\Tef,\vef\}$ o $E/F$ est une extension quadratique
sŽparable (dŽployŽe ou non). On a plongŽ $\Tef$ dans $\G$. La fonction
$\orb^\EC(t,f)$ dŽfinie pour $t\in\Tef(F)-\Z(F)$, \cad pour $t$ rŽgulier,
se prolonge en une fonction lisse $f^\EC$ sur le groupe compact $\Tef(F)$.
\end{theorem}

On dispose d'une forme plus prŽcise dans le cas particulier o 
tout est non ramifiŽ et o $f$ est la fonction caractŽristique de $K$. 
C'est le rŽsultat dŽsormais connu 
sous le nom de ``Lemme Fondamental''. 
On suppose les mesures de Haar normalisŽes de sorte que
$d\mu$ donne la mesure 1 ˆ $\OF^\times$ et $dk$ la mesure 1 ˆ $K$.

\begin{theorem}\label{fond}
Supposons $F$ non archimŽdien, $E/F$ non ramifiŽ et que
$f$ est la fonction caractŽristique de $K=SL(2,\OE)$. Si la constante $c$
du facteur transfert vaut 1 alors
$$t\mapsto f^{\EC}(t)=\Delta^\EC(t,t)\orb^\kappa(t,f)$$
est la fonction caractŽristique de $\Tef(\OF)=\OE^1$ le sous-groupe
des ŽlŽments de norme 1 dans $\OE^\times$.
\end{theorem}

\begin{proof}Si $E=F\oplus F$ c'est le corollaire \ref{lfddep}.
Supposons maintenant que $E$ est un corps.
Pour $t=a+b\tau$ on a
$$f^{\EC}(t)=\Delta^\EC(t,t)\orb^\kappa(t,f)=\int_{\vert\mu\vert\ge1}D(b)C(\mu)\kappa(\b\mu)
f\Big( \begin{pmatrix}a &-\b\mmu
\dfr\cr \b\mu& a+\b\tfr\end{pmatrix}\Big)\,\,d\mu\ptf$$
On rappelle que puisque $E/F$ est non ramifiŽ on a choisi $\tau\in\OE^\times$
de sorte que $\dfr$ et $\tfr$ sont des entiers et
$$D(b)= \vert{b(\tau-\overline\tau)}\vert=\vert{b}\vert \ptf$$
On a donc $$f^{\EC}(a+b\tau)=0$$ sauf si $a\in\OF$ et $\b\mu\in\OF$ auquel cas, 
$$f^{\EC}(a+b\tau)=\int_{\vert b\vert\le\vert b\mu\vert\le1}
\vert{b}\vert C(\mu)\kappa(\b\mu)\,d\mu\ptf$$
On rappelle que $d\mu$ donne la mesure 1 ˆ $\OF^\times$, que
$$\vert{b}\vert C(\mu)=\vert\b\vert\com{si} 
\vert\mu\vert=1\com{et} \vert{b}\vert C(\mu)=\bigg(1+\frac{1}{q}\bigg)\vert b\mu\vert\com{si}
\vert\mu\vert>1\ptf$$
De plus $\kappa(\b\mu)=(-1)^m$ si $\vert \b\mu\vert=q^{-m}$.
Donc si $a\in\OF$ et $\vert \b\vert=q^{-n}\le1$
$$f^{\EC}(a+b\tau)=(-1)^nq^{-n}+\sum_{0\le m< n}(-1)^m\bigg(1+\frac{1}{q}\bigg)q^{-m}=1\ptf
$$
\end{proof}
Sous les hypothses de ce lemme mais
avec des mesures de Haar diffŽrentes pour $\G$ et $\Tef$ 
on obtient comme transfert la fonction caractŽristique
du compact maximal de $\Tef(F)$ multipliŽe par  $\vol\big(\Tef(\OF)\big)\bsl\vol\big(\G(\OF)\big)$:
\begin{corollary}\label{tof}
$$f^{\EC}(t)=\frac{\vol\big(\G(\OF)\big)}{\vol\big(\Tef(\OF)\big)}\chi_{E/F}(t)$$
o $\chi_{E/F}$ est la fonction caractŽristique de $\Tef(\OF)$.
\end{corollary}
Nous aurons aussi besoin du lemme suivant:
\begin{lemma} Soit $F$ un corps local, $\EC=\{\Tef,\vef\}$  une donnŽe endoscopique 
de type 2 et $t$ un ŽlŽment de $\Tef(F)$. On a
$$f^{\EC}(t\mun)=f^{\EC}(t)\ptf$$
\end{lemma}

\begin{proof}
On a exhibŽ en \ref{signa} un ŽlŽment $w\in\tG(F)$ qui vŽrifie
$$w t w\mun=t\mun\com{et}\vef\big(\det(w)\big)=\vef(-1)\ptf$$ On en dŽduit que
$$\orb^\kappa(t\mun,f)=\vef(-1)\orb^\kappa(t,f)$$
et comme par ailleurs $$\Delta^\EC(t\mun,t\mun)=\vef(-1)\Delta^\EC(t,t)$$
le lemme en rŽsulte.
\end{proof}

\Section{DŽveloppement en germes
pour $SL(2)$}


Dans toute cette section $F$ dŽsigne un corps local de caractŽristique $p\ge0$.
On note $\UC$ un ensemble de reprŽsentants des classes
de conjugaison d'unipotents rŽguliers.
Cet ensemble est isomorphe au groupe $F^\times/(F^\times)^2$. 
On a donc
$$\UC\simeq Q_F\simeq G(F)\tZ(F)\bsl \tG(F)\ptf$$
Un tel isomorphisme
est bien entendu trs particulier au groupe $SL(2)$.

Soit $T$ un tore anisotrope de $G$. On considre $t\in T(F)$ et $f\in\ctyc \big(G(F)\big)$. 
On cherche un dŽveloppement asymptotique de l'intŽgrale orbitale $\orb(t,f)$ lorsque $t\to1$. 
En caractŽristique $p=0$ on dispose du dŽveloppement en germes de Shalika
indexŽ par l'ensemble des classes de conjugaison unipotentes rationnelles.
En caractŽristique 2 l'ensemble $\UC$ des classes de conjugaison unipotentes rationnelles rŽgulires
est paramŽtrŽ par un groupe compact de cardinal non dŽnombrable.
Le dŽveloppement en germes de Shalika pour $SL(2)$ ne peut pas tre dŽfini \textit{stricto sensu} 
et les sommes sur $\UC$ doivent tre remplacŽes par des intŽgrales si $p=2$.
On dispose d'un dŽveloppement en $\kappa$-germes qui lui a un sens en toute caractŽristique
et qui en caractŽristique $p\ne2$ est Žquivalent, modulo une transformation de Fourier sur un groupe fini,
au dŽveloppement en germes de Shalika.


Lorsque $p\ne 2$ le groupe $\UC$ est fini et
on a le dŽveloppement en germes de Shalika
$$\orb(t,f)= \Gamma_1(t)f(1)+\sum_{\eta\in\UC} \Gamma_\eta(t)\orb(\eta,f)\ptf$$
 On dispose des intŽgrales $\kappa$-orbitales unipotentes
$$\orb^\kappa({\bs\nu},f)\comm{o }
{\bs\nu}=\begin{pmatrix}1&1\cr0&1\cr\end{pmatrix}\ptf$$
c'est ˆ dire l'intŽgrale sur $\tG(F)/U(F)\tZ(F)$ tordue par $\kappa$.
On a par inversion de Fourier
$$\orb(\eta,f)=\frac{1}{\card\widehat\UC}\sum_{\kappa\in\widehat\UC}<\kappa, \eta>\orb^\kappa({\bs\nu},f)$$
o $\kappa$ parcourt le dual $\widehat\UC$ de $\UC$.
Le dŽveloppement de Shalika peut donc se rŽcrire
$$\orb(t,f)= \Gamma_1(t)f(1)+\frac{1}{\card\widehat\UC}
\sum_{\eta\in\UC}\sum_{\kappa\in\widehat\UC}<\kappa, \eta> \Gamma_\eta(t)\orb^\kappa({\bs\nu},f)\ptf$$
Il est alors naturel d'introduire les $\kappa$-germes:
$$\Gamma^\kappa({\bs\nu},t):=\frac{1}{\card\UC}\sum_{\eta\in\UC}<\kappa, \eta> \Gamma_\eta(t)$$
et on a alors un dŽveloppement en $\kappa$-germes
$$\orb(t,f)= \Gamma_1(t)f(1)+\sum_{\kappa\in\widehat\UC} \Gamma^\kappa({\bs\nu},t)\orb^\kappa({\bs\nu},f)\ptf$$

Le dŽveloppement en $\kappa$-germes existe aussi en caractŽristique $p=2$.
On va utiliser la stabilisation de l'intŽgrale orbitale de $t\in T(F)$ pour 
donner un dŽveloppement asymptotique de l'intŽgrale orbitale $\orb(t,f)$
au voisinage de $1$ en toute caractŽristique. On a
$$\orb(t,f)=c_1\sum_{\kappa\in\widehat\UC}\orb^\kappa(t,f)$$
et $\orb^\kappa(t,f)$ est nul sauf si $\kappa=1$ ou bien
si $\kappa=\vef$ est le caractre attachŽ par la thŽorie du corps de classe ˆ
l'extension quadratique $E$ engendrŽe par un plongement de $T(F)$ dans $M(2,F)$.

Pour $\kappa=1$ l'intŽgrale $\orb^1(t,f)$ est l'intŽgrale orbitale pour un 
ŽlŽment d'un tore elliptique
dans $GL(2,F)$ pour lequel on dispose des germes de Shalika:
$$\orb^1(t,f)= \Gamma_1^{\tG}(t)f(1)+ \Gamma_{{\bs\nu}}^{\tG}(t)\orb_{\tG}({\bs\nu},f)\ptf$$
On observera que l'intŽgrale orbitale sur $GL(2,F)$ a un sens bien que $f$ ne soit 
dŽfinie que sur $SL(2,F)$ car les conjuguŽs sous $GL(2)$ des ŽlŽments de $SL(2)$
restent dans ce groupe.

Maintenant, si $\kappa=\vef$,
$$\orb^\kappa(t,f)=\Delta^\EC(t,t')\mun f^{\EC}(t')$$
et
$$f^{\EC}(t')=f^{\EC}(1)$$ si $t'$ est assez voisin de $1$.
Le germe $\Gamma^\kappa(t)$ est donnŽ par l'inverse du facteur de transfert.
On pourra observer qu'il existe une constante $c_2$ dŽpendant du choix des mesures
telle que
$$f^{\EC}(1)=c_2\,\orb^\kappa({\bs\nu},f)\ptf$$

\Chapter{Analyse harmonique et transfert spectral local}

\Section{SŽries principales}\label{sp}

On appelle sŽries principales les reprŽsentations induites paraboliques ˆ partir d'un caractre 
unitaire $\lambda$ du sous-groupe de Levi $\M(F)$ prolongŽ ˆ $\P(F)$ trivialement sur $\U(F)$.
Elles sont rŽalisŽes par la reprŽsentation rŽgulire droite
dans l'espace $V_\lambda$ des fonctions $$\vf(pg)=p^\lambda\delta_\P(p)^{1/2}\vf(g)
\com{o} 
\delta_\P\Big( \begin{pmatrix}a&n\cr0&a\mun\cr\end{pmatrix}\Big)=\vert a\vert^2\ptf$$
On la notera $\pi_\lambda$.
De mme un caractre unitaire $\txi$ de $\tP(F)$:
$$p=\begin{pmatrix}a&n\cr0&b\cr\end{pmatrix} \mapsto
p^\txi=\mu(a)\nu(b)$$
dŽfinit une sŽrie principale $\tpi_\txi$, Žgalement notŽe $\pi(\mu,\nu)$, pour $\tG(F)$.
La notation $\pi(\mu,\nu)$ est celle de \cite{JL} o il est montrŽ que ces reprŽsentations
sont unitaires irrŽductibles, que $\pi(\mu,\nu)$ est Žquivalente ˆ $\pi(\nu,\mu)$
et que ce sont les seules Žquivalences. 
Si $\lambda$ est la restriction de $\txi$ ˆ $\P(F)$ la reprŽsentation 
$\pi_\lambda$ est la restriction de $\tpi_\txi$ ˆ $\G(F)$.
On verra en  \ref{exem} que les $\pi_\lambda$ sont irrŽductibles 
sauf si $\lambda$ est non trivial d'ordre 2.

\Section{L'opŽrateur d'entrelacement}
Les reprŽsentations $\pi_\lambda$ et $\pi_{w\lambda}$  sont Žquivalentes et
un opŽrateur d'entrelacement $\mathbf M(w,\lambda)$ entre $\pi_\lambda$ et $\pi_{w\lambda}$  
est donnŽ pour $\vf\in V_\lambda$
par le prolongement mŽromorphe de l'intŽgrale
$$\mathbf M(w,\lambda)\vf(g)=\int_{U(F)}\vf(wug)\,du$$
o  $w$ est l'ŽlŽment 
non trivial du groupe de Weyl de $\M$.
De fait cette intŽgrale ne converge que dans un c™ne.
En effet, si $$p=\begin{pmatrix}a&n\cr0&a\mun\cr\end{pmatrix} 
\com{et}\vf(pk)=\chi(p)q_F^{(s+1)\HP(p)}\vf(k)=\chi(p)\vert a\vert^{s+1}\vf(k)\com{pour $k\in K$}
$$ o $\chi$ est un caractre unitaire on a
$$\int_{U(F)}\vf(wug)\,du=\int_{F}\vf(\bs k(wu_nk))\chi(m_n)
\frac{dn}{\vert\vert(1,n)\vert\vert^{s+1}}
\com{si}wu_nk=m_nu' \bs k(wu_nk)
$$
et l'intŽgrale est convergente pour $\Re(s)>0$. Une telle intŽgrale admet un prolongement mŽromorphe.
Par exemple, si $F$ est un corps local non archimŽdien, $\chi=1$ et $\vf(k)\equiv1$
on a
$$\int_{U(F)}\vf(wug)\,du=1+(1-q_F\mun)\sum_1^\infty q_F^{-ns}=
\frac{1-q_F^{-s}+(1-q_F\mun)q_F^{-s}}{1-q_F^{-s}}
=\frac{Z(s)}{Z(1+s)}$$
o $Z(s)=1/(1-q_F^{-s})$ est la fonction ZŽta locale. Pour le cas gŽnŽral on renvoie ˆ la littŽrature.

\Section{Induction et restriction}

Soit $\tpi$ une reprŽsentation admissible irrŽductible de $\tG(F)$. On note $X(\tpi)$ le 
groupe des caractres $\bs\chi$ de $\tG(F)=GL(2,F)$ (\cad des homomorphismes dans $\CM^\times$)
tels que $\tpi\otimes\bs\chi\simeq\tpi\ptf$
De tels caractres $\bs\chi$ sont d'ordre 2 puisqu'ils doivent tre triviaux sur $\G(F)\tZ(F)$.

\begin{proposition}\label{fini} 
Le groupe $X(\tpi)$ est fini.
\end{proposition}

\begin{proof}
La finitude de $X(\pi)$ est Žvidente si $F$ est un corps 
local de caractŽristique $p\ne2$ compte tenu de la 
finitude du groupe quotient $$\tG(F)/\G(F)\tZ(F)\simeq Q_F\ptf$$ 
Si $p=2$ o ce quotient est seulement compact et la finitude est Žtablie dans
\cite{Si} (voir aussi \cite{He}).
\end{proof}

Soit $H$ un sous-groupe distinguŽ fermŽ de $\tG(F)$ contenant $\G(F)=SL(2,F)$.
On notera $X_H(\tpi)$ le sous-groupe de $X(\tpi)$ des $\chi$ dont la restriction 
ˆ $H$ est triviale et $n_H(\tpi)$ son cardinal.

\begin{proposition}\label{cliff}
Toute reprŽsentation unitaire irrŽductible
$\pi$ de $H$ apparait dans la restriction ˆ $H$ 
d'une reprŽsentation unitaire irrŽductible $\tpi$ de $\tG(F)$. 
\'Etant donnŽes deux telles reprŽsentations 
$\tpi$ et $\tpi'$ il existe un caractre
$\bs\chi$ du quotient $\tG(F)/H$ tel que $\tpi'=\tpi\otimes\bs\chi$ et leurs restrictions ˆ $H$
ont les mmes composants. Ces composants sont de la forme
$\pi^g$ o $\pi^g(x)=\pi(gxg\mun)$ avec $g\in\tG(F)$.
Par restriction ˆ $H$, la reprŽsentation
$\tpi$ se dŽcompose
en une somme de $n_H(\tpi)$ reprŽsentations deux ˆ deux inŽquivalentes
de $H$ avec multiplicitŽ 1. 
Un caractre $\psi$ du groupe additif $F$ Žtant donnŽ, 
si la reprŽsentation $\tpi$ n'est pas de dimension 1 elle
admet un modle de Whittaker relativement ˆ $\psi$ et
un seul composant irrŽductible de la restriction ˆ $\G(F)$
admet un modle de Whittaker pour $\psi$.
\end{proposition}

\begin{proof} L'analyse des propriŽtŽs de l'induction et de la restriction entre $H$ et $\tG(F)$
se fait suivant une variante de la classique thŽorie de Clifford (voir par exemple \cite{LS}).
La multiplicitŽ 1 comme l'existence et l'unicitŽ du composant $\pi$ 
ayant un modle de Whittaker pour un $\psi$ donnŽ
rŽsulte de l'unicitŽ des modles de Whittaker pour $\tG(F)$, lorsqu'ils existent,
et de l'induction par Žtages. 
Les reprŽsentations de $\tG(F)$ n'admettant pas de modle de Whittaker sont des caractres
et la proposition est triviale dans ce cas.
\end{proof}

\begin{remark}{\rm
L'existence de modles de Whittaker est essentielle
pour affirmer la multiplicitŽ 1. On verra plus loin que pour une reprŽsentation $\tpi'$
du groupe des unitŽs d'une algbre de quaternions $\tG'(F)$
la restriction au sous-groupe $\G'(F)$ des ŽlŽments de norme 1 
peut avoir des composants de multiplicitŽ 2.}
\end{remark}
\begin{definition}\label{lpaq} On appelle $L$-paquet de $\pi$
l'ensemble $L(\pi)$ des classes d'Žquivalence de reprŽsentations de la forme $\pi^g$ 
lorsque $g$ parcourt $\tG(F)$.
\end{definition}

\Subsection{Un premier exemple}

\begin{lemma}\label{exem}
Une reprŽsentation de la sŽrie principale unitaire $\pi_\lambda$ est rŽductible si et seulement si
$\lambda$ est d'ordre 2 non trivial. Dans ce cas elle se dŽcompose en deux reprŽsentations
inŽquivalentes.
\end{lemma}
\begin{proof}
Si $\tpi=\pi(\mu,\nu)$ est une sŽrie principale pour $\tG(F)$
alors $$\tpi\otimes\bs \chi=\pi(\mu\chi,\nu\chi)$$
o  $\bs\chi=\chi\circ\det$.
Donc $\tpi\simeq\tpi\otimes\bs\chi$ avec $\chi$ non trivial impose $\mu\chi=\nu$ et
$\mu=\nu\chi$ de sorte que $$\tpi=\pi(\mu,\mu\chi)=\pi(1,\chi)\otimes\mu\ptf$$
Cette reprŽsentation se restreint en une somme de deux reprŽsentations inŽquivalentes $\pi^+$
et $\pi^-$ de $\G(F)$.
\end{proof}

Soit $\tpi$ une reprŽsentation de $\tG(F)$ et $\chi\in X(\tpi)$.
Un tel caractre $\chi$ est nŽcessairement de la forme $\chi=\vef$
o $E/F$ est une extension quadratique sŽparable, non dŽployŽe si $\chi$ est non trivial,
ce que l'on supposera dŽsormais. En particulier $\tpi$ n'est pas de dimension 1.
On note $\tGFp$ le sous-groupe d'indice 2
de $\tG(F)$ des matrices dont le dŽterminant est dans l'image de norme 
$$\Nef=E^\times\to F^\times\ptf$$
On fixe un caractre additif $\psi$
D'aprs \ref{cliff} la reprŽsentation $\tpi$ est induite d'une reprŽsentation irrŽductible $\tpi^+$ de $\tGFp$
ayant un modle de Whittaker pour $\psi$
et la restriction de $\tpi$ ˆ $\tGFp$ est somme de deux reprŽsentations inŽquivalentes
$\tpi^+\oplus\tpi^-$ ŽchangŽes par conjugaison par un ŽlŽment $s\in\tG(F)$ dont le dŽterminant
n'est pas une norme de $E/F$. Soit $\tT(F)$ un tore dans $\tG(F)$ et pour
$t\in \tT(F)\cap\tGFp$ rŽgulier considŽrons la diffŽrence des caractres distribution:
$$\Xi_{\tpi}(t)=\trace\tpi^+(t)-\trace\tpi^-(t)\ptf$$ 
Cette distribution est reprŽsentŽe par une fonction localement intŽgrable (\cite{Lem})
sur $\tGFp$.
\begin{lemma}\label{diffa} Soit $\tT$ un tore de $\tG(F)$ non isomorphe ˆ $\tTef(F)$.
Alors pour $t\in\tT(F)\cap\tGFp$ rŽgulier $\Xi_{\tpi}(t)=0$.
\end{lemma}
\begin{proof}Soit $t\in\tT(F)\cap\tGFp$ o $\tT$ n'est pas isomorphe ˆ $\tTef$. Dans ce cas
il existe $s\in\tT(F)$ tel que $\vef(s)=-1$. Mais alors 
$$\trace\tpi^+(t)=\trace\tpi^+(sts\mun)=\trace\tpi^-(t)\ptf$$ 
\end{proof}

Nous allons maintenant construire d'autres couples de reprŽsentations $(\tpi^+,\tpi^-)$.

\Section{ReprŽsentation de Weil et endoscopie spectrale}

 Soit $E$ un corps extension quadratique 
sŽparable de $F$ (Žventuellement dŽployŽe).
On considre $\VF\in\mathcal S(E)$ l'espace des fonctions de Schwartz-Bruhat sur $E$. On pose
$$\diag(a,b)=\begin{pmatrix}a &0\cr 0&b\end{pmatrix}\com{,}\nil(u)=
\begin{pmatrix}1 &u\cr 0&1\end{pmatrix}
\com{et}{\bs w}=\begin{pmatrix}0 &1\cr -1&0\end{pmatrix}\ptf$$
$$\big(r_\psi(\diag(a,a\mun)\big)\Phi)(e)=
\vef(a)\vert a\vert_F \VF(ae)$$
$$\big(r_\psi(\nil(u)\big)\Phi)(e)=
\psi(u\,e\overline e)\,\VF(e)$$
$$\big(r_\psi({\bs w})\Phi\big)(e)
=\lambda(E/F,\psi)\,\widehat\VF(\overline e)$$
o
$$\widehat\VF(e):\int_E \VF(\ep)\psi(e\ep)\,d\ep
\ptf$$
On a ainsi dŽfini des opŽrateurs unitaires dans un sous-espace dense de l'espace de Hilbert $L^2(E)$
(pour une mesure de Haar additive sur $E$).
On sait que $\G(F)$ est engendrŽ par ces trois familles d'ŽlŽments et on vŽrifie 
(cf. \S1 de \cite{JL} et \cite{W}) que
les opŽrateurs associŽs vŽrifient les relations de commutation. Ils engendrent 
une reprŽsentation unitaire $r_\psi$ de $\G(F)$ appelŽe reprŽsentation de Weil (associŽe ˆ l'extension
quadratique sŽparable $E/F$ et au caractre additif $\psi$).
On peut dŽcomposer $r_\psi$ en une somme directe hilbertienne
suivant les caractres $\th$ de $E^1$
le sous-groupe compact des ŽlŽments de norme 1 de $E^\times$:
$$ \mathcal S(E)=\bigoplus_\th \mathcal S(E,\th)$$
 o $\mathcal S(E,\th)$ est le sous-espace des $\VF\in\mathcal S(E)$ satisfaisant
$$\VF(et)=\th(t)\mun\VF(e)\com{pour tout}t\in E^1 \ptf$$
La restriction de $r_\psi$ au sous-espace de Hilbert engendrŽ par $\mathcal S(E,\th)$
est une reprŽsentation qui sera notŽe $\pi(\th,\psi)$.

On notera $\Ftp=\Nef E^\times$ le sous-groupe des
$x\in F^\times$ qui sont des normes \cad tels que $x=e\overline e$ pour un $e\in E^\times$.
Si $\psi'(x):=\psi(ax)$ avec $a=\epsilon\overline\epsilon\in \Ftp$ l'application
 $\Phi\mapsto\Phi_\epsilon$, dŽfinie par $\Phi_\epsilon(e)=\Phi(e\epsilon)$,
entrelace les reprŽsentations $\pi(\th,\psi)$ et $\pi(\th,\psi')$ qui sont donc Žquivalentes.
L'application $\Phi\mapsto \widetilde\Phi$, dŽfinie par
$\widetilde\Phi(e)=\Phi(\overline e)$, entrelace $\pi(\th,\psi)$ et $\pi(\th\mun,\psi)$ qui sont donc Žquivalentes.

Nous noterons $\pi^+$ la classe d'Žquivalence de la reprŽsentation $\pi(\th,\psi)$ et
$\pi^-$ celle de $\pi(\th,\psi')$ o $\psi'(x)=\psi(a x)$
avec $a\notin \Ftp$. 

Rappelons une construction introduite dans le \S4 de \cite{JL}
et poursuivie dans \cite{BC}.
Soit $\tth$ un caractre unitaire de $E^\times$ qui prolonge $\th$;
on associe ˆ une $\VF\in\mathcal S(E,\th)$ 
une fonction $\phi$ sur $\Ftp$ en posant
$$\phi(f,x)=
\tth(e)\vert e\vert_E^{1/2}\VF(e)\qquad\hbox{pour $x=e\overline e\in \Ftp$}\ptf$$
La reprŽsentation $\pi^+$ se rŽalise dans l'espace $L^2(\Ftp)$, 
pour une mesure de Haar notŽe $d^\times x$ sur ce groupe,
au moyen des opŽrateurs suivants:
$$\Big(\pi^+\big(\diag(\alpha,\alpha\mun)\phi\big)\Big)(x)
=\big(\vef\tth(\alpha)\big)\mun \phi(\alpha^2x)\qquad\hbox{pour $\alpha\in F^\times$}$$
$$\big(\pi^+(\nil(u))\phi\big)(x)=\psi(ux)\,\phi(f,x)$$
$$\big(\pi^+(\bs w)\phi\big)(x)=\int_{\Ftp}J(x,y)\phi(y)d^\times y$$
o le noyau $J(x,y)$ est dŽfini comme suit:
on doit avoir
$$\big(\pi^+(\bs w)\phi\big)(x)=\lambda(E/F,\psi)\tth(e)\vert e\vert_E^{1/2}\widehat\Phi(\overline e)$$
avec
$$\widehat\Phi(\overline e)=\int_{E^\times}\psi(\overline e\ep)\Phi(\ep)\vert \ep\vert_E\,d^\times \ep
=\int_{E^\times}\psi(\overline e\ep)\tth(\ep)\mun\vert \ep\vert_E^{1/2}\phi(y)\,d^\times \ep\ptf
$$
En posant $y=\ep\overline\ep$, $\eta=e\ep$ et $t={\ep\overline\ep\mun}$ 
de sorte que $\overline e\ep=t\overline\eta$, on a
$$J(x,y)=\lambda(E/F,\psi)\,\tth(x)\vert x \vert_F
\int_{\vert t\vert=1}\psi_E(t\overline\eta)\tth(t\overline\eta)\mun\Big\vert \frac{\ep}{e}\Big\vert^{1/2}
\,{\dt }\leqno{(\star)}$$
o les mesures de Haar sur $E^\times$, $\Ftp$ et $T(F)$ sont normalisŽes de faon compatible ˆ la suite exacte
$$1\to T(F)\to E^\times\to \Ftp \to1$$
la flche $E^\times\to\Ftp$ Žtant la norme $\ep\mapsto \ep\overline\ep$.
On observera qu'un composant de $\pi^+$ admet un modle de Whittaker 
pour $\psi$: c'est l'espace des fonctions
$$W_\phi(g)=\big(\pi^+(g)\phi\big)(1)\ptf$$

Si $\pi$ est une reprŽsentation unitaire irrŽductible de 
$\G(F)$ nous noterons $\chi_\pi$ (et aussi ``$\trace \pi$'') le caractre de $\pi$ (cf. \ref{cara}). 
Nous aurons besoin du lemme 7.19 de \cite{BC}.%
\footnote{On observera que \cite{LL} renvoie ˆ des ŽnoncŽs du chapitre 5 des notes de l'IAS
sur le Changement de Base, devenu le chapitre 7 du livre \cite{BC}.}
Nous en reproduisons la preuve ˆ ceci prs que la partie la plus dŽlicate 
de l'argument est dŽsormais inutile puis qu'on sait grace ˆ \cite{Lem} que 
le caractre est donnŽ par une fonction localement intŽgrable.

\begin{proposition}\label{diffb} Si $g$ est conjuguŽ de $\gamma=\alpha+\beta\bs\tau\in E^1$
o $E$ est plongŽ dans $M(2,F)$ comme dans \ref{ext}
et si $g$ appatient ˆ la grosse cellule (\cad si $g\ne\P(F)$) alors
$$\trace\tpi^+(g)-\trace\tpi^-(g)=\lambda(E/F,\psi)\vef(-1)
\vef\Big(\frac{\gamma-\overline\gamma}{\tau-\overline\tau}\Big)
\frac{\big(\th(\gamma)+\th(\gamma\mun)\big)}{\vert\gamma-\gamma\mun\vert}\ptf$$
\end{proposition}

\begin{proof} 
Si
$$h=\begin{pmatrix}1 &u\cr 0&1\end{pmatrix}
\begin{pmatrix}0 &1\cr -1&0\end{pmatrix}
\begin{pmatrix}1 &v\cr 0&1\end{pmatrix}=
\begin{pmatrix}-u &1-uv\cr -1&
-v\end{pmatrix}$$
on a
$$\Big(\pi^+(h)\phi\Big)(x)=
\int_{\Ftp}J(x,y)\psi(ux+vy)\phi(y)d^\times y
$$
Un ŽlŽment $g\in\G(F)$ dans la grosse cellule s'Žcrit
$$g=\begin{pmatrix}a &b\cr c&
d\end{pmatrix}=\begin{pmatrix}\alpha &0\cr 0&
\alpha\mun \end{pmatrix}h
=\begin{pmatrix}-\alpha u &\alpha(1-uv)\cr -\alpha\mun&
-\alpha\mun v\end{pmatrix}$$
avec $ad-bc=1$ et $c\ne0$. Donc
$$u={a}{c}\com{,}v=\frac{d}{c}\com{et}\alpha=\frac{-1}{c}$$
et
$$\Big(\pi^+(g)\phi\Big)(x)=
\vef(\alpha)\tth(\alpha)\mun \Big(\pi^+(h)\phi\Big)(\alpha^2x)
$$
est reprŽsentŽ par un noyau intŽgral:
$$(\pi^+(g)\phi)(x)=\int_{\Ftp}(\vef\tth)({-c})
J(\frac{x}{c^2},y)\psi\Big(\frac{ax+dy}{c}\Big)\phi(y)d^\times y\ptf
$$
La reprŽsentation virtuelle $\pi^+-\pi^-$ est donc reprŽsentŽe par le noyau intŽgral:
$$\Big(\big(\pi^+(g)-\pi^-(g)\big)\phi\Big)(x)=\int_{\F^\times}(\vef\tth)({-c})
J(\frac{x}{c^2},y)\psi\Big(\frac{ax+dy}{c}\Big)\phi(y)d^\times y\ptf$$
La diffŽrence des caractres $\chi_{\pi^+}-\chi_{\pi^-}$, calculŽe au sens des distributions sur $\G(F)$,
est la valeur principale de l'intŽgrale sur la diagonale de ce noyau
$$\chi_{\pi^+}(g)-\chi_{\pi^-}(g)=vp\int_{\F^\times}(\vef\tth)({-c})
J(\frac{x}{c^2},x)\psi\Big(\frac{x\tfr}{c}\Big)d^\times x\ptf$$
o $\tfr$ est la trace de $g$, soit encore
$$\chi_{\pi^+}(g)-\chi_{\pi^-}(g)=vp\int_{\F^\times}(\vef\tth)({-c})
J(\frac{x}{c},xc)\psi({x\tfr})d^\times x\ptf$$
ce qui, compte tenu de la dŽfinition de $J$ (voir $(\star)$ ci-dessus), est encore Žgal ˆ
$$\lambda(E/F,\psi)\vef({-c})vp\int_{\F}
\int_{\vert t\vert=1}\,\tth(-x)\tth(t\overline\eta)\mun\psi_E(t\overline\eta)
\psi({x\tfr})\,\,dx\,\,{\dt }\ptf$$
Comme $x=y$ on a $\eta\overline\eta=x^2$ et
donc $\eta=t_1x$ o $\vert t_1\vert=1$. L'expression ci-dessus peut alors s'Žcrire
$$\lambda(E/F,\psi)\vef({-c})vp\int_{\F}
\int_{\vert t\vert=1}\,\tth(t)\mun\psi_E(-xt)
\psi({x\tfr})\,\,dx\,\,{\dt }$$
soit encore, en rappelant que $\tfr=\trace g$,
$$\lambda(E/F,\psi)\vef({-c})vp\int_{\F}
\int_{\vert t\vert=1}\,\tth(t)\mun\psi\Big({x\big(\trace g -\trace\!_{E/F} t\big)}\Big)dx\,\,{\dt }\ptf$$
On rappelle qu'au sens des distributions
$$\int_F \psi(xy)dy=\delta(x)$$
o $\delta$ est la mesure de Dirac ˆ l'origine
et que pour $t\in E^1$ les diffŽrentielles vŽrifient:
$$\partial(\trace t)=\partial(t+t\mun)=(t-t\mun)\frac{\partial t}{t}\ptf$$
On obtient
$$\trace\pi^+(t)-\trace\pi^-(t)=\lambda(E/F,\psi)\vef(-c)
\frac{\big(\th(\gamma)+\th(\gamma\mun)\big)}{\vert\gamma-\gamma\mun\vert}$$
Pour conclure on observe que si $$g=\begin{pmatrix}a &b\cr c&
d\end{pmatrix}$$ est conjuguŽ de $\gamma=\alpha+\beta\bs\tau\in E^1$
alors $\vef(c)=\vef(\beta)$ et donc
$$\vef(-c)=\vef(-1)\vef\Big(\frac{\gamma-\overline\gamma}{\tau-\overline\tau}\Big)$$
\end{proof}

On a notŽ $\tGFp$ le sous-groupe
de $\tG(F)$ des matrices dont le dŽterminant est dans l'image de la norme $\Nef$.
En fait $\pi^+:=\pi(\th,\psi)$ se prolonge ˆ $\tGFp$ 
en une reprŽsentation $\tpi^+:=\pi(\tth,\psi)$ au moyen des opŽrateurs
$$\big(\tpi^+(\diag(\alpha,\beta))\phi\big)(x)
= (\vef\tth)(\beta)\phi(\alpha\beta\mun x)\qquad\hbox{pour $\alpha\beta\in \Ftp$}\ptf$$
Le sous-groupe $\tGFp$ est d'indice 2 dans $\tG(F)$ et on dŽfinit 
une reprŽsentation $\tpi_{E,\tth}$ de $\tG(F)$ par induction:
$$\tpi_{E,\tth}:=\Ind_{\tGFp}^{\tG(F)}\pi(\tth,\psi)\ptf$$
Il rŽsulte de \ref{cliff} que $\tpi_{E,\tth}$ dŽcompose par restriction ˆ $\tGFp$
en deux reprŽsentations inŽquivalentes irrŽductibles:
$$\tpi_{E,\tth}\Big\vert_{\tGFp}=\pi(\tth,\psi)\oplus\pi(\tth,\psi_a)$$
o $a\notin \Ftp$. Les reprŽsentations $\pi(\tth,\psi)$ sont irrŽductibles en tant que reprŽsentations
de $\tGFp$ mais on verra en \ref{pitheta} que ce n'est plus toujours le cas par restriction ˆ $\G(F)$.
Une variante de \ref{diffb} montre que 
$$\Xi_{\tth}(g):=\trace\pi(\tth,\psi)(g)-\trace\pi(\tth,\psi')(g)$$ 
o $\psi'(x)=\psi(ax)$
avec $a\notin \Ftp$ admet l'expression suivante:
$$\Xi_{\tth}(g)=\lambda(E/F,\psi)\vef(-1)
\vef\Big(\frac{\gamma-\overline\gamma}{\tau-\overline\tau}\Big)
\frac{\big(\tth(\gamma)+\tth(\overline\gamma)\big)}{\vert\gamma-\overline\gamma\vert}
$$
si $g$ est conjuguŽ dans $\tGFp$ ˆ la matrice $\gamma=\alpha+\beta\bs\tau$.

Les reprŽsentations $\pi(\tth,\psi)$ et $\pi(\tth,\psi')$ sont Žquivalentes
si $\psi'=\psi_a$ o $\psi_a(x):=\psi(ax)$ et $a\in \Ftp=\Nef E^\times$.

On observera que dans \cite{BC} et \cite{LL} les formules pour $\Xi_\tth$ diffŽrent
de la n™tre par le facteur $\vef(-1)$. Cela rŽsulte d'un choix different
pour l'identification de $\T$ avec $\Tef$ (voir page 137 ligne $-4$ de \cite{BC}). 
Notre choix du facteur de transfert est aussi diffŽrent.

\begin{corollary}\label{iden}
$$\Delta^\EC(t,t)\,\Xi_{\tth}(t)=\vef(-1)\big(\tth(t)+\tth(\overline t)\big)
\ptf$$
\end{corollary}

\begin{proof} Compte tenu de \ref{psideux} et \ref{FTr} 
la formule souhaitŽe est une consŽquence immŽdiate de \ref{diffb}.
\end{proof}

\begin{proposition}\label{tpitheta}
Les reprŽsentations $\tpi_{E,\tth}$ sont unitaires 
irrŽductibles et indŽpendantes du choix de $\psi$. 
Les reprŽsentations $\tpi_{E,\tth}$ et $\tpi_{E,\tth'}$
sont inŽquivalentes sauf si
$\tth'=\tth\circ\sigma$ avec $\sigma\in\Gal(E/F)$.
C'est une sŽrie principale si $\th= 1$
\cad si $\tth$ est le composŽ d'un caractre $\mu$ de $F^\times$ avec la norme:
$\tth=\mu\circ N_{E/F}$ auquel cas
$$\tpi_{E,\tth}=\pi(\mu,\mu\vef)\ptf$$
Les $\pi_{E,\th}$ sont cuspidales si $\th\ne 1$.
De plus si $\chi$ est un caractre (unitaire) de $F^\times$,
vu comme un caractre de $\tG(F)$ par composition avec le dŽterminant,
on a:
$$\tpi_{E,\tth}\otimes\bs\chi\simeq\tpi_{E,\tth.\xi}$$
o $\xi$ est un caractre de $E^\times$ trivial sur $E^1$: 
$\xi(x)=\chi(x\overline x)$.
\end{proposition}
\begin{proof}
Les diverses assertions rŽsultent des thŽormes 4.6 et 4.7 de \cite{JL}.
\end{proof}

\begin{proposition} \label{xpi} Soit $\tpi=\tpi_{E,\tth}$.
Le cardinal du groupe $X(\tpi)$: il est d'ordre 1, 2 ou 4.
S'il est d'ordre 4 alors $X(\tpi)\simeq\ZM/2\times\ZM/2$.
\end{proposition}
\begin{proof} 
On sait que la restriction de $\tpi$ ˆ $\G(F)$ est irrŽductible 
si $X(\tpi)$ est un singleton. Sinon, il existe un caractre $\chi$ non trivial d'ordre 2 tel que
$$\tpi\otimes\bs\chi\simeq\tpi\ptf$$
Soit $E/F$ l'extension quadratique
sŽparable telle que $\chi=\vef$. Il rŽsulte de \ref{otim} que
$\tpi=\tpi_{E,\tth}$ pour un caractre $\tth$ de $E^\times$. 
Supposons d'abord $\th=1$ \cad que $\tth$ est le composŽ d'un caractre
de $F^\times$ et de la norme pour $E/F$. Dans ce cas
$\tpi_{E,\tth}$ est une sŽrie principale $\pi(\tth,\vef\tth)$ (\cad induite unitaire par le caractre
du tore des matrices diagonales dŽfini par $(\tth,\vef\tth)$)
et il en rŽsulte que $\vef$ est le seul caractre $\chi$
non trivial tel que 
$$\tpi\otimes\bs\chi\simeq\tpi\ptf$$
Nous supposerons dŽsormais $\th\ne1$.
On en donnera deux dŽmonstrations indŽpendantes.

\pni -- Premire dŽmonstration.
D'aprs \ref{tpitheta}, pour $\chi\in X(\tpi)$ on aura 
$$\tpi_{E,\tth}\otimes\bs\chi\simeq\tpi_{E,\tth.\xi}\simeq\tpi_{E,\tth}
\com{et donc}\tth(x)\chi(x\overline x)=\tth\circ\sigma(x)$$
avec $\sigma\in\Gal(E/F)$ et o $\xi(x)=\chi(x\overline x)$.
Supposons $\chi\notin\{1,\vef\}$; cela impose
$$\tth(x)\chi(x\overline x)=\tth(\overline x)\com{et donc}
\chi(x\overline x)=\tth(\overline x/x)=\xi(x)$$
o $\xi$ est un caractre non trivial d'ordre 2 de $E^\times$.
Puisque l'application
$x\mapsto x/\overline x$ est une surjection de $E^\times$ sur le tore
$\Tef(F)\simeq E^1$ ceci impose $$\tth(\overline x/x)=\th(\overline x/x)=\xi(x)$$
et donc $\th^2=1$.
Le noyau $N$ de $\xi$ est l'image par la norme du groupe 
multiplicatif $L^\times$ pour une extension quadratique $L/E$. Tout ŽlŽment de $X(\tpi_{E,\tth})$ est trivial sur 
l'image de $N$ par la norme de $E/F$ qui est un sous-groupe d'ordre 4 de $F^\times$.
Donc 
$X(\tpi_{E,\tth})$ a 4 ŽlŽments et comme chaque ŽlŽment est d'ordre 2 le groupe 
$X(\tpi_{E,\tth})$ est isomorphe ˆ $\ZM/2\times\ZM/2$.
L'extension $L/F$ est bi-quadratique; 
son groupe de Galois est le dual de $X(\tpi_{E,\tth})$.
\pni -- Seconde dŽmonstration. 
Soient $\th\ne1$, $\pi^+=\pi(\th,\psi)$ et $\pi^-=\pi(\th,\psi')$
o $\psi'(x)=\psi(ax)$ pour $a\notin \Ftp=\Nef E^\times$.
Avec la mesure de Haar normalisŽe sur le groupe compact $\Tef(F)$ on a
d'aprs \ref{diffb}, 
$$\int_{\Tef(F)}\vert\trace{\pi^+}(t)-\trace{\pi^-}(t)\vert^2\Delta_\T(t)^2\,dt=
\int_{\Tef(F)}\vert\th(t)+\th(t\mun)\vert^2\,dt$$
et donc
$$\int_{\Tef(F)}\vert\trace{\pi^+}(t)-\trace{\pi^-}(t)\vert^2\Delta_\T(t)^2\,dt=
\begin{cases} 2\com{si $\th^2\ne1$}\cr
4\com{si $\th^2=1$}\end{cases}
$$
Maintenant les relations d'orthogonalitŽ (cf. \ref{weyl}) montrent que $\pi^+$ et $\pi^-$
sont irrŽductibles si $\th^2\ne1$ et  que si $\th^2=1$ (mais $\th\ne1$) chacune
se dŽcompose en somme de deux reprŽsentations donnant au total quatre reprŽsentations
irrŽductibles inŽquivalentes. On conclut comme ci-dessus.
\end{proof}

En rŽsumŽ on a les assertions suivantes (cf. \cite{ST} lorsque la caractŽristique rŽsiduelle est 
diffŽrente de 2).

\begin{proposition}\label{pitheta} 
\pni
\begin{itemize}
\item[(i) --] Les reprŽsentations $\pi(\th,\psi)$, restriction ˆ $\G(F)$ des reprŽsentations $\tpi(\tth,\psi)$
de $\tGFp$, sont irrŽductibles sauf si $\th^2=1$ et $\th\ne1$ auquel cas elles se dŽcomposent en
deux reprŽsentations inŽquivalentes.
\item [(ii) --] Les reprŽsentations $\pi(\th,\psi)$ et $\pi(\th',\psi)$
sont inŽquivalentes sauf si $\th'=\th^{\pm1}$.
\item [(iii) --] Elles sont cuspidales si $\th\ne1$. Si $\th=1$ ce sont des ``limites de sŽrie discrtes''.
\end{itemize}
\end{proposition}

\begin{proof} La dŽfinition de $\tpi_{E,\tth}$ par induction de $\tGFp$ ˆ $\tG(F)$ montre que
$$\tpi_{E,\tth}\vert_{\G(F)}=\pi(\th,\psi)\oplus\pi(\th,\psi_a)$$
avec $a\notin \Ftp$. 
L'action du groupe de Galois de $E/F$: $$\sigma:e\mapsto\overline e$$
induit un isomorphisme entre $\pi(\th,\psi)$ et $\pi(\th\mun,\psi)$.
Ceci montre une partie de $(i)$ et $(ii)$ et on invoque \ref{xpi}.
L'assertion $(iii)$ se dŽduit de \ref{tpitheta}.
\end{proof}

Soit $\chi$ un caractre de $F^\times/(F^\times)^2$. 
Par abus de notation on notera encore
$\chi$ le caractre de $\tG(F)$ obtenu par composition avec le dŽterminant.
Rappelons maintenant le lemme 7.17 de \cite{BC}. 

\begin{proposition} \label{otim}
Si $\chi$ est non trivial,
une reprŽsentation $\tpi$ de $\tG(F)$ qui vŽrifie
$\tpi\otimes\bs\chi\simeq\tpi$
est de la forme $\tpi=\tpi_{E,\tth}$ o $\tth$ est un caractre
de $E^\times$ o $E$ est l'extension quadratique sŽparable de $F$
attachŽe ˆ $\chi$ par la thŽorie du corps de classe.
\end{proposition}

\begin{proof} 
Si $\tpi$ est une sŽrie principale et donc, dans les notations de \cite{JL}, $\tpi=\pi(\mu,\nu)$
on a dŽjˆ observŽ que $\pi(\mu,\nu)\otimes\bs\chi\simeq\pi(\mu\chi,\nu\chi)$ 
et comme par ailleurs on suppose que $\pi(\mu,\nu)\otimes\bs\chi\simeq\pi(\mu,\nu)$
ceci impose $\nu=\mu\chi$ si $\chi$ est non trivial:
$$\tpi=\pi(\mu,\mu\chi)=\pi(1,\chi)\otimes\mu=\tpi_{E,\tth}$$
o $\tth=\mu\circ N_{E/F}$ (cf. \cite[Theorem 4.6]{JL}).
Le cas des reprŽsentations spŽciales est exclu.
Supposons maintenant que $\tpi$ est cuspidale. 
Une reprŽsentation $\tpi$ de $\tG(F)$ qui vŽrifie
$\tpi\otimes\bs\chi\simeq\tpi$
est nŽcessairement induite d'une reprŽsentation irrŽductible $\tpi^+$ de $\tGFp$
et la restriction de $\tpi$ ˆ $\tGFp$ est somme de deux reprŽsentations inŽquivalentes
$\tpi^+\oplus\tpi^-$ ŽchangŽes par conjugaison par un ŽlŽment dont le dŽterminant
n'est pas une norme de $E/F$. Soit $T(F)$ un tore dans $\G(F)$ et pour
$t\in \T(F)$ rŽgulier
considŽrons la diffŽrence des caractres:
$$\Xi_{\tpi}(t)=\trace\tpi^+(t)-\trace\tpi^-(t)\ptf$$ 
D'aprs \ref{diffa}
cette diffŽrence est nulle si $\T$ n'est pas isomorphe ˆ $\Tef$.
Par ailleurs, on a observŽ en \ref{ext} que l'ŽlŽment de $\tG(F)$
notŽ $\bs w_E$ Žtait tel que $$\bs w_E \gamma\bs w_E\mun=\overline\gamma\comm{et}
\vef\big(\det(\bs w_E)\big)=\vef(-1)\ptf$$
Il en rŽsulte que $$\Xi_{\tpi}(\bs w_E g\bs w_E\mun)=\vef(-1)\Xi_{\tpi}(g)$$
et donc $$\Xi_{\tpi}(t)=\vef(-1)\Xi_{\tpi}(\overline t)$$ pour $t\in\tT(F)$.
De mme $$\Xi_{\tth}(t)=\vef(-1)\Xi_{\tth}(\overline t)\ptf$$
On se restreint maintenant ˆ des reprŽsentations de $\tG(F)$ qui 
admettent un mme caractre central.
On observe que les $\Xi_{\tth}$ forment une base orthogonale
de l'espace de Hilbert des fonctions vŽrifiant ces relations avec caractre central donnŽ. 
Compte tenu des relations d'orthogonalitŽ pour les caractres (\cite[Chapter 15]{JL},
\ref{cara} et \ref{weyl})
il en rŽsulte que $\Xi_{\tpi}$ ne peut pas tre orthogonal ˆ tous les $\Xi_{\tth}$ 
et il y a nŽcessairement un $\tth$ tel que $$\Xi_{\tpi}=\pm\Xi_{\tth}\ptf$$
\end{proof}

\Section{Transfert spectral local}\label{tsp}

Soit $F$ un corps local et $f\in\ctyc\big(\G(F)\big)$ on souhaite dŽfinir un transfert spectral 
$\pi\mapsto \pi^\EC$ dual du transfert gŽomŽtrique. Il doit vŽrifier
$$\trace L(\pi)(f^\EC)=\trace \pi^\EC(f)$$
o $L(\pi)$ est la somme des reprŽsentations dans
le $L$-paquet contenant $\pi$. On a vu que c'est l'ensemble des 
classes de reprŽsentation $\pi^g$ o $g$ parcourt $\tG(F)$. 
Cet ensemble est fini de cardinal 1,2 ou 4 avec toujours multiplicitŽ 1.
Nous verrons que pour les unitŽs des algbres de quaternions les $L$-paquets sont de cardinal 1 ou 2
mais certains $L$-paquets sont des singletons mais avec multiplicitŽ 2.

Lorsque $\EC$ est la donnŽe endoscopique principale \cad $\EC=(SL(2),1)$ le transfert
gŽomŽtrique est trivial: $f^\EC=f$ et le transfert spectral vŽrifie
$$ \pi^\EC=L(\pi)\ptf$$ 

Si $\EC$ est une donnŽe endoscopique attachŽe ˆ une extension
quadratique (Žventuellement dŽployŽe) $E/F$ et $\th$ 
un caractre unitaire du groupe $T_\EC(F)$
on souhaite montrer l'existence des reprŽsentations virtuelles $\th^\EC$ (\cad combinaisons
linŽaires formelles de reprŽsentations irrŽductibles) de $\G(F)$ telles que
$$\int_{\Tef(F)} f^\EC(t)\th(t)\,dt=\theta(f^\EC)=\trace\theta(f^\EC)=\trace\theta^\EC(f)$$

\begin{proposition}\label{spectrans}\pni (1)
Si $E=F\oplus F$ et si $\th$ est un caractre du tore diagonal, on a
$$\int_{\Tef(F)} f^\EC(t)\th(t)\,dt=\trace\pi_{E,\th}(f)$$
o $\pi_{E,\th}$ est la sŽrie principale induite-parabolique (normalisŽe)
de $\th$ vu comme un caractre du sous-groupe 
parabolique $P$. 
\pni (2)
Si $E$ est un corps, pour notre choix du plongement de $E^1$ dans $\G(F)$ on a
$$\int_{\Tef(F)} f^\EC(t)\th(t)\,dt=\trace\pi_{E,\th}^+(f)-\trace\pi_{E,\th}^-(f)$$
o:
$$\pi_{E,\th}^+=\pi(\th,\psi)\comm{et} \pi_{E,\th}^-=\pi(\th,\psi')$$
si $\psi'(x)=\psi(ax)$ pour un $a\notin\Nef E^\times$.

\end{proposition}
\begin{proof} On observe qu'avec notre choix de la constante $c$ du facteur de transfert
et d'aprs \ref{psideux}
$$\Delta^\EC(t,t)^2=\vef(-1)\Delta_\T(t)^2\ptf$$
Si $E=F\oplus F$ l'assertion (1) est une consŽquence immŽdiate de \ref{deploy}
et de la formule pour la trace des sŽries principales.
Lorsque $E$ est un corps et compte tenu de \ref{iden} et on a
$$\Delta_\T(t)^2\,\Xi_{\th}(t)=
\Delta^\EC(t,t)\big(\th(t)+\th(\overline t)\big)
$$
o
$$\Xi_{\th}(t)=\trace\pi_{E,\th}^+(t)-\trace\pi_{E,\th}^-(t)\ptf$$
Maintenant au vu de  \ref{trans} et \ref{fond}
on a
$$\int_{\Tef(F)} f^\EC(t)\th(t)\,dt=\frac{1}{2}\int_{\Tef(F)} f^\EC(t)\big(\th(t)+\th(t\mun)\big)\,dt 
\mskip 130mu$$
$$\mskip 250mu=\frac{1}{2}\int_{\Tef(F)}\Delta_\T(t)^2\,\,\Xi_\th(t)\orb^\kappa(f,t)\,dt$$
On rappelle que
$$\orb^\kappa(t,f)=\orb(t,f)-\orb(t',f)$$
o  $t'$ est stablement conjuguŽ mais non conjuguŽ ˆ $t$.
On distingue deux cas: si il existe une seule classe de $\G(F)$-conjugaison
de tores dans $\G(F)$ isomorphe ˆ $\Tef$ alors $t$ et $t\mun$
sont stablement conjuguŽs mais non conjuguŽs dans $\G(F)$ et $w_T=1$.
Dans l'autre cas $t$ et $t\mun$ sont conjuguŽs dans $\G(F)$ mais il y a deux classes de 
$\G(F)$-conjugaisons de tores isomorphes ˆ $\Tef$ et $w_T=2$. 
Dans tous les cas  la formule d'intŽgration de Weyl \ref{Weyl} montre que:
$$\int_{\Tef(F)} f^\EC(t)\th(t)\,dt=\int_{\G(F)}\!\!\Xi_\th(x)f(x)\,dx
=\trace\pi_{E,\th}^+(f)-\trace\pi_{E,\th}^-(f)\ptf$$
 \end{proof}
 
 \begin{remark}\label{choix}
 {\rm La correspondance $\theta\mapsto\theta^\EC$ suppose divers choix: 
celui du caractre additif $\psi$ et l'identification entre
$T$ et $\Tef$ et du facteur de transfert. Comme dŽjˆ observŽ juste aprs \ref{iden}
ce n'est pas la mme identification pour les tores qui est utilisŽe dans \cite{LL}. 
Il semble que l'identification choisie ici soit plus naturelle. Elle justifie alors
le choix du facteur de transfert.}
\end{remark}

On pourra observer que la proposition \ref{spectrans} fournit une nouvelle
dŽmonstration du Lemme Fondamental \ref{fond}. En effet, si l'extension $E/F$ 
est non ramifiŽe (resp. dŽployŽe), l'espace de la reprŽsentation
$\pi_{E,\th}^+=\pi(\th,\psi)$ (resp $\pi_{E,\th}$) est non-ramifiŽe 
(i.e. possde un vecteur invariant sous un sous)groupe hyperspŽcial)
si et seulement $\theta$ est non ramifiŽ.

\Section{Action de l'opŽrateur d'entrelacement}

Soit $\pi_\xi$ la reprŽsentation de la sŽrie principale induite parabolique
(normalisŽe) par le caractre $\xi$ de $F^\times$. On note ${\mathbf R}(\xi)$ l'opŽrateur 
d'entrelacement normalisŽ entre $\pi_\xi$ et $\pi_{\xi\mun}$. 

On suppose maintenant que $\xi=\vef$.
Donc, $\xi$ est d'ordre 2 et l'opŽrateur ${\mathbf R}(\xi)$ est un endomorphisme
dont le carrŽ est l'identitŽ.
Si $E=F\oplus F$ alors $\xi=1$ et ${\mathbf R}(\xi)$ est l'identitŽ.
Si $E$ est un corps extension quadratique de $F$;
 la reprŽsentation $\pi_\xi$
se dŽcompose en deux sous-reprŽsentations irrŽductibles:
$$\pi_\xi=\pi({\bs 1}_E,\psi)\oplus\pi({\bs 1}_E,\psi')$$
o ${\bs 1}_E$ est le caractre trivial du groupe compact $E^1$.
On pose $\pi^+_E=\pi({\bs 1}_E,\psi)$ et $\pi^-_E=\pi({\bs 1}_E,\psi')$.
On observera que $\pi^+_E$ admet un modle de Whittaker pour $\psi$.
Le lemme ci-aprs est \cite[Lemma 3.6]{LL}.
\begin{lemma}\label{entrelace}
L'opŽrateur ${\mathbf R}(\vef)$ vaut $\pm1$ sur $\pi^\pm_E$
\end{lemma}
\begin{proof} 
Comme $\vef$ est d'ordre 2 et non trivial, l'endomorphisme
${\mathbf R}(\vef)$ est non trivial et de carrŽ l'identitŽ; il vaut nŽcessairement $+1$
sur un des composants et $-1$ sur l'autre. Reste ˆ montrer que c'est
$+1$ sur le composant admettant un modle de Whittaker. D'aprs 
la formule (7.1) page 106 de \cite{BC}
on dispose d'un dŽveloppement asymptotique de la forme
$$\Psi(\alpha)\sim c\nu(\alpha)\vert\alpha\vert^{1/2}\big\{\varphi(1)+\vef(\alpha){\mathbf R}(\vef)\varphi(1)\big\}\ptf
$$
On a deux expressions pour l'opŽrateur d'entrelacement appliquŽ
ˆ ce dŽveloppement asymptotique:
$$ b\vef(\alpha)\Psi(\alpha)\sim{\mathbf R}(\vef)c\nu(\alpha)\vert\alpha\vert^{1/2}\big\{\varphi(1)+\vef(\alpha){\mathbf R}(\vef)\varphi(1)\big\}$$
o $b=\pm1$. On a donc
$$\big(b\vef(\alpha)-{\mathbf R}(\vef)\big)\big(1+\vef(\alpha){\mathbf R}(\vef)\big)\varphi(1)=0$$
soit encore 
$$(b-1)\big(\vef(\alpha)+{\mathbf R}(\vef)\big)\varphi(1)=0$$ce qui impose $b=1$.
 \end{proof}

\Chapter {La formule des traces pour $SL(2)$}

\Section{Un peu d'histoire}

DŽsormais $F$ est un corps global. Nous allons tout d'abord
expliciter la formule des traces pour $SL(2)$ sous sa forme non-invariante
en suivant les techniques de troncatures gŽomŽtrique et spectrale utilisŽes dans \cite{JL} qui ont ŽtŽ
gŽnŽralisŽes par Arthur au cas d'un groupe rŽductif quelconque. 
Du c™tŽ spectral, \cad pour les sŽries d'Eisenstein,
la troncature a ŽtŽ introduite par Selberg en rang 1
puis gŽnŽralisŽe par Langlands en rang arbitraire.

Nous nous conformerons pour l'essentiel aux notations d'Arthur \cite{A1,A2}.
Les diverses troncatures utilisent la fonction
$\widehat\tau_\P$ qui pour $SL(2)$ ou $GL(2)$ est
la fonction caractŽristique des rŽels strictement positifs.
La troncature dans \cite[pages 529-531]{JL}  utilise une fonction notŽe
 $\chi(x)$ dŽfinie par une  inŽgalitŽ large
dŽpendant d'une constante $c_1$ alors que, suivant les conventions d'Arthur, nous utilisons
des inŽgalitŽs strictes. Sans cette diffŽrence
on aurait $\chi(x)=\widehat\tau_P(\HP(x)-\Tr)$ pour $\Tr=\log(c_1)$.
Pour les corps de nombres on obtient les mmes intŽgrales car la frontire est de mesure nulle;
ce n'est pas le cas pour un corps de fonctions.

Sauf mention expresse du contraire les mesures sur les groupes adŽliques sont les mesures
de Tamagawa. En particulier $\tau(\G)=\tau(\tG)=\tau(\M)=\tau(\U)=1$. De plus
$\tau(\Tef)$ est Žgal ˆ 1 si $E=F\oplus F$ et Žgal ˆ 2 si $E$ est un corps.

\Section{La troncature gŽomŽtrique}

Soit $f\in\ctyc\big(\G(\adef)\big)$. 
 Le noyau de la formule des traces est
$$K(x,y)=\sum_{\gamma\in\G(F)}f(x\mun\gamma y)$$
et on note $k(x)$ sa restriction ˆ la diagonale $x=y$.
Le noyau tronquŽ sur la diagonale est, 
pour $\Tr\in\RM$
$^(\!$\footnote{Ce paramtre est usuellement notŽ $T$,
en particulier chez Arthur. Il est ici notŽ $X$ pour libŽrer la lettre $T$ rŽservŽe aux tores.}$^)$.
$$k_{geom}^\Tr(x)=k(x)-\sum_{\xi\in\M(F)}
\sum_{\gamma\in \P(F)\bsl\G(F)}\widehat\tau_\P(\HP(\gamma x)-\Tr)
k_{P,\xi}(\gamma x)$$
o
$$k_{\P,\xi}(x)=\int_{\U(\adef)}f(x\mun\xi u x) du\ptf$$
On note $\G(F)_{ell}$ l'ensemble des ŽlŽments elliptiques dans $G(F)$ (rŽguliers ou non)
et on pose
$$k_{ell}(x)=\sum_{\gamma\in\G(F)_{ell}}f(x\mun\gamma x)\ptf$$
On note $$\M(F)^\star=\{\xi\in\M(F)\,|\, \xi\notin Z(F)\}\ptf$$
l'ensemble des ŽlŽments rŽguliers du tore dŽployŽ maximal. 
Pour $\xi\in M(F)^\star$ on pose
$$k_\xi(x)=\sum_{\gamma\in M(F)\bsl G(F)}f(x\mun\gamma\mun\xi \gamma x)$$
et
$$k_{\xi}^\Tr(x)=k_\xi(x)-\sum_{\gamma\in \P(F)\bsl\G(F)}\widehat\tau_\P(\HP(\gamma x)-\Tr)
\big(k_{\P,\xi}(\gamma x)+k_{\P,\xi\mun}(\gamma x)\big) du
\ptf$$ 
Les ŽlŽments $\xi$ et $\xi\mun$ Žtant conjuguŽs dans $\G$
sous le groupe de Weyl et toujours distincts on a donc
$$k_\xi(x)=k_{\xi\mun}(x)\comm{et}k_\xi^\Tr(x)=k_{\xi\mun}^\Tr(x)\ptf$$
C'est la contribution au noyau tronquŽ sur la diagonale
de la classe de conjugaison de $\xi\in M(F)^\star$. 
La contribution des classes de conjugaison d'ŽlŽments hyperboliques 
\cad des ŽlŽments semi-simples
ayant des valeurs propres rationnelles distinctes est donc
$$k_{hyp}^\Tr(x)=\sum_{\xi\in\M(F)^\star/W}k_{\xi}^\Tr(x)\ptf$$
Enfin, on introduit\footnote{On prendra garde que notre dŽfinition de $k_{unip}$ 
diffre de celle d'Arhur. D'une part, chez Arthur $k_{unip}$
ne contient que la contributions des unipotents alors que, de plus, nous sommons sur le centre
et il contient aussi la contribution de 1 qui est unipotent et elliptique et que nous avons rangŽ 
dans la contribution elliptique $k_{ell}$.}
$$k_{unip}(x)=\sum_{z\in Z(F)}\sum_{\gamma\in P(F)\bsl G(F)}\sum_{\nu\in\U(F)^\star}
 f(x\mun\gamma\mun z\nu \gamma x)$$
 o $U(F)^\star$ est l'ensemble des ŽlŽments $\nu\in\U(F)$ tels que $\nu\ne1$.
La contribution des classes de conjugaison des ŽlŽments de $\G(F)$ 
produits d'un ŽlŽment unipotent rŽgulier par un ŽlŽment du centre est donnŽe par
$$k_{unip}^\Tr(x)=k_{unip}(x)-\sum_{z\in Z(F)}
\sum_{\gamma\in \P(F)\bsl\G(F)}\widehat\tau_\P(\HP(\gamma x)-\Tr)k_{P,z}(\gamma x)\ptf$$
La fonction $k_{geom}^\Tr(x)$ se dŽcompose en la somme des contributions elliptiques, hyperboliques et unipotentes:
$$k_{geom}^\Tr(x)=k_{ell}(x)+k_{hyp}^\Tr(x)+k_{unip}^\Tr(x)\ptf$$
Chaque terme fournit une intŽgrale sur $$[\G]=\G(F)\bsl\G(\adef)$$ qui est convergente.
 La convergence de l'intŽgrale de $k_{ell}(x)$ rŽsulte de le finitude 
 du nombre de classes de conjugaisons qui contribuent. 
 La preuve de la convergence des intŽgrales des deux autres termes,
 au moins pour $\Tr$ assez grand, repose sur l'usage
 de la formule de Poisson (voir le cas unipotent plus bas) et
on renvoie ˆ \cite{LW} et \cite{LLe} pour des preuves dŽtaillŽes valables 
pour les groupes gŽnŽraux en toute caractŽristique. On pose:
$$J_\bullet^\Tr(f)=\int_{[\G]}k_\bullet^\Tr(x)\ddx$$
et on a
$$J_{geom}^\Tr(f)=J_{ell}(f)+J_{hyp}^\Tr(f)+J_{unip}^\Tr(f)\ptf$$
C'est le c™tŽ gŽomŽtrique de la formule des traces non invariante.
Pour $\Tr$ tendant vers l'infini ces expressions sont asymptotiques 
ˆ des polyn™mes en $\Tr$ si $F$ est un corps de nombres et
ˆ des ŽlŽment de PolExp pour les corps de fonctions (cf. \cite{LLe}). Il est usuel de remplacer
ces expressions par leur polyn™me asymptotique ŽvaluŽ en $\Tr=0$. 

\Subsection{Contribution elliptique}

L'intŽgrale orbitale adŽlique $\orb(\gamma,f)$ de $\gamma$ est l'intŽgrale
$$\orb(\gamma,f)=\int_{I_\gamma(\adef)\bsl\G(\adef)}f(x\mun\gamma x)\,\ddx$$
o $I_\gamma$ est le centralisateur de $\gamma$. On supposera que les mesures quotients
sont dŽfinies au moyen des mesures de Tamagawa.
Si $f$ est dŽcomposable i.e. $f=\otimes_v f_v$
on a une dŽcomposition en produit d'intŽgrales orbitales locales; toutefois  il convient de
tenir compte des facteurs qui servent ˆ normaliser les
mesures locales et globales permettant la convergence des produits.  
Pour les mesures de Tamagawa les facteurs de normalisation sont des fonctions $L$.
En particulier:

\begin{definition}\label{normell}  Soit $E$ un corps quadratique sur $F$.
Pour $\gamma\in\Tef(F)$ rŽgulier l'intŽgrale orbitale globale avec les mesures de Tamagawa s'Žcrit:
$$\orb(\gamma, f)=\int_{\Tef(\adef)\bsl\G(\adef)}
f(\tx\mun \gamma\tx)d\tx=
L(1,\kappa)\prod_v\frac{\orb_v(\gamma, f_v)}{L(1,\kappa_v)}$$
o $\kappa=\vef$, les intŽgrales orbitales locales Žtant calculŽes au moyen des
mesures de Tamagawa locales non normalisŽes.
\end{definition}

La classe de conjugaison d'un $\gamma$ elliptique contribue par le produit de
son intŽgrale orbitale adŽlique multipliŽe par le nombre de Tamagawa 
$\tau(I_\gamma)$ du centralisateur:
$$J_{ell}(f)=\sum_{\gamma\in\Gamma_{ell}}\tau(I_\gamma)\orb(\gamma,f)$$
o $\Gamma_{ell}$ est un ensemble de reprŽsentants des classes de 
$\G(F)$-conjugaison elliptiques.
On peut regrouper les termes au moyen de la conjugaison stable sur $F$.
Notons $\Sigma_{ell}$ un ensemble de reprŽsentants
des classes de conjugaison stables rationnelles elliptiques. Si on pose
$$\mathcal D(\gamma,F):=\tI_\gamma(F)\bsl\tG(F)/\G(F)$$
on a
$$J_{ell}(f)=\sum_{\gamma\in \Sigma_{ell}}\tau(I_\gamma)\sum_{\delta\in \mathcal D(\gamma,F)}
\orb(\delta\mun\gamma\delta,f)\ptf$$ 
Pour $\gamma\in\Z(F)$ l'ensemble $\mathcal D(\gamma,F)$ est trivial.
Un $\gamma$ elliptique rŽgulier dans $\G(F)$ engendre dans $M(2,F)$ une $F$-algbre
qui est une extension quadratique sŽparable $E$ de $F$. 
Notons $\Sigma_E$ un ensemble de reprŽsentants
des classes de conjugaison stables d'ŽlŽments elliptiques rŽguliers attachŽs 
ˆ la classe d'isomorphie de
l'extension quadratique sŽparable non dŽployŽe $E$ (cf.~\ref{ext})
et posons
$$J_E(f)=\tau(\Tef)\sum_{\gamma\in \Sigma_E}\sum_{\delta\in \mathcal D(\gamma,F)}
\orb(\delta\mun\gamma\delta,f)\ptf$$
On a alors
$$J_{ell}(f)=\tau(\G)\sum_{z\in Z(F)}f(z)+\sum_E J_E(f)$$
o $E$ parcourt l'ensemble des classes
d'isomorphisme d'extensions quadratiques sŽparable de $F$.

\begin{lemma}\label{jef} 
En posant $E^\star=E^1-\{\pm1\}$ o $E^1$ est le sous-groupe des ŽlŽments de norme 1
on a
$$J_E(f)=\frac{\tau(\Tef)}{2}\sum_{\gamma\in E^\star}\sum_{\delta\in \mathcal D(\gamma,F)}
\orb(\delta\mun\gamma\delta,f)\ptf$$
\end{lemma}
\begin{proof}  Le facteur $1/2$ rŽsulte de ce que dans $E^\star$ les ŽlŽments
$\gamma$ et $\gamma\mun$ sont toujours stablement conjuguŽs
mais distincts. \end{proof}

\Subsection{Contribution hyperbolique}
On observe que pour $\xi\in \M(F)^\star$ on a
$$\int_{U(F)\bsl U(\adef)}\sum_{\nu\in\U(F)} f(x\mun u\mun\xi\nu ux)du
=\int_{U(\adef)} f(x\mun u\mun\xi u x)du$$
et
$$k_{\P,\xi}(x)=\int_{\U(\adef)}f(x\mun\xi u x) du=
\int_{\U(\adef)}f(x\mun u\mun\xi u x) du\ptf$$
On introduit l'expression
$$h_\xi(x,\Tr)=\Big(\sum_{\nu\in\U(F)} f(x\mun\xi\nu x)\Big)
-\widehat\tau_\P(\HP( x)-\Tr)\Big(k_{P,\xi}( x)+k_{P,\xi\mun}(x)\Big)$$
on voit alors que
$$k_{\xi}^\Tr(x)=\sum_{\gamma\in \P(F)\bsl\G(F)}h_\xi(\gamma x,\Tr)$$
et on pose
$$J_{\xi}^\Tr(f)=\int_{\P(F)\bsl\G(\adef)}h_\xi(x,\Tr)\,dx
\ptf$$
On en dŽduit que  pour $\xi\in \M(F)^\star$
$$J_{\xi}^\Tr(f)=\int_{\P(F)\bsl\G(\adef)}h_\xi(x,\Tr)dx=
\int_{M(F)\bsl\G(\adef)} f(x\mun \xi  x)\tw(x,\Tr)dx$$
o, en notant $w$ l'ŽlŽment non trivial du groupe de Weyl, on a posŽ
$$\tw(x,\Tr)=1-\widehat\tau_\P(\HP( x)-\Tr)-\widehat\tau_\P(\HP(w x)-\Tr)\ptf$$
On observe que si $x=muk$ est une dŽcomposition d'Iwasawa
$$\HP(x)=\HP(m)\com{et} \HP(wx)=\HP(wu)-\HP(m)$$
On a  
$$\tw(x,\Tr)=1\com{si} \HP(wu)-\Tr\le\HP(m)\le\Tr$$
ainsi que
$$\tw(x,\Tr)=-1\com{si}\Tr<\HP(m)< \HP(wu)-\Tr$$
et $\tw(x,\Tr)=0$ sinon.
On en dŽduit que pour $u$ et $\Tr$ fixŽs la fonction $m\mapsto \tw(muk,\Tr)$ 
est ˆ support compact. On note $M(\adef)^1$  le groupe des $m\in\M(\adef)$ avec $\HP(m)=0$.
On normalise la mesure de Haar en imposant la mesure de Lebesgue (resp. la mesure de comptage)
sur $\M(\adef)^1\bsl\M(\adef)$ pour les corps de nombres (resp. les corps de fonctions) et
$\vol\big(M(F)\bsl M(\adef)^1\big)=1$.
On pose
$$\tv(x,\Tr)=\int_{M(\adef)^1\bsl M(\adef)}\tw(mx,\Tr)dm$$
et
$$J_{\xi}^\Tr(f)=\int_{M(\adef)\bsl\G(\adef)} f(x\mun\xi x)\tv(x,\Tr)dx\ptf$$
Il est usuel de ne considŽrer que la valeur en $\Tr=0$ : on pose $\tv(x)=\tv(x,0)$
et 
$$J_{\xi}(f)=\int_{M(\adef)\bsl\G(\adef)} f(x\mun\xi x)\tv(x)dx\ptf$$
Si $wu=m_1u'k_1$ on a  $\HP(wu)=\HP(m_1)$. 
On voit alors que si $F$ est un corps de nombres
$$\tv(x)=- \HP(wu)=\log_{q_F}\big(\vert\vert(1,n)\vert\vert\big)\com{si} x=muk\ptf$$
Pour les corps de fonctions on a
$$\tv(x)=1-\HP(wu)=1+\log_{q_F}\big(\vert\vert(1,n)\vert\vert\big)   \ptf$$
Nous laissons au lecteur le soin de faire le calcul explicite
des PolExp lorsque $\Tr$ est rationnel (cf. \cite{LLe}). 

La distribution $J_{\xi}(f)$ est une intŽgrale orbitale pondŽrŽe globale. 
Puisque $\HP$ est un logarithme on a pour les corps de nombres
(et ˆ un dŽcalage prs pour les corps de fonctions)
$$\tv(x)=\sum_v \tv_v(x_v)$$
et l'intŽgrale est une somme  sur toutes les places $v$
de produits  d'objets locaux ˆ savoir des intŽgrales orbitales 
locales ordinaires aux places $v'\ne v$ et pondŽrŽe en $v$. 

La contribution des classes de conjugaison hyperboliques est donnŽe par
$$J_{hyp}^\Tr(f)=\int_{[\G]}k_{hyp}^\Tr(x)dx
=\sum_{\xi\in\M(F)^\star/W}J_{\xi}^\Tr(f)=\frac{1}{2}\sum_{\xi\in\M(F)^\star}J_{\xi}^\Tr(f)
\ptf$$

\Subsection{Contributions unipotentes}
On pose
$$\Phi_{unip}^\Tr(f,x,z)=\Big(\sum_{\nu\in\U(F)^\star}f(x\mun z\nu x)\Big)
-\widehat\tau_\P(\HP(x)-\Tr)\int_{\U(\adef)}f(x\mun zu x)\,du$$
et
$$f^K(x)=\int_Kf(k\mun xk)\,dk\ptf$$
Par dŽfinition,
$$J_{unip}^\Tr(f)=\sum_{z\in Z(F)} J_{unip}^\Tr(f,z)
\com{o}J_{unip}^\Tr(f,z)=\int_{\M(F)\U(\adef)\bsl\G(\adef)}\Phi_{unip}^\Tr(f,x,z)\,dx\ptf$$
En dŽcomposant $x$ sous la forme $x=muk$ on voit que
$$J_{unip}^\Tr(f,z)=\int_{\M(F)\bsl\M(\adef)}\Phi_{unip}^\Tr(f^K,m,z)\delta_P(m)\mun\,dm\ptf$$
On pose 
$$g(n,z)=\int_K f\Big(k\mun z\begin{pmatrix}1&n\cr 0&1\end{pmatrix}k\Big)dk
= f^K\Big(z\begin{pmatrix}1&n\cr 0&1\end{pmatrix}\Big)$$
ainsi que
$$m=\begin{pmatrix}a&0\cr 0&a\mun\end{pmatrix}
\com{et}
\nu=\begin{pmatrix}1&\xi\cr 0&1\end{pmatrix}\ptf$$
Avec ces notations on a $\delta_P(m)=|a|^2$ et
$$\Phi_{unip}^\Tr(f^K,m,z)=\Big(\sum_{\xi\in F^\times}g(\a^{-2}\xi,z)\Big)
-\widehat\tau_\P\big(\log_{q_F}|a|-\Tr\big)\int_{\adef}g(\a^{-2}n,z)dn\ptf
$$
ConsidŽrons les intŽgrales, portant sur des sous-ensembles du groupe $C_F$ des classes d'idles,
la mesure sur ce groupe Žtant notŽe $\da$:
$$A^\Tr(f,z)=
\int_{|\a|\ge q_F^{-\Tr}}\Big(\sum_{\xi\in F^\times}g(\a^{2}\xi,z)\Big)|\a|^2\da $$
et
$$B^\Tr(f,z)=
\int_{|\a|<q_F^{-\Tr}}\Bigg(\sum_{\xi\in F^\times}
g(\a^{2}\xi,z)-\int_{\adef}g(\a^{2}n,z)dn\Bigg)|\a|^2\da \ptf$$
\begin{lemma}\label{undeux}
$$J_{unip}^\Tr(f,z)=A^\Tr(f,z)+B^\Tr(f,z)\ptf$$
\end{lemma}
\begin{proof}
La convergence de la premire intŽgrale rŽsulte de la compacitŽ du support de $g$.
La convergence de la seconde intŽgrale est prouvŽe en faisant appel 
aux techniques mises en {\oe}uvre dans la thse de Tate 
et qui reposent sur l'usage de la formule de Poisson.
\end{proof}
Pour $X\ge0$ on pose $v_F(X)=X$ pour les corps de nombres et
pour les corps de fonctions on pose
$v_F(X)=E(\Tr)$  la partie entire de $X$,
faisant ainsi appara"tre des PolExp lorsque $\Tr$ est rationnel (cf. \cite{LLe}).
On a alors
$$J_{unip}^\Tr(f,z)=A^0(f,z)+B^0(f,z)+\Big(\int_{\adef}g(n,z)dn\Big)v_F(X)\ptf$$ 
Une formule plus explicite pour $J_{unip}(f)$ sera donnŽe
au moyen de la prŽ-stabilisation.

\Section{Troncatures et identitŽ fondamentale}

La troncature spectrale repose sur l'opŽrateur de troncature d'Arthur $\Lambda^\Tr$ qui gŽnŽralise
la troncature des sŽries d'Eisenstein due ˆ Selberg et Langlands. Pour $SL(2)$ 
et une fonction $\varphi$ sur $[\G]$ l'opŽrateur de troncature s'Žcrit:
$$\Lambda^\Tr \varphi(x)=\varphi(x)-\sum_{\gamma\in P(F)\bsl G(F)}
\int_{u\in[\U]}\widehat\tau_\P(\HP(\gamma x)-\Tr) \varphi(u\gamma x)du\ptf$$
C'est un projecteur orthogonal dans l'espace de Hilbert $L^2([\G])$
qui agit trivialement sur le sous-espace des fonctions cuspidales.
Une troncature compatible avec la dŽcomposition spectrale,
est obtenue en faisant agir l'opŽrateur $\bs\Lambda^\Tr$ 
sur la premire variable du noyau $K(x,y)$:
$$K_{spec}^\Tr(x,y)=K(x,y)\,-\sum_{\gamma\in P(F)\bsl G(F)}
\int_{u\in[\U]}\widehat\tau_\P(\HP(\gamma x)-\Tr) K(u\gamma x, y)du\ptf$$
On considre alors sa restriction ˆ la diagonale:
$$k_{spec}^\Tr(x)=K_{spec}^\Tr(x,x)\ptf$$

\begin{proposition} \label{idfond}
Pour $\Tr$ assez grand (dŽpendant du support de $f$) on a
$$k_{geom}^\Tr(x)=k_{spec}^\Tr(x)\ptf$$
\end{proposition}
\begin{proof}
On doit montrer que pour $\Tr$ assez grand
$$\sum_{\gamma\in P(F)\bsl G(F)}\sum_{\xi\in\M(F)}
\widehat\tau_\P(\HP(\gamma x)-\Tr)\int_{\U(\adef)} f(x\mun\gamma\mun u\xi \gamma x) du\ptf$$
est Žgal ˆ
$$\sum_{\gamma\in \P(F)\bsl\G(F)}
\sum_{\delta\in\U(F)\bsl\G(F)}\widehat\tau_\P(\HP(\gamma x)-\Tr)
\int_{\U(\adef)}f(x\mun\gamma\mun u\delta x) du$$ 
qui peut s'Žcrire
$$\sum_{\gamma\in \P(F)\bsl\G(F)}
\sum_{\delta'\in\U(F)\bsl\G(F)}\widehat\tau_\P(\HP(\gamma x)-\Tr)
\int_{\U(\adef)}f(x\mun\gamma\mun u\delta' \gamma x) du\ptf$$
Il suffit alors d'observer que
$$\widehat\tau_\P(\HP( x)-\Tr) f(x\mun u\delta' x)=0$$ sauf peut-tre pour $\delta'\in \P(F)$
si $\Tr$ est assez grand.
\end{proof}

\begin{remark}{\rm
Nous avons utilisŽ une troncature spectrale simple: la restriction ˆ la diagonale
du noyau avec troncature sur la premire variable.
Dans le {\it Morning Seminar} (cf. \cite{LW}) on a introduit une troncature spectrale plus compliquŽe, 
qui elle a l'avantage de donner lieu ˆ une ŽgalitŽ pour toutes les valeurs de $\Tr$ et qui est
vraie plus gŽnŽralement pour les espaces tordus \cite[Proposition 8.2.2]{LW}. 
Dans le cas non tordu la troncature spectrale simple donne une ŽgalitŽ pour $\Tr$ assez grand
qui n'est pas vraie dans le cas tordu le plus gŽnŽral (voir \cite[Proposition 10.3.4]{LW}).
}\end{remark}

\Section{DŽcomposition spectrale}
Nous aurons besoin de la dŽcomposition spectrale de la reprŽsentation rŽgulire droite
$\rho$ dans l'espace de Hilbert $L^2([\G])$ o $[\G]=\G(F)\bsl\G(\adef)$.
On note $L^2_{cusp}([\G])$ l'adhŽrence de l'espace des fonctions $\vf$ lisses et cuspidales
\cad  telles que $\vf^0$, le {\it terme constant le long de $\P$}, est nul:
 $$\vf^0(x):=\int_{\U(F)\bsl\U(\adef)}\vf(ux) du=0\ptf$$
On note $L^2_\P([\G])$ est l'adhŽrence de l'espace des fonctions de la forme
$$\theta_\phi(x)=\sum_{\gamma\in\P(F)\bsl\G(F)}\phi(\gamma x)$$
o $\phi$ est une fonction lisse ˆ support compact sur $\U(\adef)\P(F)\bsl\G(\adef)$.
On a une premire dŽcomposition
$$L^2([\G])=L^2_{cusp}([\G])\oplus L^2_\P([\G])\ptf$$
En effet si $\vf$ est orthogonale ˆ toutes les fonctions $\theta_\phi$ alors $\vf$ est cuspidale car
$$\int_{[\G]}f(x)\th_\phi(x)dx=\int_{\U(\adef)\P(F)\bsl\G(\adef)}\vf^0(x)\phi(x)\,\ddx\,\ptf$$
Si $\lambda$ est un caractre (non nŽcessairement unitaire) de $\M(F)\bsl\M(\adef)$
prolongŽ ˆ $\P(\adef)$ trivialement sur $\U(\adef)$ on note $V_\lambda$ 
l'espace des fonction $\Phi$ qui vŽrifient
$$\Phi(muk,\lambda)=m^{\lambda+r}\Phi(k,\lambda)$$
o $m^r=\delta_\P(m)^{1/2}$ et dont la restriction ˆ $K$ est dans $L^2(K)$.
C'est l'espace de la reprŽsentation de la sŽrie principale adlique
$\pi_\lambda$ \cad de  l'induite parabolique du caractre $\lambda$, 
le groupe agissant par translations ˆ droite:
$$\big(\pi_\lambda(g)\Phi\big)(x,\lambda)=\Phi(xg,\lambda)\ptf$$
On analyse le spectre de $L^2_\P([\G])$ gr‰ce aux sŽries d'Eisenstein $E_\lambda$
$$E_\lambda(x,\Phi)=\sum_{\gamma\in P(F)\bsl G(F)}\Phi(\gamma x,\lambda)
\comm{pour}\Phi\in V_\lambda\ptf$$
Les sŽries d'Eisenstein dŽfinissent, au moins formellement, des opŽrateurs d'entrelacements
entre l'espace $V_\lambda$ et l'espace des formes automorphes.
Mais ces  sŽries ne convergent pas pour les valeurs utiles du paramtre $\lambda$ et
un prolongement mŽromorphe est nŽcessaire pour montrer que
$L^2_\P([\G])$ se dŽcompose en la somme d'un spectre rŽsiduel et d'un spectre continu:
$$ L^2_\P([\G])= L^2_{res}([\G])\oplus L^2_{cont}([\G])$$
Le spectre rŽsiduel $L^2_{res}([\G])$ pour $\G=SL(2)$ est de dimension 1: il
se rŽduit ˆ la reprŽsentation triviale.
Le spectre discret est la somme 
du spectre cuspidal, du spectre rŽsiduel:
$$ L^2_{disc}([\G])= L^2_{cusp}([\G])\oplus L^2_{res}([\G])\ptf$$
Le spectre continu peut Žcrire formellement
$$L^2_{cont}([\G])=\frac{1}{2}\int_{\lambda\in\Lambda}^\oplus V_\lambda\ d\lambda$$
o $\Lambda$ est le dual de Pontryagin de $M(F)\bsl\M(\adef)$. 
La prŽsence du facteur $1/2$ vient de ce que
les reprŽsentations $\pi_\lambda$ et $\pi_{w\lambda}$ sont Žquivalentes.
De faon plus explicite un ŽlŽment $\vf\in L^2_{cont}([\G])$ peut s'Žcrire (encore formellement)
$$\vf(x)=\frac{1}{2}\int_\spec \sum_{\Phi\in{\mathcal B}_\lambda}\langle
\vf, E_\lambda(\bullet,\Phi)\rangle_{[\G]} E_\lambda(x,\Phi)\,d\lambda$$
o ${\mathcal B}_\lambda$ est une base orthonormale de $V_\lambda$.
Nous renvoyons ˆ la littŽrature pour un traitement rigoureux de ce qui prŽcde.

\Section{Le c™tŽ spectral de la formule des traces}

L'opŽrateur de troncature opre trivialement sur le spectre cuspidal. On a donc une
dŽcomposition
$$K_{spec}^\Tr(x,y)=K_{cusp}(x,y)+K_{res}^\Tr(x,y)+K_{cont}^\Tr(x,y)\ptf$$
Le c™tŽ spectral de la formule des traces est donnŽ par l'intŽgrale sur $[\G]$ du
noyau $K_{spec}^\Tr$  restreint ˆ la diagonale:
$$J_{spec}^\Tr(f):=\int_{[\G]}k_{spec}^\Tr(x)dx\com{o}
k_{spec}^\Tr(x)=K_{spec}^\Tr(x,x)$$
et on a 
$$J_{spec}^\Tr(f):=J_{cusp}(f)+J_{res}^\Tr(f)+J_{cont}^\Tr(f)\ptf$$
L'opŽrateur $K_{cusp}(x,y)$ est un opŽrateur ˆ trace. On obtient donc
$$J_{cusp}(f)=\trace \rho_{cusp}(f)=\sum_{\pi}m_{cusp}(\pi)\trace\pi(f)$$
o $m_{cusp}(\pi)$ est la multiplicitŽ de $\pi$ dans le spectre cuspidal.
La partie rŽsiduelle fournit l'intŽgrale de $\Lambda^\Tr f$ sur le domaine fondamental:
$$J_{res}^\Tr(f)=\frac{1}{\vol([\G])}\int_{[\G]}\Lambda^\Tr {\bs 1}\,\,\ddx
\int_{G(F)}f(x)\,dx\ptf$$
Il reste ˆ calculer l'intŽgrale
$$\int_{[\G]}k_{cont}^\Tr(x)dx=\frac{1}{2}\int_{[\G]}
\int_\speccont \sum_{\Phi\in{\mathcal B}_\lambda}\Lambda^\Tr E_\lambda(\pi_\lambda(f)\Phi,x) 
\overline{E_\lambda(x,\Phi)} \,d\lambda\,\ddx\ptf$$
Pour obtenir des formules explicites il faut inverser l'ordre des intŽgrations. 
Pour les corps de fonctions la compacitŽ du dual d'un rŽseau rend la preuve aisŽe; 
c'est plus dŽlicat pour les corps de nombres.
On trouvera dans \cite{LW} et \cite{LLe} des preuves valables 
pour des groupes gŽnŽraux en toute caractŽristique. 
Nous esquissons les arguments pour les corps de nombres. 
Le cas des corps de fonction se traite de faon analogue
et ne prŽsente aucune difficultŽ supplŽmentaire autre que de notation. (voir \cite{LLe}).

Pour les corps de nombres si $\lambda=sr+\chi$ avec $\chi$ unitaire trivial sur $\RM^\times_+$
vu comme sous-groupe de $\M(F)\bsl\M(\adef)$ et $\Re(s)>0$ on pose
$$\ve(\lambda)=\int_{\M(F)\bsl\M(\adef)}\widehat\tau_\P(\HP(m))m^{-\lambda}\,dm\ptf
$$
Cette fonction est nulle si $\chi$ est non trivial sur $\M(F)\bsl\M(\adef)^1$. Si $\chi$ est trivial
elle vaut $1/s$.
La formulation pour les corps de fonctions est laissŽe au lecteur (cf.~\cite{LLe}).
On introduit,  avec des notations analogues ˆ celles de \cite{LW},
$$\omega^\Tr(\lambda,\mu)=\sum_\ws\sum_\wt
q_F^{\langle \ws\lambda-\wt\mu,\Tr\rangle}\ve(\ws\lambda-\wt\mu)
{\mathbf M}(\wt,\mu)\mun{\mathbf M}(\ws,\lambda)$$ les sommes en $\ws$ et $\wt$
portant sur le groupe de Weyl pour le tore dŽployŽ. 
On observera que $\ve(\ws\lambda-\wt\mu)$ est nul si $(\ws\chi-\wt\chi)\ne0$.

\begin{lemma}
$$\int_{[\G]}\Lambda^\Tr E_\lambda(x,\pi_\lambda(f)\Phi) 
\overline{E_{-\overline\mu}(x,\Phi)}\,\ddx=\big\langle\omega^\Tr(\lambda,\mu)\pi_\lambda(f)\Phi,\Phi
\big\rangle_K$$
On a alors
$$J_{cont}^\Tr(f)=\frac{1}{2}\int_\Lambda\trace(\lim_{\mu\to\lambda}\omega^\Tr(\lambda,\mu)
\pi_\lambda(f))\,d\lambda\ptf
$$
\end{lemma}
\begin{proof}On note $E_\lambda^0$  le ``terme constant'' de $E_\lambda$:
$$E_\lambda^0(x,\Phi)=\int_{\U(F)\bsl\U(\adef)}E_\lambda(ux,\Phi)\,du$$
et on rappelle que
$$E_\lambda^0(x,\Phi)=\Phi(x,\lambda)+\big(\mathbf M(\ws,\lambda)\Phi\big)(x,\ws\lambda)\ptf$$
On commence par calculer le produit scalaire
$$A(\lambda,\mu):=\int_{[\G]}\Lambda^\Tr E_\lambda(x,\pi_\lambda(f)\Phi)
\overline{E_{-\overline\mu}(x,\Phi)}\,\ddx $$
pour des valeurs des paramtres pour lesquelles les sŽries:
$$E_\lambda(x,\Phi)=\sum_{\gamma\in P(F)\bsl G(F)}\Phi(\gamma x,\lambda)$$
convergent. On observe que dans ce cas
$\Lambda^\Tr E_\lambda(x,\pi_\lambda(f)\Phi)$ est donnŽ par la sŽrie
$$\sum_{\gamma\in\P(F)\bsl\G(F)}B(\gamma x,\pi_\lambda(f)\Phi,\lambda)$$
o
$$B(x,\Phi,\lambda)=\Phi(x,\lambda)-\widehat\tau_\P(\HP(x)-\Tr)
\big(\Phi( x,\lambda)+\big(\mathbf M(\ws,\lambda)\Phi\big)(x,\ws\lambda)\big)
$$
et donc
$$A(\lambda,\mu)=\int_{\P(F)\U(\adef)\bsl\G(\adef)}
B(x,\pi_\lambda(f)\Phi,\lambda)\overline{E^0_{-\overline\mu}(x,\Phi)}\,\ddx $$
et on obtient
$$A(\lambda,\mu)=
\big\langle\omega^\Tr(\lambda,\mu)\pi_\lambda(f)\Phi,\Phi\big\rangle_K\ptf$$
\end{proof}
On doit maintenant passer ˆ la limite lorsque $\mu$ tend vers $\lambda$.
Explicitons cette limite pour les corps de nombres. 
Supposons que $\lambda$ est unitaire de la forme
$$m^\lambda=m^{sr+\chi}\com{o} 
m^r=\vert m^\alpha\vert^{1/2}$$
et o $\chi$ est trivial sur $\RM^\times_+$ identifiŽ ˆ un sous groupe de $C_F$.
Soit $w$  l'ŽlŽment non trivial du groupe de Weyl,
on notera ${\mathbf M}'(w,\lambda) $ la dŽrivŽe de 
${\mathbf M}(w,\lambda)$ par rapport ˆ $s$ et on pose
$${\mathcal M}(\lambda) ={\mathbf M}(w,\lambda)\mun{\mathbf M}'(w,\lambda)\ptf$$
Enfin on pose $\delta(\lambda)=1$ si  la restriction de $\lambda$ ˆ $\M(F)\bsl\M(\adef)^1$ 
est triviale et et $\delta(\lambda)=0$ sinon.
Avec ces notations on obtient l'expression suivante:
$$
\omega^\Tr(\lambda,\lambda)=c\,\Tr-{\mathcal M}(\lambda)+
\delta(2\lambda)\frac{\sin s \langle \Tr,Y\rangle}{s}{\mathbf M}(w,\lambda)
$$
o   $c$ et $Y$ sont des constantes.
Les deux premiers termes donnent, pour les corps de nombres, par intŽgration en $\lambda$
un polyn™me du premier degrŽ en $\Tr$:
$$J_{cont}^\Tr(f)=\frac{c}{2}\Big(\int_\speccont\trace \pi_\lambda(f)\,d\lambda\Big)\,\Tr-
\frac{1}{2}\int_\speccont\trace\Big({\mathcal M}(\lambda)\pi_\lambda(f)\Big)\,d\lambda\ptf$$
Le troisime terme donne, par intŽgration en $\lambda$, une expression
$J_{compl}^\Tr(f)$
qui est une somme d'intŽgrales oscillantes en $s$ indexŽes par les caractres d'ordre 2.
Lorsque le paramtre de troncature tend vers l'infini l'intŽgrale oscillante tend,
ˆ un facteur scalaire prs, vers la mesure de Dirac 
et on obtient ˆ la limite le 
\begin{lemma} \label{osc}
Si $F$ est un corps de nombres
$$\frac{1}{2}\sum_\chi\delta(2\chi)\lim_{\Tr\to\infty}\int_\RM
\frac{2\sin s \langle \Tr,Y\rangle}{2s}{\mathbf M}(w,sr+\chi)\pi_{sr+\chi}(f)\frac{ds}{2\pi}$$
est Žgal ˆ
$$\frac{1}{4}\sum_{\{\lambda\,\vert\,2\lambda=0\}}
{\mathbf M}(w,\lambda)\pi_\lambda(f)\ptf
$$
\end{lemma}
\begin{proof}
En effet, pour $h$ lisse et ˆ dŽcroissance rapide sur $\RM$
$$\lim_{N\to\infty}\int_\RM h(x)\,\frac{2\sin Nx}{x}\,\,\frac{dx}{2\pi}=h(0)\ptf$$
\end{proof}

On a une formule analogue pour les corps de fonctions. Au total on a
$$\int_{[\G]}k_{cont}^\Tr(x)\ddx=J_{cont}^\Tr(f)+J_{compl}^\Tr(f)\ptf$$
L'expression spectrale se dŽcompose donc en une partie discrte et une partie continue:
$$J_{spec}^\Tr(f)=J_{disc}^\Tr(f)+J_{cont}^\Tr(f)$$
et la partie discrte $J_{disc}^\Tr(f)$ est somme de trois termes:
$$J_{disc}^\Tr(f)=J_{cusp}(f)+J_{res}^\Tr(f)+J_{compl}^\Tr(f)\ptf$$
Les deux derniers termes ont une limite finie lorsque $\Tr$ tend vers $+\infty$
et en fait $J_{disc}^\Tr(f)$ est indŽpendant de $\Tr$. On pose
$$J_{res}(f)=\lim_{\Tr\to\infty}J_{res}^\Tr(f)
\com{et}
J_{compl}(f)=\lim_{\Tr\to\infty}J_{compl}^\Tr(f)
$$
et on a
$$J_{disc}(f)=J_{cusp}(f)+J_{res}(f)+J_{compl}(f)=\sum_\pi a(\pi)\trace \pi(f)$$ 
o les $a(\pi)$ sont des nombres rationnels.
La distribution $J_{disc}$, qui est une combinaison linŽaire de traces, est donc
invariante par conjugaison de $f$ sous $\G(\adef)$.
Les deux premiers termes correspondent ˆ la trace de l'opŽrateur $\rho_{disc}(f)$
dŽfini par $f$ dans le spectre discret $L^2_{disc}([\G])$:
$$\trace \rho_{disc}(f)=\trace \rho_{cusp}(f)+\trace \rho_{res}(f)=J_{cusp}(f)+J_{res}(f)\ptf$$
Pour $\pi$ dans le spectre discret $a(\pi)=m(\pi)$ la multiplicitŽ de $\pi$
dans le spectre; ce sont des entiers positifs\footnote{Il a ŽtŽ dŽmontrŽ \cite{Rama}
que la multiplicitŽ est toujours 1. Mais nous n'utiliserons pas ce fait.}.
En particulier
$$J_{res}(f)=\trace \rho_{res}(f)=\trace {\bs 1}(f)=\int_{[\G]}f(x)\ddx\ptf$$
Les contributions discrtes 
provenant de la troncature du spectre continu donnent le terme complŽmentaire $J_{compl}(f)$
pour lequel les $a(\pi)$ ne sont plus nŽcessairement des entiers positifs. De fait ˆ la limite on a 
vu en \ref{osc} que
$$J_{compl}(f)=\sum_{\{\lambda\,\vert\,2\lambda=0\}}
\frac{1}{4}\,
\trace \big({\mathbf M}(w,\lambda)\pi_\lambda(f)\big)$$
o on somme sur les  caractres $\lambda$ d'ordre 2 de $M(F)\bsl\M(\adef)\simeq C_F$.%
$^(\!$\footnote{Si on compare avec \cite{LW} on voit que 
le scalaire $1/4$ est l'inverse du produit du cardinal $\#W$ du groupe de Weyl $W$ et de la valeur absolue
du dŽterminant de $(1-w)$ agissant dans l'espace vectoriel rŽel ${\mathfrak a}_M$ attachŽ ˆ $M$:
$\#W.\vert\det(1-w)\vert=4$.
}$^)$
Si $\lambda=0$ la sŽrie principale $\pi_\lambda$
est irrŽductible et l'opŽrateur d'entrelacement ${\mathbf M}(w,0)$ est le scalaire $-1$:
$$\trace \big({\mathbf M}(w,0)\pi_1(f)\big)=-\trace\big(\pi_0(f)\big)\ptf$$
On a donc $a(\pi_0)=-1/4$.
Si $\lambda\ne0$ la reprŽsentation $\pi_\lambda$ se dŽcompose en une somme
infinie de reprŽsentations irrŽductibles deux ˆ deux inŽquivalentes qui peuvent se regrouper en deux
sous-reprŽsentations:
$$\pi_\lambda=\pi_\lambda^+\oplus\pi_\lambda^-$$
o $\pi_\lambda^+$ est la somme des composants admettant un modle de Whittaker pour un caractre
$\psi$ de $U(F)\bsl U(\adef)$ localement partout sauf en un nombre pair de places. Cette paritŽ
est indŽpendante du choix de $\psi$. Il rŽsulte de \ref{entrelace}, par produit sur toutes les places, que
$$\trace \big({\mathbf M}(w,\lambda)\pi_\lambda(f)\big)=\trace\pi_\lambda^+(f)-\trace\pi_\lambda^-(f)$$
et donc $a(\pi)=\pm1/4$ si $\pi$ est un composant de $\pi_\lambda^\pm$.

\Section{Les formules des traces}

On rappelle que, d'aprs \ref{idfond}, pour $\Tr$ assez grand (dŽpendant du support de $f$) on a 
$$k_{geom}^\Tr(x)=k_{spec}^\Tr(x)\ptf$$
On obtient par intgration sur $[\G]$ la formule des traces non invariante:
\begin{proposition}\label{ftninv}Pour $\Tr$ assez grand on a l'identitŽ:
$$J_{ell}(f)+J_{hyp}^\Tr(f)+J_{unip}^\Tr(f)=J_{disc}(f)+J_{cont}^\Tr(f)\ptf$$
\end{proposition}
Le terme elliptique $J_{ell}(f)$ et le terme spectral discret $J_{disc}(f)$
sont des distributions invariantes par conjugaison de $f$
\cad en remplaant $$f:g\mapsto f(g)\comm{par}f^x:g\mapsto f(x\mun g x)\ptf$$
Les autres termes de formule des traces ne sont pas invariants sous une telle conjugaison: 
ils dŽpendent du choix du compact maximal $K$ et du sous-groupe parabolique minimal $P$. 

\Chapter{Pr\'e-stabilisation}\label{prestab}

\Section{La mŽthode}
On va Žtendre chaque terme du dŽveloppement de $k_{geom}^\Tr(x)$ et de $k_{spec}^\Tr(x)$
en une fonction sur $$[\tG]=\tG(F)\tZ(\adef)\bsl\tG(\adef)\ptf$$
Ce sont ces fonctions que l'on souhaite analyser: en effet on sait que la conjugaison
sous $GL(2)$ Žquivaut ˆ la conjugaison stable pour $SL(2)$.
La transformation de Fourier sur 
$$ Q_F\simeq G(\adef)\tG(F)\tZ(\adef)\bsl\tG(\adef)$$
fournira la prŽ-stabilisation. 
On observe tout d'abord que 
$$k(\tx)=\sum_{\gamma\in\G(F)}f(\tx\mun\gamma\tx)$$
a un sens pour $\tx\in\tG(\adef)$ puisque $\det(\tx\mun\gamma\tx)=1$
et que $k(\eta \tx)=k(\tx)$ pour tout $\eta\in\tZ(\adef)\tG(F)$. On a ainsi dŽfini une fonction
sur $$[\tG]=\tG(F)\tZ(\adef)\bsl\tG(\adef)\ptf$$
Nous allons montrer qu'il en est de mme pour tous les $k^\Tr_\bullet(x)$:

\begin{lemma}\label{troncinv}
Les fonctions $$x\mapsto k^\Tr_\bullet(x)$$ sont invariantes ˆ gauche par
$\tG(F)\tZ(\adef)$.
\end{lemma}
\begin{proof}L'invariance sous $\tZ(\adef)$ est Žvidente. On sait dŽjˆ
que ces fonctions sont invariantes ˆ gauche par $\G(F)$. Il suffit 
donc de prouver l'invariance par les $\xi\in\tM(F)$. Mais de tels $\xi$ normalisent
$\P(F)$ et $\G(F)$ et donc laissent invariantes les troncatures.
\end{proof}

Par intŽgration sur $[\tG]$
contre un caractre $\kappa$ de $Q_F$ on dŽfinit des distributions
 $$\J^\Tr_\bullet(f,\kappa)=\int_{[\tG]}\kappa(\det \tx)k_\bullet^\Tr(\tx)\ddx$$
et on a, si $Q_F$ est muni de la mesure de Haar lui donnant le volume 1,
$$J^\Tr_\bullet(f)=\sum_\kappa \J^\Tr_\bullet(f,\kappa)\ptf$$
C'est ce que nous appellerons la prŽ-stabilisation de la formule des traces pour $G=SL(2)$ 
Nous Žtudions tous les termes plus en dŽtail ci-dessous. 


On rappelle que pour $\Tr$ assez grand (dŽpendant du support de $f$) on a 
d'aprs \ref{idfond}
$$k_{geom}^\Tr(x)=k_{spec}^\Tr(x)$$
et on voit de mme que
$$k_{geom}^\Tr(\tx)=k_{spec}^\Tr(\tx)\ptf$$
On en dŽduit que pour tout $\kappa$ une identitŽ que nous allons expliciter.
$$\J_{geom}^\Tr(f,\kappa)=\J_{spec}^\Tr(f,\kappa)\ptf$$

\Section{Contributions elliptiques}

Soit $E/F$ une extension quadratique sŽparable. On note $\vef$ 
 le caractre de $C_F$ associŽ ˆ $E/F$ par la thŽorie du corps de classe; c'est
 de faon naturelle un caractre de $Q_F$.
On note $\tTef$ le tore dŽfini par la restriction des scalaire ˆ la Weil
pour $E/F$ du groupe multiplicatif $E^\times$ et $\Tef$ le sous tore des ŽlŽments de
norme 1: 
$$\tTef(F)=E^\times\comm{et} \Tef(F)=E^1$$
o $E^1$ est le noyau de la norme $$\Nef:E^\times\to F^\times\ptf$$
On peut rŽaliser $\Tef$ comme un sous-groupe de $\G$ et 
$\tTef$ comme un sous-groupe de $\tG$.
Les intŽgrales orbitales qui interviennent dans l'expression $\J_E(f,\kappa)$ 
qui est la contribution des termes elliptiques rŽguliers relatif ˆ l'extensions
quadratique non dŽployŽe $E/F$
qui ont en facteur l'intŽgrale de $\kappa$ de l'image par le dŽterminant de
$$\tZ(\adef)\tTef(F)\bsl \tTef(\adef)=\AM_F^\times E^\times\bsl\AM_E^\times
\simeq C_F\bsl C_E\ptf$$ 
Cette intŽgrale est nulle sauf si $\kappa$ est constant sur $\Nef(C_E)$ \cad si
$\kappa=1$ ou $\kappa=\vef$.
On a donc $$J_E(f)=\J_E(f,1)+\J_E(f,\vef)\ptf$$

On  a dŽfini en \ref{normell}
une intŽgrale orbitale globale pour $\gamma\in\Tef(F)$ rŽgulier en posant
$$\orb(\gamma, f)=\int_{\Tef(\adef)\bsl\G(\adef)}
f(\tx\mun \gamma\tx)d\tx=
L(1,\kappa)\prod_v\frac{\orb_v(\gamma, f_v)}{L(1,\kappa_v)}$$
o $\kappa=\vef$.
Les mesures utilisŽes pour le calcul des intŽgrales orbitales locales
$\orb_v(\gamma, f_v)$ sont les quotients de mesures de Tamagawa  locales non normalisŽes
sur $\T_v$ et $\G_v$. 
On peut maintenant dŽfinir une intŽgrale $\kappa$-orbitale globale en posant
$$\orb^{\vef}(\gamma, f)=\sum_{\eta\in \mathcal D(\gamma,\adef)}\vef(\det(\eta))\orb(\eta\mun\gamma\eta,f)$$
o 
$$\mathcal D(\gamma,\adef)=\tTef(\adef)\G(\adef)\bsl\tG(\adef)\ptf$$
Mais ce groupe est la limite inductive sur la famille des ensembles finis $S$ de places:
$$\mathcal D(\gamma,\adef)=\lim_{\to}\mathcal D(\gamma,F_S)
\com{o}\mathcal D(\gamma,F_S)=\prod_{v\in S}\mathcal D(\gamma,F_v)
$$
et si on utilise les mesure locales comme ci-dessus on a
$$\orb^{\vef}(\gamma, f)=
\int_{\tTef(\adef)\bsl\tG(\adef)}\kappa(x)f(x\mun \gamma x)\ddx
=
L(1,\kappa)\prod_v\frac{\orb_v^{\kappa_v}(\gamma, f_v)}{L(1,\kappa_v)}
\ptf$$

\begin{remark}{\rm%
Dans [LL] le facteur de transfert est intŽgrŽ dans la dŽfinition des $\kappa$-intŽgrales orbitales locales
$\orb^{\kappa_v}(\gamma, f_v)$.}
\end{remark}
\begin{lemma}\label{kglob}  Pour $\gamma\in E^\star$ on a l'identitŽ
$$\sum_{\eta\in \mathcal D(\gamma,F)}\orb(\eta\mun\gamma\eta,f)=
\frac{1}{2}\sum_\kappa\orb^\kappa(\gamma, f)\ptf$$
\end{lemma}
\begin{proof}
On rappelle que 
$\mathcal D(\gamma,F)\simeq \Nef E^\times\bsl F^\times$.
Maintenant $$\tTef(\adef)\bsl\tG(\adef)\simeq
\tTef(\adef)\bsl\tG(\adef)^{+}\cup \tTef(\adef)\bsl\tG(\adef)^{+}\eta$$
o
$$\tG(\adef)^{+}=\{x\in\tG(\adef)\,\vert\,\det(x)\in F^\times\Nef\ade_E^\times\}
\com{et}\det(\eta)\notin F^\times\Nef\ade_E^\times\ptf$$
Par ailleurs, si on pose 
$$\G(\adef)^+=\{x\in\tG(\adef)\,\vert\,\det(x)\in F^\times\}$$
en utilisant que $F^\times\cap\Nef\ade_E^\times= \Nef E^\times$
(principe de Hasse valable pour les extensions quadratiques)
on voit que
$$E^\times \Tef(\adef)\bsl\G(\adef)^+\simeq\tTef(\adef)\bsl\tG(\adef)^{+}\ptf$$
La formule de Poisson pour le groupe discret $ \mathcal D(\gamma,F)$
vu comme sous-groupe d'indice 2 de 
$$ \mathcal D(\gamma,\adef)\simeq \Nef \ade_E^\times\bsl \adef^\times
$$
montre alors que
$$\sum_{\eta\in \mathcal D(\gamma,F)}\orb(\eta\mun\gamma\eta,f)=
\frac{1}{2}\sum_\kappa\orb^\kappa(\gamma, f)\ptf$$
\end{proof}

\begin{remark}{\rm  On aurait pu utiliser le rŽsultat plus gŽnŽral  \cite[ThŽorme 3.9]{tam}
qui, dans notre cas, affirme que
$$\tau(\Tef)\sum_{\eta\in \mathcal D(\gamma,F)}\orb(\eta\mun\gamma\eta,f)=
\tau(\G)\sum_\kappa\orb^\kappa(\gamma, f)$$
et on sait que $\tau(\Tef)=2$ et $\tau(\G)=1$.}
\end{remark}

\begin{corollary}\label{contrib-ell}
On suppose les $\kappa$-intŽgrales orbitales calculŽes au moyen 
du quotient des mesures de Tamagawa. Alors
$$\J_E(f,\vef)=\frac{\tau(\Tef)}{4}\sum_{\gamma\in E^\star}\orb^{\vef}(\gamma, f)=
\frac{1}{2}
\sum_{\gamma\in E^\star}\orb^{\vef}(\gamma, f)\ptf$$
\end{corollary}

\begin{proof}
D'aprs \ref{jef},
$$J_E(f)=\frac{\tau(\Tef)}{2}\sum_{\gamma\in E^\star}
\sum_{\eta\in \mathcal D(\gamma,F)}\orb(\eta\mun\gamma\eta,f)\ptf$$ 
Il suffit alors d'invoquer \ref{kglob}
\end{proof}

\Section{Contributions hyperboliques}
Le calcul de la contribution hyperbolique fait ci-dessus fournit l'expression suivante:
$$J_\xi^\Tr(f)=\sum_\kappa\J_\xi^\Tr(f,\kappa)$$
avec
$$\J_\xi^\Tr(f,\kappa)=\int_{\tZ(\adef)\tM(F)\bsl\tG(\adef)} f(\tx\mun\xi \tx)\tw(\tx,\Tr)\kappa(x)d\tx$$

\begin{lemma}\label{htriv}
Pour $m\in\M(\adef)$
$$\J_m^\Tr(f,\kappa)=\int_{\tZ(\adef)\tM(F)\bsl\tG(\adef)} f(\tx\mun m \tx)\tw(\tx,\Tr)\kappa(x)d\tx=0$$
si $\kappa$ est non trivial.
\end{lemma}
\begin{proof}
En effet,  
$$\tm_1\mapsto f(\tx\mun\tm_1\mun m \tm_1\tx) \tw(\tm_1x,\Tr)$$ est constant
pour $\tm_1\in \tM(\adef)$ et donc l'intŽgrale sur $\tZ(\adef)\tM(F)\bsl\tM(\adef)$ est nulle.
\end{proof}

On en conclut que
$$J_\xi^\Tr(f)=\J_\xi^\Tr(f,1)\ptf$$
Nous dirons que la contribution des ŽlŽments $F$-hyperboliques,
quoique non invariante, est \ts-stable.

\Section{Contributions unipotentes}

On souhaite calculer  $\J_{unip}^\Tr(f,z)$. 
Comme ci-dessus on  pose:
$$\Phi_{unip}^\Tr(f,x,z)=\sum_{\nu\in\U(F)^\star}f(x\mun z\nu x)
-\widehat\tau_\P(\HP(x)-\Tr)\int_{\U(\adef)}f(x\mun zu x)\,du$$
et
$$\J_{unip}^\Tr(f,\kappa,z)=\int_{Q_F}\kappa(y)\Big(
\int_{\M(F)\U(\adef)\bsl\G(\adef)}\Phi_{unip}^\Tr(f,xy,z)\,\ddx\Big)\,dy$$ 
o $Q_F$ est le groupe compact
$$Q_F=\tZ(\adef)\tM(F)\G(\adef)\bsl\tG(\adef)\simeq
\tZ(\adef)\tM(F)\M(\adef)\bsl\tM(\adef)$$
muni de la mesure de Haar qui lui donne le volume 1. Avec ces conventions
$$\J_{unip}^\Tr(f,\kappa,z)=
\int_{\tZ(\adef)\tM(F)\U(\adef)\bsl\tG(\adef)}\kappa(x)\,\,\Phi_{unip}^\Tr(f,x,z)\,\ddx$$
soit encore
$$\J_{unip}^\Tr(f,\kappa,z)=
\int_{\tZ(\adef)\tM(F)\bsl\tM(\adef)}\kappa(\tm)\,\,\Phi_{unip}^\Tr(f^K,\tm,z)\delta(\tm)\mun\,d\tm$$
et, par inversion de Fourier,
$$ \J_{unip}^\Tr(f,z)=\sum_\kappa \J_{unip}^\Tr(f,\kappa,z)$$
o on somme sur le dual de Pontryagin de $Q_F$. 
La somme {\it a priori} infinie ne porte que sur un nombre fini de termes, dŽpendant de $f$.
Ici ce n'est pas le support de $f$ qui importe mais sa lissitŽ, ˆ savoir le sous-groupe 
ouvert compact $H\in\G(\ade_F^\infty)$ sous lequel $f$ est bi-invariante.
On a notŽ $\ade_F^\infty$ l'anneau des adles finis et $H$ est un sous-groupe
d'indice fini du sous-groupe compact maximal de $\G(\ade_F^\infty)$:
$$K^\infty=\prod_{v\notin\{\infty\}}SL(2,\OFv)\ptf$$

ConsidŽrons, comme ci-dessus,  la fonction sur $\adef$:
$$g(n,z)=f^K( zu)
\com{o}
u=\begin{pmatrix}1&n\cr 0&1\end{pmatrix}\ptf$$
Posons
$$A^\Tr(f,\kappa,z)=\int_{{a\in\adef^\times\,,\,|\a}|\ge q_F^{-\Tr}}\kappa({\a})
\orbu({\a},z)|{\a}|^{}\da$$
et 
$$B^\Tr(f,\kappa,z)=\int_{a\in C_F\,,\,|{\a}|<q_F^{-\Tr}}\kappa({\a})h(a,z)
|{\a}|^{}\da $$
o $h(a,z)$ est la fonction sur le groupe $C_F=\adef^\times/F^\times$ des classes d'idles
$$h(a,z)=\sum_{\zeta\in F^\times}\orbu({\a}^{}\zeta,z)
-\int_{n\in\adef}\orbu({\a} n,z)dn
\ptf$$
On a alors
\begin{lemma}\label{unun}
Les mesures de Haar Žtant choisies de sorte que 
$Q_F$ ait le volume 1 on a
$$\J_{unip}^\Tr(f,\kappa,z)=\frac{1}{2}\Big(A^\Tr(f,\kappa,z)+B^\Tr(f,\kappa,z)\Big)\ptf$$
\end{lemma}
\begin{proof}
Le facteur $1/2$ vient du changement de variable $a\mapsto a^2$ dans la comparaison avec les intŽgrales
qui apparaissant dans \ref{undeux}. En effet, pour les corps de nombres
la mesure de Haar sur le groupe multiplicatif des rŽels vŽrifie
$d^\times t^2=2d^\times t$ et pour les corps de fonctions on doit sommer sur des rŽseaux
isomorphes ˆ $\ZM$ ou $2\ZM$ et on observe que 
$$\sum_{n\in2\ZM}\vf(n)=\frac{1}{2}\Big(\sum_{n\in\ZM}\vf(n)+\sum_{n\in\ZM}(-1)^n\vf(n)\Big)\ptf
$$
\end{proof}
Pour obtenir de expressions plus explicites on introduit des variantes rŽgularisŽes. 
Pour $s\in\CM$ et $\Re(s)>0$ on pose
$$C(f,s,\kappa,z)=\int_{\adef^\times}\kappa({\a})|{\a}|^{1+s} \orbu({\a},z)\da 
$$
et
$$D^\Tr(f,s,\kappa,z)=\int_{|{\a}|<q_F^{-\Tr}}\Big(\kappa({\a})|{\a}|^{1+s}\int_{n\in\adef} \orbu({\a}n,z)
dn\Big)\da \,$$
soit encore
$$D^\Tr(f,s,\kappa,z)=\Big(\int_{|{\a}|<q_F^{-\Tr}}\kappa({\a})|{\a}|^{s}\da\Big)
 \Big(\int_{\adef}\orbu(n,z)dn\Big)$$ 
qui sont convergentes pour $\Re(s)>0$.  On a alors
$$A^\Tr(f,\kappa,z)+B^\Tr(f,\kappa,z)=\lim_{s\to0}\big(C(f,s,\kappa,z)-D^\Tr(f,s,\kappa,z)\big)\ptf$$
On supposera dŽsormais que
 $f$ est un produit de fonctions locales $f_v$.
On introduit des facteurs locaux:
$$C(f_v,s,\kappa_v,z)=\int_{F_v^\times}\kappa_v(\a)|\a|^{1+s} \orbu_v(\a,z)\,\da\ptf
$$
On observe que pour presque toute place $v$ la fonction  
$$n\mapsto g_v(n,z)=\int_{K_v} f_v(k\mun zuk)dk
\com{o}
u=\begin{pmatrix}1&n\cr 0&1\end{pmatrix}$$
est la fonction caractŽristique
de l'anneau des entiers. Dans ce cas
$$C(f_v,s,\kappa_v,z)=L(1+s,\kappa_v)$$
o
$L(\bullet,\kappa_v)$ est le facteur local de la fonction $L$ de Hecke pour $\kappa_v$.
En passant au produit sur toutes les places on obtient que
$$C(f,s,\kappa,z)=L(1+s,\kappa)\prod_v\frac{C(s,\kappa_v)}{L(1+s,\kappa_v)}\ptf$$
Le produit ne porte que sur un nombre fini de places (dŽpendant du support de $f$)
et la seule singularitŽ au voisinage de $s=0$ est celle de la fonction $L$.

Pour $\kappa=1$ la fonction $L$ est la fonction Zeta du corps $F$:
$$L(1+s,1)=Z(1+s)$$
 et donc
$C(s,1,z)$ a un p™le en $s=0$; ce p™le est compensŽ par le p™le de $D^\Tr(s,1,z)$. 
On obtient une expression explicite en calculant la diffŽrence des termes d'ordre 0 
des dŽveloppements de Laurent. 
Reste ˆ montrer que 
$\J^0_{unip}(f,1,z)$ 
est   une somme de produits sur toutes les places d'intŽgrales orbitales 
pondŽrŽes sous le groupe adjoint. 
Le terme provenant de $D^0$ est, ˆ une constante prs, l'intŽgrale orbitale unipotente
sous le groupe adjoint. Par ailleurs
$$C(s,1,z)= Z(1+s)\prod_v \frac{C_v(s,1,z)}{Z_v(1+s)}$$
le produit sur les places $v$ Žtant toujours convergent et holomorphe au voisinage de $s=0$
puisque presque tous les facteurs valent 1.
Il suffit maintenant d'observer que la valeur en $s=0$ de $C_v(s,1,z)$ et de sa dŽrivŽe $C'_v(s,1,z)$ sont
les intŽgrales orbitales unipotentes pondŽrŽes (locales) sous le groupe adjoint:
$$C_v(0,1,z)= \int_{\tZ_v\U_v\bsl\tG_v}f_v(x\mun z{\bs\nu}x)\,\ddx=\orb^1_v(z{\bs\nu},f_v)
\com{o}
{\bs\nu}=\begin{pmatrix}1&1\cr 0&1\end{pmatrix}$$
et, pour une certaine constante $c_v$, on a
$$C'_v(0,1,z)= c_v
\int_{\tZ_v\U_v\bsl\tG_v}f_v(x\mun z{\bs\nu}x)H_{\P_v}(x)\,\ddx\ptf$$

Si $\kappa=\vef\ne1$ le terme $D^\Tr(s,\kappa)$ est nul. Donc $ \J_{unip}^\Tr(f,\kappa)$ 
est indŽpendant de $\Tr$ si $\kappa\ne1$ et on omettra alors l'exposant. 

\begin{lemma}\label{unipz} 
Si $\kappa=\vef\ne1$  
$$ \J_{unip}(f,\kappa,z)=\frac{1}{2}\lim_{s\to 0} C(f,s,\kappa,z)=
\frac{1}{2}\lim_{t\to z}\orb^\kappa(t,f)$$
pour $t\in\Tef(F)$.
\end{lemma}

\begin{proof}
Si $\kappa\ne1$  on a
$$C(0,\kappa,z)=\lim_{s\to 0}L(1+s,\kappa)\prod_v\frac{C(f_v,s,\kappa_v,z)}{L(1+s,\kappa_v)}
=L(1,\kappa)\prod_v\frac{C(0,\kappa_v,z)}{L(1,\kappa_v)}$$
o
$$C_v(0,\kappa_v,z)= \int_{\tZ_v\U_v\bsl\tG_v}\kappa_v(x)f_v(x\mun z{\bs\nu}x)\,\ddx$$
soit encore
$$C(0,\kappa,z)= L(1,\kappa)\prod_v\lim_{t\to z}
\Delta_{\kappa_v}(t,t)\frac{\orb^{\kappa_v}(t,f_v)}{L(1,\kappa_v)}
=\lim_{t\to z}\orb^\kappa(t,f)$$
pour $t\in\Tef(F)$.
\end{proof}

\Section{PrŽ-stabilisation spectrale}

La fonction sur $\tG(\adef)$:
$$k_{spec}^\Tr(\tx)=K(\tx,\tx)\,-\sum_{\gamma\in P(F)\bsl G(F)}
\int_{u\in[\U]}\widehat\tau_\P(\HP(\gamma \tx)-\Tr) K(u\gamma \tx, \tx)du\ptf$$
passe au quotient sur $[\tG]$ et on a
$$J_{spec}^\Tr(f)=\sum_\kappa\J_{spec}^\Tr(f,\kappa)\ptf$$
On pose $$f^x(y)=f(x yx\mun)\comm{et}\pi^x(y)=\pi(xyx\mun)$$ et on a
$$\J_{disc}(f,\kappa)=\int_{Q_F}J_{disc}(f^\tx)\kappa(\det \tx)dx=\sum_\pi a^\kappa(\pi)\trace \pi(f)$$ 
avec
$$a^\kappa(\pi)=\int_{Q_F}a(\pi^\tx)\kappa(\det \tx)dx\ptf$$
On a utilisŽ que les $\kappa$ sont d'ordre 2.
On verra que mme pour des reprŽsentations du spectre discret
il n'est pas toujours vrai que $a^\kappa(\pi)$ soit un entier positif.

\begin{proposition}
$$J_{cont}^\Tr(f)=\J_{cont}^\Tr(f,1)$$
\cad que les $\J_{cont}^\Tr(f,\kappa)$ sont tous nuls si $\kappa\ne1$.
\end{proposition}

\begin{proof}
Tout $x\in\tG(\adef)$ peut s'Žcrire de faon unique comme
$$x=ax_1\comm{avec} x_1\in\G(\adef) \comm{et} 
a=\begin{pmatrix} \det(x)&0\cr0&1\cr\end{pmatrix}$$
et on dŽfinit une fonction $\tf$ sur $\tG(\adef)$ en posant:$$\tf(x)=\th\big(\det(x)\big)f(x_1)$$
o $\th$ est une fonction sur $\adef^\times$ de support assez petit de sorte que
les deux conditions:
$\xi\in F^\times$ et $\th(\xi)\ne0$ impliquent $\xi=1$. On impose de plus que
$\th(1)=1$.
La fonction $\tf$ permet de dŽfinir un noyau $\tK$ pour $\tG$:
$$\tK(\tx,\ty)=\sum_{\tgamma\in\tG(F)}\tf(\tx\mun\tgamma \ty)$$
et  
$$\tk_{spec}^\Tr(\tx)=\tK(\tx,\tx)\,-\!\!\!\!\!\!\!\!
\sum_{\tgamma\in \tP(F)\bsl \tG(F)}
\int_{u\in \U(\adef)}
\widehat\tau_\tP(\HtP(\tgamma \tx)-\Tr) \tK(u\tgamma \tx, \tx)du\ptf$$
On observe que $\tK(\tx,\tx)=K(\tx,\tx)$ et de mme
$ \tK(u\tgamma \tx, \tx) =K(u\tgamma \tx, \tx)$ et donc
$$k_{spec}^\Tr(\tx)=\tk_{spec}^\Tr(\tx)$$ ce qui implique
$$\J_{spec}^\Tr(f,\kappa)=\int_{[\tG]}\tk_{spec}^\Tr(\tx)\kappa(\tx)\dtx$$
o le second membre est le c™tŽ spectral de la formule des traces pour $\tG$ tordue par le
caractre $\kappa$ pour la fonction $\tf$ (cf. \cite{LW}). La contribution du spectre continu 
est, ˆ un scalaire prs, une intŽgrale de la forme
$$\int_{\tlambda\in\widetilde\spec} 
\trace \big({\mathcal M}^\Tr(\tlambda)\tpi_\tlambda(\tf)\,.\, I_\kappa\big)
d\tlambda$$
o $I_\kappa$ est l'opŽrateur envoyant
l'espace $V_\tlambda$ de $\tpi_\tlambda$ sur 
celui de $\tpi_\tlambda\otimes\kappa$.
Il convient d'observer que $${\mathcal M}^\Tr(\tlambda)\tpi_\tlambda(\tf)$$ 
est un opŽrateur dans l'espace de $\pi_\tlambda$ alors que $I_\kappa$
Žchange l'espace de $\pi_\tlambda$ et celui de $\pi_\tlambda\otimes\kappa$.
La trace est nulle sauf si $$\tpi_\tlambda\simeq \tpi_\tlambda\otimes\kappa$$
ce qui ne peut intervenir que pour un ensemble de $\tlambda$
de mesure nulle.
\end{proof}

\Chapter{Stabilisation}

\Section{La \ts-stabilisation}

La \ts-stabilisation de la formule des traces est, par dŽfinition,
ce que l'on obtient en composant la prŽ-stabilisation de la formule des traces non-invariante
avec un transfert endoscopique pondŽrŽ. Un tel transfert n'a pas ŽtŽ Žtabli en gŽnŽral; toutefois
on dispose dans la littŽrature du lemme fondamental pondŽrŽ et l'existence
du transfert pondŽrŽ gŽomŽtrique est facile ˆ Žtablir pour les fonctions ˆ support rŽgulier;
le cas gŽnŽral reste conjectural.
Dans notre cas le transfert pondŽrŽ est disponible puisque pour la donnŽe endoscopique 
principale $\EC=\{G,1\}$ le transfert est l'identitŽ $f^\EC=f$ 
et que par ailleurs il n'y a pas de distributions pondŽrŽes ˆ considŽrer pour les 
autres groupes endoscopiques: ce sont des tores.
Pour la donnŽe endoscopique  $\EC=\{\G,1\}$ on voit que
$$\J_{geom}^\Tr(f,1)\comm{et}\J_{spec}^\Tr(f,1)$$
donnent le c™tŽ gŽomŽtrique et le c™tŽ spectral de la formule des traces \ts-stable pour $\G=SL(2)$.
Les termes elliptiques dans  $\J_{geom}^\Tr(f,1)$ et les termes
discrets dans $\J_{spec}^\Tr(f,1)$ sont $\tG(\adef)$-invariants et donc stables. 
Mais les termes hyperboliques ou unipotents
et les termes continus ne le sont pas strictement: de fait ils ne sont mme pas 
$\G(\adef)$-invariants. Mais ils le sont si on se limite ˆ la conjugaison sous $\M(\overline{\adef})$
pour les ŽlŽments dans $\P(\adef)$, ce qui est d'ailleurs
une conjugaison triviale pour les ŽlŽments hyperboliques; mais elle est non triviale pour les unipotents.

\Section{Termes instables}
On supposera dans cette section que $\kappa=\vef$ o $E$ est un corps, extension quadratique
sŽparable du corps global $F$. 
Il rŽsulte de ce qui prŽcde (cf. \ref{prestab}) que
$$J_{geom}^\Tr(f,\vef)=\J_E(f,\vef)+ \J_{unip}(f,\vef)$$
que l'on notera $$\J_{geom}(f,\vef)$$ puisqu'il est indŽpendant de $\Tr$ et
de mme 
$$\\J_{spec}^{\Tr}(f,\vef)=\J_{spec}(f,\vef)$$ est indŽpendant de $\Tr$.
On a donc
$$\J_{geom}(f,\vef)=\J_{spec}(f,\vef)$$
et
$$\J_{geom}(f,\vef)=\J_E(f,\vef)+ \J_{unip}(f,\vef)\ptf$$

Le transfert local \ref{trans} et le lemme fondamental \ref{fond}
fournissent un transfert global encore notŽ $f\mapsto f^\EC$.
Toutefois les mesures de Haar locales et globales sont ˆ comparer.

\begin{lemma}\label{transfunip}
 Pour $\EC=(\Tef,\vef)$ avec
 $\vef\ne1$ le terme $ \J_{unip}(f,\vef)$
fournit la contribution de $\Z(F)$
 ˆ l'expression gŽomŽtrique de la formule des traces pour le
tore $\Tef$ attachŽ ˆ l'extension quadratique sŽparable $E/F$ associŽe ˆ $\vef$.
$$\J_{unip}(f,\vef)=\frac{1}{2}\sum_zC(0,\vef,z)=\frac{1}{2}\sum_zf^\EC(z)\ptf$$
\end{lemma}
\begin{proof}
D'aprs \ref{unipz}, compte tenu de \ref{trans} et \ref{fond},
et en utilisant des mesure de Tamagawa locales non normalisŽes on a
pour $t\in\Tef(\adef)$:
$$C(0,\vef,z)
=\lim_{t\to z}\orb^\EC(t,f)=\lim_{t\to z}f^\EC(t)
=f^\EC(z)\ptf$$
\end{proof}
\begin{lemma}\label{stabell}
$$J_{geom}(f,\vef)=\frac{1}{2}\sum_{\gamma\in\Tef(F)}f^\EC(\gamma)$$
pour la fonction $f^{\EC}$ avec $\EC=(\Tef,\vef)$
\end{lemma}
\begin{proof} D'aprs  \ref{ft} et \ref{contrib-ell}
$$\J_E(f,\vef)=\frac{1}{2}\sum_{\gamma\in E^\star}\orb^{\vef}(\gamma, f)
=\frac{1}{2}\sum_{\gamma\in E^\star} f^\EC(\gamma)
\ptf$$
Par ailleurs  \ref{transfunip} montre que
$$ \J_{unip}(f,\kappa)=\frac{1}{2}\sum_{z\in \Z(F)}\orb^{\vef}(z{\bs\nu}, f)
=\frac{1}{2}\sum_z f^\EC(z)\ptf$$
\end{proof}
Le second membre de \ref{stabell}
est, ˆ un scalaire prs, le c™tŽ gŽomŽtrique de la formule des traces 
(en l'occurence de la formule de Poisson) pour $\Tef$.
C'est la contribution endoscopique pour
le groupe endoscopique $\Tef$:
$$J_{geom}(f,\vef)
=\iota(\EC)J_{geom}^\EC(f^\EC)$$
La formule de Poisson s'Žcrit
$$J^\EC(f^\EC):=\sum_{\eta\in\Tef(F)}f^\EC(\eta)=\frac{1}{\tau(\Tef)}
\sum_{\theta\in\widehat\Tef} \theta(f^\EC)
$$
o $\widehat\Tef$ est le dual de Pontyagin de $\Tef(F)\bsl\Tef(\adef)$.
On a donc
$$J_{geom}(f,\vef)=\frac{1}{4}\sum_{\th\in\widehat\Tef} \trace\pi_{E,\th}^+(f)-\trace\pi_{E,\th}^-(f)\ptf$$
L'ŽgalitŽ
$$J_{geom}(f,\vef)=J_{spec}(f,\vef)$$
montre donc que 
$$J_{spec}(f,\vef)=J_{cusp}(f,\vef)+J_{cont}(f,\vef)=
\frac{1}{4}\sum_{\th\in\widehat\Tef} \trace\pi_{E,\th}^+(f)-\trace\pi_{E,\th}^-(f)\ptf$$

On rappelle que $\pi_{E,\th}^+\simeq\pi_{E,\th\mun}^+$ et 
$\pi_{E,\th}^-\simeq\pi_{E,\th\mun}^-$. Donc lorsque $\th\ne\th\mun$ 
de telles reprŽsentations apparaissent dans   $J_{spec}(f,\kappa)$
pour  $\kappa\in \{1,\vef\}$.
La  $\kappa$-multiplicitŽ de $\pi_{E,\th}^\pm$
vaut $1/2$ pour $\kappa=1$ et
$\pm1/2$  pour $\kappa=\vef$ ,

Si $\th=\th\mun$ mais $\th\ne1$ alors $\th$  est d'ordre 2 et dŽfinit
une extension quadratique $L$ de $E$ et $L/F$ est une extension bi-quadratique.
Dans ce cas $\pi_{E,\th}^+\oplus\pi_{E,\th}^-$ se dŽcompose en une somme de 4 reprŽsentations
inŽquivalentes et on notera $\pi_L$ le composant ayant un modle de Whittaker.
Il interviendra pour 3 extensions quadratiques $E_i$ entre $F$ et $L$ 
soit donc au total avec multiplicitŽ $3/4$ dans la partie instable
et il interviendra avec multiplicitŽ moyenne $1/4$ dans la partie stable $J_{disc}(f,1)$
et donc au total 1 dans $J_{disc}(f)$. Les autres apparaissent
avec multiplicitŽ $-1/4=(1/4)-2(1/4)$
dans la partie instable et $+1/4$ dans la partie stable
et donc n'apparaissent pas dans $J_{disc}(f)$.

Enfin, le terme correspondant ˆ $\th=1$  est la contribution de $J_{cont}(f,\vef)$,
\cad un des termes discrets issus de la troncature du spectre continu.

\Section{La formule des traces \ts-stable}
Si $\kappa=\vef$ est non trivial et $\EC=\{\Tef,\vef\}$ on pose
 $H=\Tef$, $ \iota(\EC)=\tau(\Tef)/4=1/2$ et on a montrŽ
$$SJ_{geom}^{\H}(f^\EC)=\sum_{\eta\in\Tef(F)}f^\EC(\eta)
\com{et}SJ_{spec}^{\H}(f^\EC)=\frac{1}{\tau(\Tef)}
\sum_{\theta\in\widehat\Tef} \theta(f^\EC)
\ptf$$
On a alors
$$J^{\G,\Tr}_\bullet(f,\vef)=J^{\G}_\bullet(f,\vef)=\iota(\EC)SJ_\bullet^{\H}(f^\EC)\ptf$$
Maintenant si $\EC=\{SL(2),1\}$ on a $f^\EC=f$; on pose $\iota(\EC)=1$ et
$$SJ_\bullet^{\H,\Tr}(f^\EC)=J^{\G,\Tr}_\bullet(f,1)\ptf$$
Puisque $$J^{\G,\Tr}_\bullet(f)=\sum_\kappa \J^\Tr_\bullet(f,\kappa)$$
la formule des traces pour $\G$ peut s'Žcrire comme une combinaison linŽaire
d'expressions, que nous appellerons \ts-stables, sur les groupes endoscopiques 
$$J^{\G,\Tr}_\bullet(f)=
\sum_\EC\iota(\EC)SJ_\bullet^{\H,\Tr}(f^\EC)$$
o $\bullet={geom}$ ou ${spec}$.

\Section{Sur certaines distributions non-invariantes}

Pour la stabilisation des distributions non invariantes par conjugaison
nous aurons besoin d'une distribution non invariante auxiliaire: pour $m\in\M(F)$ on pose
pour $g\in\G(\adef)$:
$$\Psi(f,m,g)=\delta(m)^{1/2}\int_K\int_{\U(\adef)}\HP(kg)f(k\mun muk)\,du\,dk\ptf$$
Maintenant si $m\in\M(\adef)^\star$ on a
$$\Psi(f,m,g)=\Delta_\M(m)\int_{M(\adef)\bsl\G(\adef)} 
f(x\mun mx)\big(\HP(xg)-\HP(x)\big)\,\ddx\ptf$$
En effet si $x=m_xu_xk_x$ alors $\HP(xg)-\HP(x)=\HP(k_xg)$ donc
$$\Psi(f,m,g)=\Delta_\M(m)\int_{M(\adef)\bsl\G(\adef)} f(x\mun mx)\HP(k_xg)\,\ddx$$
et l'ŽgalitŽ cherchŽe rŽsulte du changement de variable $u\mapsto u'$ o $x=uk$ et
$$u'=m\mun u\mun m u\ptf$$

Du c™tŽ gŽomŽtrique les ŽlŽments hyperboliques et unipotents  
donnent naissance ˆ des distributions 
non-invariantes appelŽes intŽgrales orbitales pondŽrŽes notŽes $J_{m}(f)$
pour $m\in\M(\adef)$. Pour $m\in\M(\adef)^\star$ on pose
$$J_{m}(f)=\Delta_\M(m)\int_{M(\adef)\bsl\G(\adef)}f(x\mun mx)\tv(x)\,\ddx\ptf$$

\begin{lemma} \label{h-inv} On pose $w(m):=w mw\mun$.
Pour $m\in\M(\adef)^\star$ on a
$$J_ m(f^{g})-J_ m(f)=-\big(\Psi(f,m,g)+\Psi(f,w(m),g)\big)\ptf$$
\end{lemma}
\begin{proof}
 $$J_ m(f^{g})-J_ m(f)=\Delta_\M(m)\int_{M(\adef)\bsl\G(\adef)}\Big( f(g x\mun mxg\mun)-f(x\mun mx)\Big)
\tv(x)\,\ddx$$
soit encore
$$J_ m(f^{g})-J_ m(f)=\Delta_\M(m)\int_{M(\adef)\bsl\G(\adef)} f(x\mun mx)\big(\tv(xg)-\tv(x)\big)\,\ddx$$
On a vu que 
$$\tv(x)=\ve-\big(\HP(x)+\HP(wx)\big)$$ 
o $\ve=0$ ou 1 suivant que l'on travaille avec un corps de nombres 
ou un corps de fonctions et donc
$$J_ m(f^{g})-J_ m(f)=-\big(\Psi(f,m,g)+\Psi(f,w(m),g)\big)\ptf$$
\end{proof}

On considre l'unipotent
$$u=\begin{pmatrix}1&n_u\cr 0&1\cr\end{pmatrix}\in\U(\adef)$$
et on pose $$\ll(x)=\log_{q_F}(\vert n_u\vert)\com{si}x=k\mun uk\com{avec}k\in K
\com{et} n_u\in\adef^\times\ptf$$  
Cette fonction est bien dŽfinie car le normalisateur dans $\tG(\adef)$
de l'ensemble des $u$ avec $n_u\in\adef^\times$ est $\tP(\adef)$ et donc
$n_u$ est bien dŽfini ˆ une unitŽ prs. Pour $z\in\Z(\adef)$ on dŽfinit 
$$J_z(f)=\int_{\tZ(\adef)\U(\adef)\bsl\tG(\adef)}f(x\mun z{\bs\nu}x)\ll(\nu(x))\ddx
\com{o}
{\bs\nu}=\begin{pmatrix}1&1\cr 0&1\cr\end{pmatrix}$$ 
  et $\nu(x)=x\mun {\bs\nu} x$.
\begin{lemma}\label{u-inv}
$$J_ z(f^{g})-J_ z(f)=-2\Psi(f,z,g)\ptf$$
\end{lemma}
\begin{proof}
$$J_ z(f^{g})=\int_{\U(\adef)\bsl\G(\adef)}f(g x\mun z{\bs\nu}xg\mun)\ll(\nu(x))\,\ddx $$
soit encore
$$J_ z(f^{g})=\int_{\U(\adef)\bsl\G(\adef)}f(x\mun z {\bs\nu} x)
\ll(\nu({xg}))\,\ddx $$ 
et en d'observant que 
$$\HP(xg)=\HP(kg)+\HP(x)$$ 
et que pour $m\in\tM(\adef)$:
$$\ll(m\mun um)=\log_{q_F}(\vert m^{-\alpha}n_u\vert)\com{soit encore}
\ll(m\mun um)=\ll(u)-2\HP(m)$$
on obtient
$$\ll(\nu(xg))=\ll(\nu(x))-2\HP(kg)\ptf$$
\end{proof}

On note $A(g)$ l'opŽrateur de multiplication
par $\HP(kg)$  dans $L^2(K)$:
$$\big(A(g)\vf\big)(k)=\HP(kg)\vf(k)\ptf$$
\begin{lemma}\label{tra}
$$\trace A(g) \pi_\lambda(f)=\int_{\M(\adef)}\Psi(f,m,g)\,m^\lambda\phantom{}\,dm\ptf$$
\end{lemma}
\begin{proof} Il suffit d'observer que le noyau $K(k_1,k_2)$ sur $K\times K$ 
reprŽsentant l'opŽrateur $\pi_\lambda(f)$ dans $L^2(K)$ a pour formule
$$K(k_1,k_2)=\int_{\M(\adef)}\int_{U(\adef)} f(k_1\mun m u k_2)
m^\lambda\delta_\P(m)^{1/2} dm du \ptf$$
\end{proof}
\begin{lemma}\label{bgl}
$$\pi_{w\lambda}(g)\mathbf M'(\lambda) \pi_\lambda(g)\mun
-\mathbf M'(\lambda)=\mathbf M(\lambda) A(g)+A(g)\mathbf M(\lambda)
\ptf$$
\end{lemma}
\begin{proof}
On observe que si $\lambda=sr+\chi$
on a
$$\frac{d}{ds}\pi_\lambda(g)=A(g)\pi_\lambda(g)
\com{et}
\pi_{w\lambda}(g)\mathbf M(\lambda) =\mathbf M(\lambda)\pi_\lambda(g)\ptf$$
En dŽrivant en $s$ la seconde expression on obtient
$$-A(g)\pi_{w\lambda}(g)\mathbf M(\lambda)+\pi_{w\lambda}(g)\mathbf M'(\lambda) =
\mathbf M'(\lambda)\pi_\lambda(g)+\mathbf M(\lambda)A(g)\pi_\lambda(g)
$$
\end{proof}

Du c™tŽ spectral apparaissent des distributions non-invariantes 
appelŽes caractres pondŽrŽs:
$$f\mapsto \Phi(\lambda,f)=\trace \mathcal N(\lambda)\pi_\lambda(f)$$
o $\mathcal N(\lambda)$ est la dŽrivŽe logarithmique de l'opŽrateur
d'entrelacement normalisŽ $\mathbf R(\lambda)$:
$$\mathcal N(\lambda)=\mathbf R(\lambda)\mun\mathbf R'(\lambda)\ptf$$

\begin{lemma}\label{sp-inv}
On a l'identitŽ de distributions suivante:
$$ \Phi(\lambda,f^{g})-\Phi(\lambda,f)=\trace  A(g)\big(\pi_\lambda(f)+\pi_{w\lambda}(f)\big)\ptf$$
\end{lemma}
\begin{proof}  Pour ce calcul il revient au mme de remplacer 
$\mathbf R$ par $\mathbf M$ dans la dŽfinition de $\mathcal N$ puisque
les dŽrivŽes logarithmiques du facteur de normalisation se compensent. Par dŽfinition
$$ \Phi(\lambda,f^{g})-\Phi(\lambda,f)=
\trace \mathcal N(\lambda)\Big(\pi_\lambda(f^{g})-\pi_\lambda(f)\Big)$$
\cad 
$$\trace\Big( \mathcal N(\lambda)\pi_\lambda(g)\mun\pi_\lambda(f)\pi_\lambda(g)-
 \mathcal N(\lambda)\pi_\lambda(f)\Big)$$
soit encore 
$$\trace\Big( \pi_\lambda(g) \mathcal  N(\lambda)\pi_\lambda(g)\mun- 
\mathcal  N(\lambda)\Big)\pi_\lambda(f)$$
On observe que $\mathbf M(\lambda) \pi_\lambda(g)= \pi_{w\lambda}(g)\mathbf M(\lambda)$
et l'assertion cherchŽe est alors consŽquence immŽdiate du lemme \ref{bgl}.
\end{proof}

On pose
$$I_\lambda(f)=\trace \big(\mathcal N(\lambda)\pi_\lambda(f)\big)+
\int_{\M(\adef)}J_ m(f)\,m^\lambda\phantom{}\,dm$$
et si $\widehat{\M(\adef)}$ est le dual de Pontryagin de $\M(\adef)$ on pose
$$I_m(f)=\phantom{}\,J_ m(f)+\int_{\widehat{\M(\adef)}}
\trace \big(\mathcal N(\lambda)\pi_\lambda(f)\big)m^{-\lambda}\,d\lambda\ptf
$$

\begin{proposition}\label{locinv} Les distributions $I_\lambda(f)$ et $I_m(f)$
sont stablement invariantes: $$I_\lambda(f^g)=I_\lambda(f)$$ pour tout $g\in\tG(\adef)$.
Il en est de mme pour $I_m$.
\end{proposition}

\begin{proof} 
Il rŽsulte de \ref{h-inv} et \ref{u-inv} que
$$\int_{\M(\adef)}\Big(J_ m(f^{g})-J_ m(f)\Big)\,m^\lambda\phantom{}\,dm
=-\int_{\M(\adef)}\Psi_ m(f)\, \big(m^\lambda+m^{w\lambda}\big)\phantom{}\,dm\ptf$$
Par ailleurs \ref{tra} et \ref{sp-inv} montrent que
$$\Phi(\lambda,f^{g})-\Phi(\lambda,f)=\trace  A(g)\big(\pi_\lambda(f)+\pi_{w\lambda}(f)\big)$$
est Žgal ˆ
$$\int_{\M(\adef)}\Psi_ m(f)\, \big(m^\lambda+m^{w\lambda}\big)\phantom{}\,dm\ptf$$
\end{proof}

\Section{La forme stablement invariante}

Une forme invariante pour la formule des traces pour $SL(2)$ est obtenue
en faisant passer du c™tŽ gŽomŽtrique la partie non-invariante du c™tŽ spectral. 
Les termes non-invariant du c™tŽ gŽomŽtrique sont les contributions hyperbolique
et le terme correspondant ˆ $\kappa=1$ de la prŽ-stabilisation de la contribution unipotente.
On rappelle que l'on a posŽ
$$J_{\xi}(f)=\int_{M(\adef)\bsl\G(\adef)} f(x\mun\xi x)\tv(x)\ddx $$
et pour $\xi=z\in\Z(F)$ on pose $$J_z(f)=
\int_{K}\int_{\U(\adef)}f_v(k\mun zuk)\lambda(u)\,dk\,du\ptf$$
Les termes non invariant du c™tŽ spectral font intervenir la dŽrivŽe logarithmique
de l'opŽrateur d'entrelacement.
Mais l'opŽrateur d'entrelacement ${\mathbf M}(s,\lambda)$ 
peut s'Žcrire comme le produit 
$${\mathbf M}(s,\lambda)=m(\lambda){\mathbf R}(\lambda)$$
d'un facteur scalaire $m(\lambda)$ quotient de fonctions $L$ de Hecke 
et d'un opŽrateur appelŽ opŽrateur d'entrelacement normalisŽ notŽ ${\mathbf R}(\lambda)$.
Cela permet de dŽcomposer sa dŽrivŽe logarithmique
 ${\mathcal M}(\lambda)$ en une somme d'un opŽrateur scalaire et d'un autre opŽrateur:
$${\mathcal M}(\lambda)=\frac{m'(\lambda)}{m(\lambda)}+{\mathcal N}(\lambda)$$
o ${\mathcal N}(\lambda)$ est la dŽrivŽe logarithmique de $R(\lambda)$.
La clef du passage ˆ la forme invariante est le 
\begin{lemma}\label{inv}
$$\int_\speccont\trace\Big({\mathcal N}(\lambda)\pi_\lambda(f)\Big)\,d\lambda=
\sum_{\xi\in\M(F)}N_\xi(f)$$
et
$$I_\xi(f)=J_\xi(f)+N_\xi(f)$$
est une distribution invariante.
\end{lemma}
\begin{proof}
Cela rŽsulte de la  proposition \ref{locinv}.
\end{proof}
Le c™tŽ gŽomŽtrique s'Žcrit
$$I_{geom}(f)=I_{ell}(f)+I_{p}(f)$$
o
$$I_{ell}(f)=J_{ell}(f)\comm{et} I_{p}(f)=\sum_{\xi}I_\xi(f)$$
Le c™tŽ spectral de la formule des traces invariante s'Žcrit
$$I_{spec}(f)=I_{disc}(f)+I_{cont}(f)$$
o
$$I_{disc}(f)=J_{disc}(f)\com{et}I_{cont}=\frac{-1}{2}
\int_\speccont \frac{m'(\lambda)}{m(\lambda)}\trace\pi_\lambda(f)\,d\lambda\ptf$$
La stabilisation invariante s'Žcrit:
\begin{theorem} On a les identitŽs
$$I_{geom}(f)=\sum_\EC\iota(\EC)SI_{geom}^{\EC}(f^\EC)
\com{et}I_{spec}(f)=\sum_\EC\iota(\EC)SI_{spec}^{\EC}(f^\EC)
$$
avec 
$$SI^\EC_{ell}=J^\EC_{ell}\com{,}SI^\EC_{disc}=J^\EC_{disc}
\com{,}SI^\G_{p}=I_{p}\com{et}SI^\G_{cont}=I_{cont}$$
et
$$I^{\G}_{geom}(f)=I^{\G}_{spec}(f)\com{et}SI^{\EC}_{geom}(f)=SI^{\EC}_{spec}(f)\ptf$$
\end{theorem}

\begin{remark}{\rm  
Les distributions qui apparaissent dans 
$J^{\G}_\bullet(f,1)$ sont les distributions
qui apparaissent dans la formule des traces non-invariante pour $\tG$.
De mme, les distributions qui apparaissent dans 
$I^{\G}_\bullet$ sont les distributions
qui apparaissent dans la formule des traces stable pour $\tG$.}
\end{remark}

\Chapter{Formes intŽrieures: UnitŽs des algbres de quaternions}

\Section{ThŽorie locale}

Soit $F$ un corps local et soit $\HM$ l'algbre de quaternions (non dŽoloyŽe)
de centre $F$.
On note $\tG'(F)=\HM^\times$ son groupe multiplicatif 
et $\HM^1=\G'(F)$ le sous-groupe des ŽlŽments de norme 1.

Soit $E$ un corps extension quadratique sŽparable de $F$. On note 
$\alpha\mapsto \overline\alpha$ l'action du groupe de Galois de $E/F$.
Alors on peut identifier $\HM$ ˆ l'ensemble des matrices dans $M(2,E)$ de la forme
$$a=\begin{pmatrix} \alpha&\overline\beta{\mathfrak d}\cr\beta&\overline\alpha\cr\end{pmatrix}$$
{avec}$$\alpha\in E\com{,} \beta\in E
\com{,} \mathfrak d\in F^\times\com{mais} \mathfrak d\notin N_{E/F}E^\times\ptf$$
 On a plongŽ $E$
comme le sous-ensemble des matrices diagonales dans $\HM$ 
et l'action de l'ŽlŽment non trivial du groupe de Galois est induit par la conjugaison par 
$$\bs w_E=\begin{pmatrix}0&\mathfrak d\cr 1&0\cr\end{pmatrix}\ptf$$
Comme $\mathfrak d$ n'est pas une norme on a $\vef(\mathfrak d)=-1$ d'o le
\begin{lemma}\label{signb}
$$\vef\big(\det(\bs w_E)\big)=\vef(-\mathfrak d)=-\vef(-1)\ptf$$
\end{lemma}

\Subsection{Correspondance de Jacquet-Langlands}

On a une bijection entre les classes de conjugaison d'ŽlŽments de $\HM^\times$
et les classes de conjugaison d'ŽlŽments elliptiques ou insŽparables de $\tG(F)$.
Cette bijection est duale de
la bijection, appelŽe correspondance de Jacquet-Langlands et notŽe $\JL$:
$$\tpi'\mapsto \JL(\tpi')$$ entre l'ensemble des reprŽsentations (unitaires) irrŽductibles
 de $\tG'(F)=\HM^\times$ dans l'ensemble des sŽries discrtes
de $\tG(F)=GL(2,F)$ et cette bijection est compatibles aux torsions par des caractres
et ˆ la bijection entre classes de conjugaison. De plus
on a une ŽgalitŽ au signe prs (le signe de Kottwitz) pour la restriction
de caractres des reprŽsentations aux ŽlŽments elliptiques rŽguliers (voir le chapitre 15 de \cite{JL})
$$\trace \tpi'(t)=-\trace \JL(\tpi')(t)\ptf$$

\Subsection{DŽcomposition par restriction ˆ $\G'(F)$}

Pour le corps des rŽels la norme 
des quaternions n'est pas surjective (elle ne prend que des valeurs positives).
Un caractre de $\HM^\times$ correspond ˆ deux caractres de $GL(2,\RM)$
qui diffrent par le caractre signe et le groupe $X(\tpi')$ est toujours trivial. Les reprŽsentations
de $\HM^\times$ restent donc irrŽductibles par restriction ˆ $G'(\RM)$.

Si $F$ est non archimŽdien et si $\tpi'$ est une reprŽsentation admissible
irrŽductible (unitaire) de $\tG'(F)$ correspondant ˆ $\tpi$, le groupe $X(\tpi')$ des caractres
tels que \hbox{$\tpi'\otimes\bs\chi'\simeq\tpi'$} est isomorphe ˆ $X(\tpi)$; il
est donc de cardinal 1, 2 ou 4. Cependant, on verra ci-dessous que l'algbre d'entrelacement
de la restriction $\pi'$ de $\tpi'$ ˆ $\G'(F)$ n'est pas toujours isomorphe ˆ celle
la restriction de $\tpi$ ˆ $\G(F)$ quoique de mme dimension. 
Si $\tpi'\otimes\vef\simeq\tpi'$ il existe un caractre $\tth$ de $E^\times$
tel que $\JL(\tpi')=\tpi_{E,\tth}$ et $\tpi'$ sera notŽe $\tpi'(\tth)$.
Notons $\G'(F)_E$ le noyau de $\vef\circ\det$.
Par restriction ˆ $\G'(F)_E$ la reprŽsentation $\tpi'(\tth)$ se dŽcompose en somme de 
deux reprŽsentations irrŽductibles inŽquivalentes que nous noterons $\tpi_{E,\tth}^{'-+}$
et $\tpi_{E,\tth}^{'-}\ptf$
Pour $\tG(F)_E$ l'existence de modles de Whittaker permettait
de distinguer entre $\tpi_{E,\tth}^+$ et $\tpi_{E,\tth}^-$\,\,. Pour $\G'(F)_E$ 
et il ne semble pas y avoir de faon simple pour distinguer les $\tpi_{E,\tth}^{'\pm}\ptf$

\Subsection{Transferts gŽomŽtrique et spectral}

On utilise le mme facteur de transfert 
$$\Delta^\EC(t,t')=\lambda(E/F,\psi)\mun
\vef\Big(\frac{\gamma-\overline\gamma}{\tau-\overline\tau}\Big)
{\vert{\gamma-\overline\gamma}\vert}$$ 
\begin{lemma} 
Pour $\EC=(\Tef,\vef)$ il existe une fonction 
$$f^{\EC}\in \ctyc\big(\Tef(F)\big)$$
appelŽe transfert de $f$ pour la donnŽe endoscopique $\EC$ telle que
$$f^{\EC}(t')=\Delta^\EC(t,t')\orb^\kappa(t,f)\ptf$$
\end{lemma}

\begin{proof}Si $F$ est non archimŽdien $\orb(t,f)=\orb(\overline t,f)$
pour $t$ assez voisin de 1 dans dans $\G'(F)$ et $\orb^\kappa(t,f)=0$; la lissitŽ 
de $f^\EC$ en rŽsulte.
Si $F=\RM$ on peut identifier $\T_{\CM/\RM}(\RM)$
avec le groupe des nombre complexes de module 1 et
$$\ve_{\CM/\RM}\Big(\frac{\gamma-\overline\gamma}{\tau-\overline\tau}\Big)
{\vert{\gamma-\overline\gamma}\vert}=c\,(e^{i\vf}-e^{-i\vf})$$ 
(pour un $\vf\in\RM$ et $c$ de module 1)
qui est lisse de mme que l'intŽgrale orbitale, pour les mesures de Haar normalisŽes,
les groupe $\G'(\RM)$ et $\T_{\CM/\RM}$ Žtant compacts.
\end{proof}

Compte tenu de \ref{signb} le transfert spectral est alors de la forme
$$\theta^\EC(t)
= \lambda(E/F,\psi)
\vef\Big(\frac{\gamma-\overline\gamma}{\tau-\overline\tau}\Big)
\frac{\big(\th(\gamma)-\th(\gamma\mun)\big)}{\vert\gamma-\gamma\mun\vert}\ptf$$
Si $\th$ la restriction de $\tth$ aux ŽlŽments de norme 1.
On note $\pi'(\th)^\pm$ les restrictions de $\tpi_{E,\tth}^{'\pm}$ ˆ $\G'(F)$. 
On voit comme dans \ref{diffa}
que la diffŽrence des caractres $$\trace\pi_{E,\th}^{'+}(t)-\trace\pi_{E,\th}^{'-}(t)$$ 
est nulle sur les tores non isomorphes ˆ $\Tef$.
Il rŽsulte de   la stabilisation des formules des traces pour 
 $\G'$ sur des corps globaux, esquissŽe ci-dessous, que pour 
$t\in\Tef(F)$ rŽgulier on peut choisir les $\tpi_{E,\tth}^{'\pm}$ de sorte que 
$$\theta^\EC(t)=\trace\pi_{E,\th}^{'+}(t)-\trace\pi_{E,\th}^{'-}(t)\ptf$$
On sait dŽjˆ que les $\pi'(\th)^\pm$ sont irrŽductibles inŽquivalentes si $\th$ n'est pas d'ordre 2.
Par contre, si $\th$ est d'ordre 2 non trivial, on voit que
$$\trace\pi_{E,\th}^{'+}(t)\equiv\trace\pi_{E,\th}^{'-}(t)$$
pout tout $t\in\G'(F)$.
Les deux reprŽsentations sont donc Žquivalentes et l'algbre de d'entrelacement 
de la restriction de $\tpi'(\tth)$ ˆ $\G'(F)$
est isomorphe ˆ l'algbre $M(2,\CM)$ des matrice $2\times 2$ complexes
alors que celle de la restriction de $\tpi_{E,\tth}$ ˆ $\G(F)$ Žtait $\CM^4$:
les algbres d'entrelacement sont de mme dimension mais non isomorphes.

\Section{ThŽorie globale}

Soit $F$ un corps global, $\HM$ une algbre de quaternions sur $F$ et soit
$\G'$ le groupe des ŽlŽments de norme 1.
On note $I_\gamma$ le centralisateur de $\gamma$, 
$\Gamma_F$ un ensemble de reprŽsentants des classes de
conjugaison dans $\G'(F)$ et $\Pi$ un ensemble de reprŽsentants des classes d'Žquivalence de
reprŽsentations unitaires irrŽductibles de $\G'(\adef)$ qui interviennent dans la dŽcomposition spectrale.
de $L^2(\G'(F)\bsl\G'(\adef))$.
Les classes de conjugaison sont de deux sortes : 
Il y a l'ŽlŽment neutre et d'autre part le ŽlŽments semi-simples rŽguliers dont le centralisateur
est un tore anisotrope. Les ŽlŽments hyperboliques ou unipotents n'apparaissent
pas dans $\G'(F)$ mme si ils apparaissent localement aux places $v$ o $\G'(F_v)$ est dŽployŽ
\cad partout sauf en un nombre fini pair, mais non nul, de places.
La compacitŽ de $\G'(F)\bsl\G'(\adef)$ implique que le spectre
de $L^2(\G'(F)\bsl\G'(\adef))$ est purement discret et on note $m(\pi)$ la multiplicitŽ
de $\pi$ dans le spectre discret.
Pour $\G'$ la formule des traces est donc ŽlŽmentaire:
$$\sum_{\gamma\in\Gamma_F} \vol\big(I_\gamma(F)\bsl I_\gamma(\adef)\big)
\orb(\gamma,f)
=\sum_{\pi'\in\Pi'} m(\pi')\trace(\pi'(f)\ptf$$
La prŽ-stabilisation et la stabilisation s'Žtablissent comme dans le cas dŽployŽ
sauf qu'il n'y a pas ˆ distinguer entre la \ts-stabilisation et la stabilisation invariante.
Le seul phŽnomne nouveau est que les multiplicitŽs peuvent ne pas tre plus grandes
que 1.  C'est clair pour certaines reprŽsentations instables mais c'est aussi le cas
pour des reprŽsentations stables. En effet on a vu ci-dessus que certains
composants des $L$-paquets peuvent avoir localement multiplicitŽ 2.

\Section{Sur les preuves}
On renvoie pour la preuve des ŽnoncŽs ci-dessus ˆ \cite{LL} dont les arguments
s'Žtendent {\it verbatim} au cas des corps locaux et globaux de caractŽristique quelconque.

\appendix

\Chapter{Rappels d'analyse harmonique}

\Section{Relations d'orthogonalitŽ}

Soit $H$ un groupe localement compact unimodulaire. Soit $\pi$ une reprŽsentation
de la sŽrie discrte \cad une reprŽsentation unitaire irrŽductible dans un espace de Hilbert $V$
dont tous les coefficients 
$$f(x)=\lg\pi(x)v,w\lr$$
sont de carrŽ intŽgrable. On peut montrer que 
$$\int_H \vert f(x)\vert^2dx\le C(\pi)\vert\vert v\vert\vert. \vert\vert w\vert\vert\ptf$$
Si $\sigma$ est une autre reprŽsentation
de la sŽrie discrte dans $V'$ l'intŽgrale
$$A(v,v',w,w')=\int_H\lg\pi(x)v,w\lr \overline{\lg\sigma(x)v',w'\lr}\,dg$$
est convergente. 

\begin{lemma}\label{orth} Il existe une constante $d_\pi$ dŽpendant de $\pi$ et du choix
de la mesure de Haar sur $H$ telle que
$$ A(v,v',w,w')= \begin{cases} d_\pi \lg v,v'\lr \overline{\lg w,w'\lr}&\com{Si $\pi=\sigma$}\cr
0&\com{Si $\pi\not\simeq\sigma$}
\end{cases}$$
\end{lemma}
\begin{proof}
 Pour tout $y\in H$
$$A(v,v',w,w')=A(\pi(y)v,\sigma(y)v',w,w')$$ 
et donc il existe un opŽrateur (linŽaire continu) $I:V\to V'$ (dŽpendant de $w$ et $w'$)
tel que
$$ A(v,v',w,w')= \lg Iv,v'\lr=A(\pi(y)v,\sigma(y)v',w,w')= \lg I\pi(y)v,\sigma(y)v'\lr$$
et donc
$$I\pi(y)=\sigma(y)I$$
Si $A$ n'est pas identiquement nul l'opŽrateur $I$ est non nul et entrelace $\pi$ et $\sigma$.
Donc $A$ n'est pas identiquement si et seulement si $\pi$ et $\sigma$ sont Žquivalentes.
\end{proof}
Le mme raisonnement s'applique si $\pi$ a tous ses coefficients ˆ support compact
(ou plus gŽnŽralement intŽgrables) et $\sigma$ unitaire irrŽductible arbitraire.

Supposons dŽsormais que $H=\bs\G(F)$ o $\bs\G$ est un groupe rŽductif 
dŽfini sur un corps. Pour ne pas avoir ˆ prendre en compte le centre 
nous supposerons pour simplifier que $\bs\G$ est semi-simple.
Supposons que $\pi$ a tous ses coefficients ˆ support compact. Ce sera le cas
des reprŽsentations super-cuspidales. 
Soit $\mathcal B_\pi=\{e_i\,|\,{i\in\NM}\}$ une base orthonormale de l'espace de Hilbert $V$
o se rŽalise la reprŽsentation $\pi$. Notons $\chi_\pi$ son caractre
qui est une distribution que l'on
supposera reprŽsentable par une fonction localement intŽgrable.
On appelle coefficient normalisŽ une fonction de la forme
$$f_\pi(x)=\frac{1}{d_\pi}\lg\pi(x)e_0,e_0\lr$$
o $e_0$ est un ŽlŽment de $\mathcal B_\pi$.

\begin{lemma}\label{coef}soient $\pi$ et $\pi'$ 
deux reprŽsentations admissibles de carrŽ intŽgrables.
Un coefficient normalisŽ pour $\pi$  vŽrifie
$$\trace\pi(\overline{f_\pi})
=1\com{et} 
\trace\pi'(\overline{f_\pi})=0 \com{si}\pi'\not\simeq\pi\ptf$$
\end{lemma}
\begin{proof}La trace de l'opŽrateur $$\pi'(f_\pi)=\int_H\pi'(x)\overline{f_\pi(x)}\,dx$$
est donnŽe par la sŽrie
$$\trace\pi'(f_\pi)=\sum_{e\in\mathcal B_{\pi'}}\langle \pi'(f_\pi)e,e\rangle
=\frac{1}{d_\pi}\sum_{e\in\mathcal B_{\pi'}}\int_H\langle \pi'(x)e,e\rangle
\overline{\langle \pi(x)e_0,e_0\rangle}\,dx
$$
et le lemme rŽsulte de \ref{orth}.
\end{proof}

\Section{Caractres et intŽgrales orbitales}


\begin{proposition}\label{cara}
Les reprŽsentations admissibles irrŽductibles des sous-groupes distinguŽs ouverts d'indice fini de $GL(n,F)$,
ainsi que celles de $SL(n,F)$ et de leurs formes intŽrieures,
admettent des caractres distribution qui sont reprŽsentables par
des fonctions localement intŽgrables.
\end{proposition}

\begin{proof}
C'est bien connu si $F$ est archimŽdien ou si $F$ est non archimŽdien de caractŽristique zŽro
pour des groupes rŽductifs arbitraires par les travaux d'Harish-Chandra.
Pour les corps non archimŽdiens, sans restriction sur la caractŽristique,
le cas $GL(2)$ est traitŽ dans \cite{JL}. 
On dispose dŽsormais des rŽsultats de \cite{Lem} pour tout $n$. 
Le cas des groupes rŽductifs arbitraires reste conjectural.\end{proof}

\begin{lemma}\label{orbchi} Si $\pi$ est cuspidale
les intŽgrales orbitales
d'un coefficient normalisŽ $f_\pi$, pour un ŽlŽment $x$ semi-simple rŽgulier, vŽrifient
$$\orb(x,f_\pi)=\begin{cases}\chi_\pi(x)&\com{si $x$ est elliptique}
\cr 0&\com{sinon}\cr
\end{cases}$$
\end{lemma}
\begin{proof} Supposons $x$ elliptique rŽgulier.
Soit $\vf$ une fonction lisse ˆ support compact dans un voisinage 
assez petit de $x$ de sorte que tous les points de ce voisinage sont aussi
elliptiques rŽguliers. On a
$$\int_H \orb(x,f_\pi)\vf(x)\,dx=\frac{1}{d_\pi}\int_H\Big(\int_{I_x\bsl H}
\lg\pi(y\mun x y)e_0,e_0\lr\,dy\Big)\vf(x)\,dx
$$
qui, si on munit $I_x$ de la mesure de masse totale $1$, est encore Žgal ˆ
$$\frac{1}{d_\pi}\int_H
\lg\pi(\vf)\pi(y)e_0,\pi(y)e_0\lr\,dy
=\frac{1}{d_\pi}\sum_{e\in\mathcal B_\pi}\int_H
\lg\pi(\vf)\pi(y)e_0,e\lr
\overline{\lg\pi(y)e_0,e\lr}\,dy
$$
et les relations d'orthogonalitŽ \ref{orth} donnent
$$\int_H \orb(x,f_\pi)\vf(x)\,dx=\sum_{e\in\mathcal B_\pi}
\lg\pi(\vf)e,e\lr=\trace \pi(\vf)$$
et donc
$$\int_H \orb(x,f_\pi)\vf(x)\,dx=\int_H \chi_\pi(x)\vf(x)\,dx\ptf$$
ConsidŽrons maintenant un ŽlŽment $x$ semi-simple non elliptique. 
Nous devons montrer l'annulation de son intŽgrale orbitale. C'est classique;
faisons le lorsque $H=SL(2,F)$ pour un coefficient construit ˆ partir
de vecteurs lisses \cad $K$-finis.
Par hypothse $x=y\mun ay$ avec $a\in A(F)$ o $A$ est le tore diagonal et $a$ rŽgulier.
On a $$H=SL(2,F)=A(F)U(F)K$$ et donc
$$ \orb(x,f_\pi)\,dx=\int_{A(F)\bsl H}\lg\pi(y\mun a y)e_0,e_0\lr\,dy=\int_{K}\int_{U(F)}
\lg\pi(k\mun u\mun a uk)e_0,e_0\lr\,du\,dk$$
soit encore en utilisant que $A(F)$ normalise $U(F)$ et que $a$ est rŽgulier on a
$$ \orb(x,f_\pi)\,dx=d(a)\int_{K}\int_{U(F)}
\lg\pi(k\mun a uk)e_0,e_0\lr\,du\,dk
$$
Maintenant il suffit d'observer que si les $e_i$ sont des vecteurs lisses
$$\int_{U(F)}\lg\pi( u)e_1,e_2\lr\,du =0$$
si $\pi$ est cuspidale. En effet, par dŽfinition de la cuspidalitŽ, le sous-espace
des co-invariants (lisses) sous $U(F)$ est l'espace (des vecteurs lisses) tout entier et on peut donc Žcrire
$e_2$ comme somme finie de vecteurs de la forme $(\pi(u_i)-1)\ve_i$.
\end{proof}

\begin{lemma} \label{weyl}
Soit $\pi$ une reprŽsentation cuspidale et
$f_\pi(x)$ un coefficient normalisŽ.
Soit
$\sigma$ une reprŽsentation unitaire somme finie de reprŽsentations de carrŽ intŽgrables de $\G(F)$. 
La multiplicitŽ $n(\pi,\sigma)$ de $\pi$ dans $\sigma$ est donnŽe par l'entier
 $$n(\pi,\sigma):=\trace\overline\sigma(f_\pi)=\langle \chi_\sigma,\chi_\pi\rangle\ptf$$ 
\end{lemma}
\begin{proof}
Les relations d'orthogonalitŽ de Weyl \ref{coef} montrent que
$$\trace\sigma( \overline{f_\pi})=n(\pi,\sigma)\ptf$$
 La formule d'intŽgration de Weyl fournit l'ŽgalitŽ
$$\trace\sigma(\overline f_\pi)=\int_{\G(F)}\overline{f_\pi(x)}{\chi_\sigma(x)}\,dx
=\sum_\T w_T\mun\int_{\T(F)} \Delta_\T(t)^2{\chi_\sigma(t)}\overline{\orb(f_\pi,t)}\,dt$$
o la somme porte sur les classes de conjugaisons de tores dans $\G(F)$ 
et $w_T$ est l'ordre du quotient $N_T(F)/T(F)$ o $N_T$ est le normalisateur de $T$.
Les tores compacts sont munis de la mesure de Haar normalisŽe ($\vol(T(F)=1$).
Puisque $\pi$ est cuspidale, il rŽsulte de  \ref{orbchi} qu'en un point $t\in\T(F)$ rŽgulier
$$\orb(f_\pi,t)=\chi_\pi(t)$$
si $\T(F)$ est un tore elliptique dans $\G(F)$
et l'intŽgrale orbitale est nulle pour un tore dŽployŽ. On en dŽduit que:
$$\trace\sigma(\overline f_\pi)
=\sum_\T w_T\mun\int_{\T(F)} \Delta_\T(t)^2{\chi_\sigma(t)}\overline{\chi_\pi(t)}\,dt
=\langle \chi_\sigma,\chi_\pi\rangle$$
\end{proof}

\vfe

\end{document}